\documentclass[10pt, letterpaper,reqno]{amsart}
\usepackage[left=1in,right=1in,bottom=0.5in,top=0.8in]{geometry}
\usepackage{amsfonts}
\usepackage{amsmath, amssymb}
\usepackage{graphicx}
\usepackage[font=small,labelfont=bf]{caption}
\usepackage{epstopdf}
\usepackage{xcolor}
\usepackage{amsthm}
\usepackage{float}
\usepackage{enumitem}
\usepackage{pgfplots}
\usepackage{listings}
\usepackage{longtable}
\usepackage{mathrsfs}
\usepackage{dsfont}
\usepackage{mathtools}
\usepackage{dsfont}
\usepackage{stmaryrd}
\usepackage[pdfpagelabels,hyperindex]{hyperref}
\hypersetup{linkbordercolor=green}

\newcommand*{\rom}[1]{\expandafter\@slowromancap\romannumeral #1@}
 \colorlet{lgray}{white!80!black}
\colorlet{lred}{white!85!red}
\colorlet{lgreen}{white!60!green}
\colorlet{dgreen}{black!30!green}
\colorlet{lpurple}{white!60!purple}
\colorlet{lblue}{white!60!blue}
\definecolor{green}{rgb}{0.1,0.8,0.1}
\definecolor{yellow}{rgb}{1.0,0.85,0.25}
\definecolor{purple}{rgb}{1.0, 0, 1.0}
\definecolor{blue}{rgb}{0, 0, 1.0}

\tikzstyle{unfused}=[lgray, line width=1.5pt, ->]
\tikzstyle{fused}=[lgray, line width=4pt, ->]
\tikzstyle{dual}=[black, line width=1pt, dashed]
\tikzstyle{lightdual}=[black, line width=0.5pt, dashed]
\tikzstyle{cut}=[black, line width=1.0pt]

\theoremstyle{plain}
\newtheorem{theorem}{Theorem}[section]
\newtheorem{lemma}[theorem]{Lemma}
\newtheorem{proposition}[theorem]{Proposition}
\newtheorem{corollary}[theorem]{Corollary}
\newtheorem{conjecture}[theorem]{Conjecture}

\theoremstyle{remark}
\newtheorem{definition}[theorem]{Definition}
\newtheorem{remark}[theorem]{Remark}

\colorlet{shadecolor}{gray!20}
\pgfplotsset{compat=1.9}
\usetikzlibrary{shapes.multipart}
\usetikzlibrary{patterns}
\usetikzlibrary{shapes.multipart}
\usetikzlibrary{arrows}
\usetikzlibrary{decorations.markings}
\usepgflibrary{decorations.shapes}
\usetikzlibrary{decorations.shapes}
\usepgflibrary{shapes.symbols}
\usetikzlibrary{shapes.symbols}
\usetikzlibrary{decorations.pathreplacing}
\tikzstyle{fleche}=[>=stealth', postaction={decorate}, thick]
\tikzstyle{axis}=[->, >=stealth', thick, gray]
\tikzstyle{paths}=[>->, >=stealth', thick]
\tikzstyle{path}=[->, >=stealth', thick]
\tikzstyle{grille}=[dotted, gray]
\DeclareFontFamily{U}{mathx}{}
\DeclareFontShape{U}{mathx}{m}{n}{<-> mathx10}{}
\DeclareSymbolFont{mathx}{U}{mathx}{m}{n}
\DeclareMathAccent{\widehat}{0}{mathx}{"70}
\DeclareMathAccent{\widecheck}{0}{mathx}{"71}

\newcommand{\arcleftthicklg}{\raisebox{-17pt}{\begin{tikzpicture}
        \draw[thick] (0,0.5) arc (90:270:0.5);
        \node[right] at (-.15,0.65) {\small $x$};
        \node[right] at (-.15,-0.65) {\small $y$};
        \node[left] at (-0.5,0) {\textcolor{red}{\small  $u$}};
\end{tikzpicture}}}
\newcommand{\arcrightthicklg}{\raisebox{-17pt}{\begin{tikzpicture}
        \draw[thick] (0,-0.5) arc (-90:90:0.5);
        \node[left] at (.15,0.65) {\small $x$};
        \node[left] at (.15,-0.65) {\small $y$};
        \node[right] at (0.5,0) {\textcolor{red}{\small  $v$}};
\end{tikzpicture}}}
\newcommand{\bulkrightlg}{\raisebox{-15pt}{\begin{tikzpicture}
        \draw[thick] (0,0)--(0.6,0.6);
        \node[below left] at (.15,0.05) {\small $y$};
        \node[above right] at (0.45,0.55) {\small $x$};
        \node[above] at (0.3,0.3) {\textcolor{red}{\small $\a$}};
\end{tikzpicture}}}
\newcommand{\bulkleftlg}{\raisebox{-15pt}{\begin{tikzpicture}
        \draw[thick] (0.6,0)--(0,0.6);
        \node[below right] at (0.45,.05) {\small $y$};
        \node[above left] at (0.15,0.55) {\small $x$};
        \node[above] at (0.3,0.3) {\textcolor{red}{\small  $\a$}};
\end{tikzpicture}}}

\newcommand{\arcleftthick}{\raisebox{-17pt}{\begin{tikzpicture}
        \draw[thick] (0,0.5) arc (90:270:0.5);
        \node[right] at (-.15,0.65) {\small $x$};
        \node[right] at (-.15,-0.65) {\small $y$};
        \node[left] at (-0.5,0) {\textcolor{red}{\small  $c_1$}};
\end{tikzpicture}}}
\newcommand{\arcrightthick}{\raisebox{-17pt}{\begin{tikzpicture}
        \draw[thick] (0,-0.5) arc (-90:90:0.5);
        \node[left] at (.15,0.65) {\small $x$};
        \node[left] at (.15,-0.65) {\small $y$};
        \node[right] at (0.5,0) {\textcolor{red}{\small  $c_2$}};
\end{tikzpicture}}}
\newcommand{\bulkright}{\raisebox{-15pt}{\begin{tikzpicture}
        \draw[thick] (0,0)--(0.6,0.6);
        \node[below left] at (.15,0.05) {\small $y$};
        \node[above right] at (0.45,0.55) {\small $x$};
        \node[above] at (0.3,0.3) {\textcolor{red}{\small $a$}};
\end{tikzpicture}}}
\newcommand{\bulkleft}{\raisebox{-15pt}{\begin{tikzpicture}
        \draw[thick] (0.6,0)--(0,0.6);
        \node[below right] at (0.45,.05) {\small $y$};
        \node[above left] at (0.15,0.55) {\small $x$};
        \node[above] at (0.3,0.3) {\textcolor{red}{\small $a$}};
\end{tikzpicture}}}

\newcommand{\bulkrightdotted}{\raisebox{-15pt}{\begin{tikzpicture}
        \draw[dashed] (0,0)--(0.6,0.6);
        \node[below left] at (.15,.05) {\small $y$};
        \node[above right] at (0.4,0.50) {\small $x$};
\end{tikzpicture}}}
\newcommand{\bulkleftdotted}{\raisebox{-15pt}{\begin{tikzpicture}
        \draw[dashed] (0.6,0)--(0,0.6);
        \node[below right] at (0.45,0.05) {\small $y$};
        \node[above left] at (0.2,0.5) {\small $x$};
\end{tikzpicture}}}

\DeclareFontFamily{OT1}{pzc}{}
\DeclareFontShape{OT1}{pzc}{m}{it}{ <-> s*[1.2] pzcmi7t }{}
\DeclareMathAlphabet{\mathpzc}{OT1}{pzc}{m}{it}

\DeclareFontFamily{U}{mathx}{}
\DeclareFontShape{U}{mathx}{m}{n}{<-> mathx10}{}
\DeclareSymbolFont{mathx}{U}{mathx}{m}{n}
\DeclareMathAccent{\widehat}{0}{mathx}{"70}
\DeclareMathAccent{\widecheck}{0}{mathx}{"71}
\DeclareMathOperator{\wt}{wt}

\DeclareMathOperator{\stat}{}%I removed 'stat,' from here since I don't think we need to say stat in the notation.
\DeclareMathOperator{\lgrw}{LGRW}
\DeclareMathOperator{\grw}{GeoRW}
\DeclareMathOperator{\Geom}{Geo}
\DeclareMathOperator{\Ga}{Gamma}
\DeclareMathOperator{\BM}{BM}
\DeclareMathOperator{\KPZstat}{KPZ}
\DeclareMathOperator{\lpp}{Geo}
\DeclareMathOperator{\lgg}{LG}

\def \be{\begin{equation}}
\def \ee{\end{equation}}
\def \lb{\left(}
\def \rb{\right)}
\def \lbb{\left\{}
\def \rbb{\right\}}
\def \lbe{\lb}
\def \rbe{\rb}
\def \lbE {\left[}
\def \rbE {\right]}
\def \hZ {\mathcal{Z}}
\def \one {\mathds{1}}
\def \bl {\boldsymbol\lambda}
\def \bL {\mathbf{L}}

\def \bX {\boldsymbol X}
\def \bb {\boldsymbol b}
\def \aa {\boldsymbol a}
\def \bbb {\boldsymbol \beta}
\def \aaa{\boldsymbol \alpha}
\def \U {\mathcal{U}}
\def \pU {\mathbf{U}}
\def \cU {\mathscr{U}}

\def \Ub{\U_{\lpp}^{
\begin{tikzpicture}[scale=0.2] 
\draw[dotted] (0,0) -- (1,0)--(1,1)--(0,1)--(0,0); 		
\draw[thick] (0,1) -- (0,0) -- (1,0); 
\end{tikzpicture}}}
\def \Ul{\U_{\lpp}^{
\begin{tikzpicture}[scale=0.2]
		\draw[dotted] (0,0) -- (0,1)--(-1,0)--(0,0);
            \draw[thick] (0,0) -- (-1,0);
		\end{tikzpicture}}}
\def \Ur{\U_{\lpp}^{
\begin{tikzpicture}[scale=0.2]
		\draw[dotted] (0,0) -- (1,0)--(0,-1)--(0,0);
            \draw[ thick] (0,0) -- (0,-1);
		\end{tikzpicture}}}

\def \rmUb{\mathrm{U}_{\lpp}^{
\begin{tikzpicture}[scale=0.2] 
\draw[dotted] (0,0) -- (1,0)--(1,1)--(0,1)--(0,0); 		
\draw[thick] (0,1) -- (0,0) -- (1,0); 
\end{tikzpicture}}}
\def \rmUl{\mathrm{U}_{\lpp}^{
\begin{tikzpicture}[scale=0.2]
		\draw[dotted] (0,0) -- (0,1)--(-1,0)--(0,0);
            \draw[thick] (0,0) -- (-1,0);
		\end{tikzpicture}}}

\def \rmUr{\mathrm{U}_{\lpp}^{
\begin{tikzpicture}[scale=0.2]
		\draw[dotted] (0,0) -- (1,0)--(0,-1)--(0,0);
            \draw[ thick] (0,0) -- (0,-1);
		\end{tikzpicture}}}

\def \rmUblg{\mathrm{U}_{\lgg}^{
\begin{tikzpicture}[scale=0.2] 
\draw[dotted] (0,0) -- (1,0)--(1,1)--(0,1)--(0,0); 		
\draw[thick] (0,1) -- (0,0) -- (1,0); 
\end{tikzpicture}}}
\def \rmUllg{\mathrm{U}_{\lgg}^{
\begin{tikzpicture}[scale=0.2]
		\draw[dotted] (0,0) -- (0,1)--(-1,0)--(0,0);
            \draw[thick] (0,0) -- (-1,0);
		\end{tikzpicture}}}

\def \rmUrlg{\mathrm{U}_{\lgg}^{
\begin{tikzpicture}[scale=0.2]
		\draw[dotted] (0,0) -- (1,0)--(0,-1)--(0,0);
            \draw[ thick] (0,0) -- (0,-1);
		\end{tikzpicture}}}

\def\rellcorner{\begin{tikzpicture}[scale=0.2] 
\draw[dotted] (0,0) -- (1,0)--(1,1)--(0,1)--(0,0); 		
\draw[thick] (0,1) -- (0,0) -- (1,0); 
\end{tikzpicture}}
\def\reacuteangle{\begin{tikzpicture}[scale=0.2]
		\draw[dotted] (0,0) -- (0,1)--(-1,0)--(0,0);
            \draw[thick] (0,0) -- (-1,0);
		\end{tikzpicture}}
\def\rerotateangle{
\begin{tikzpicture}[scale=0.2]
		\draw[dotted] (0,0) -- (1,0)--(0,-1)--(0,0);
            \draw[ thick] (0,0) -- (0,-1);
		\end{tikzpicture}
}

\def \cUb{\cU^{\begin{tikzpicture}[scale=0.2] 
\draw[dotted] (0,0) -- (1,0)--(1,1)--(0,1)--(0,0); 		
\draw[thick] (0,1) -- (0,0) -- (1,0); 
\end{tikzpicture}}}
\def \cUl{\cU^{\begin{tikzpicture}[scale=0.2]
		\draw[dotted] (0,0) -- (0,1)--(-1,0)--(0,0);
            \draw[thick] (0,0) -- (-1,0);
		\end{tikzpicture}}}
\def \cUr{\cU^{\begin{tikzpicture}[scale=0.2]
		\draw[dotted] (0,0) -- (1,0)--(0,-1)--(0,0);
            \draw[ thick] (0,0) -- (0,-1);
		\end{tikzpicture}}}
\def \Ubg{\U_{\lgg}^{\begin{tikzpicture}[scale=0.2] 
\draw[dotted] (0,0) -- (1,0)--(1,1)--(0,1)--(0,0); 		
\draw[thick] (0,1) -- (0,0) -- (1,0); 
\end{tikzpicture}}}
\def \Ulg{\U_{\lgg}^{\begin{tikzpicture}[scale=0.2]
		\draw[dotted] (0,0) -- (0,1)--(-1,0)--(0,0);
            \draw[thick] (0,0) -- (-1,0);
		\end{tikzpicture}}}
\def \Urg{\U_{\lgg}^{\begin{tikzpicture}[scale=0.2]
		\draw[dotted] (0,0) -- (1,0)--(0,-1)--(0,0);
            \draw[ thick] (0,0) -- (0,-1);
		\end{tikzpicture}}}

\def \a {\alpha}
\def \b {\beta}
\def \la{\lambda}
\def \k {\kappa}
\def \d {\mathsf{d}}

\def \e {\mathsf{e}}
\def \p{\mathbf{p}}
\def \q{\mathbf{q}}
\def \ep {\varepsilon}
\def \o {\omega}
\def \vo {\varpi}
\def \hP{\mathcal{P}}
\def \gp {\mathcal{GP}}
\def \htP{\widetilde{\hP}}
\def \X {\mathfrak{X}}
\def \x {\boldsymbol x}
\def \PP {\mathbb{P}}
\def \cP{\mathscr{P}}
\def \RR {\mathbb{R}}
\def \ZZ{\mathbb{Z}}
\def \CC{\mathbb{C}}
\def \Gav{\Ga^{-1}}
\def \PiP{\mathcal{GP}}
\def \hQ{\mathcal{Q}}
\def \gq {\mathcal{GQ}}
\def \ru{\Re u}
\def \wth {\widehat{\wt}} 
\def \S {\mathbb{S}}

\newcommand{\nn}{\varepsilon}
\newcommand{\eps}{\varepsilon}
\newcommand{\R}{\mathbb{R}}
\newcommand{\Sign}{\mathsf{Sign}}

\usepgflibrary{fpu}
\makeatletter
\let\NAT@parse\undefined
\makeatother
 
\title{Stationary measures for integrable polymers on a strip}
\author[G. Barraquand]{Guillaume Barraquand}
\address{G. Barraquand,
	Laboratoire de Physique de l'\'Ecole Normale Supérieure, 
	ENS, Université PSL, CNRS, Sorbonne Université, Université Paris Cité, F-75005 Paris, France}
\email{barraquand@math.cnrs.fr}
\author[I. Corwin]{Ivan Corwin}
\address{I. Corwin, Columbia University,
	Department of Mathematics,
	2990 Broadway,
	New York, NY 10027, USA.}
\email{ivan.corwin@gmail.com}
\author[Z. Yang]{Zongrui Yang}
\address{Z. Yang, Columbia University,
	Department of Mathematics,
	2990 Broadway,
	New York, NY 10027, USA.}
\email{zy2417@columbia.edu}
\begin{document}
\begin{abstract}
We prove that the stationary measures for the free-energy increment process for the geometric last passage percolation (LPP) and log-gamma polymer model on a diagonal strip is given by a marginal of a two-layer Gibbs measure with a simple and explicit description. This result is shown subject to certain restrictions on the parameters controlling the weights on the boundary of the strip. However, from this description and an analytic continuation argument we are able to access the stationary measure for all boundary parameters. Taking an intermediate disorder limit of the log-gamma polymer stationary measure in a strip we readily recover (modulo convergence of the polymer to the open KPZ equation, Conjecture \ref{conj:convergence}) the conjectural description from \cite{barraquand2021steady} of the open KPZ stationary measure for all choices of boundary parameters $u,v\in \RR$ (thus going beyond the restriction $u+v\geq 0$ from \cite{corwin2021stationary}). 
\end{abstract}
\maketitle

\section{Introduction and main results}
\subsection{Preface}
This paper brings structures that have been valuable in studying full or half-space integrable probabilistic models to bear on time-homogeneous models on an interval with two-sided boundary conditions. Our aim is to provide exact and concise descriptions of the stationary measure for these models. The structures are Gibbs measures (or line ensembles) related to the branching rule for symmetric polynomials (e.g. Schur, $q$-Whittaker, Hall-Littlewood, and spin variants). In the full and half-space context these families of measures (special cases or variants of Macdonald processes) are preserved under Markovian dynamics that include, as marginals, various integrable probabilistic models. Thus the study of these measures, their marginals, and their asymptotics contains rich information about the original models. 

Until now it was unclear how to define versions of these measures indexed by the interval with two-sided boundary conditions. The accomplishment of this paper is the understanding of how to properly define these, now infinite mass, measures (our two-layer Gibbs measures) and the realization that the preservation of Markovian dynamics (which follow from summation identities reminiscent of the Cauchy and Littlewood type identities for symmetric functions) translates into the fact that these measures contain as marginals the stationary probability measures for integrable probabilistic models on an interval. Our approach is structural -- the stationary measures come from variants of known identities in a way that should be generalizable to other Yang-Baxter solvable models in integrable probability. Our work here only focuses on the construction of stationary measures. It is a compelling direction to probe temporal correlations and fluctuations (as has been a focus of attention in the full and half-space contexts for many years).

We demonstrate our approach for two well-studied models, the geometric last passage percolation and log-gamma polymer models. In particular, we show that stationary measures for the increment processes of the last passage time in geometric LPP and the free energy in the log-gamma polymer can be realized as a marginal of certain exponential reweightings of a pair of geometric or log-gamma increment random walks with free starting point (hence the infinite mass of the measure). These are the first descriptions of the stationary measures for these two models in a strip.

Besides serving as a proof of concept for our approach, we chose these models for two other reasons. The first is that the method of matrix product ansatz (MPA) does not easily apply here. That approach \cite{derrida1993exact} involves writing the stationary measure in terms of a product of matrices, one for each particle occupation number (e.g. for open ASEP where there are particles and holes, there are two matrices, $D$ and $E$), which satisfy a certain quadratic algebra with boundary vectors. This approach has been extensively studied for the last 30 years and notably related in the case of open ASEP to Askey-Wilson polynomials \cite{Uchiyama_2004} (see also \cite{CW11})  and processes \cite{Bryc2017} where it has enabled an understanding of the precise phase diagram, large deviations and fluctuations of the open ASEP stationary measure. 

Though there has been work involving multiple particles per site or species of particles, to our knowledge the MPA has never been developed to deal with infinite (countable or uncountable) occupation numbers per site. In particular, for the geometric LPP and log-gamma polymer, the increments of the last passage time or free energy play the role of occupation variables and are precisely of this countable and uncountable type, respectively. We do show here that it is possible to formulate a MPA for these models, though the quadratic algebra relations now involve infinitely many matrices (in the countable case) or operators (in the uncountable case) satisfying relations with infinitely many other of the matrices or operators. Finding representations for these algebras seems quite challenging. However, we show here that the relevant marginal of our two-layer Gibbs measures that describes the stationary measure for these models can be rewritten in matrix product form and verify that our two-layer Gibbs measures satisfy the relevant quadratic algebras. To our knowledge this is the first instance where techniques from integrable probability (i.e., measures written in terms of symmetric functions) have come into direct contact with the method of MPA. We anticipate (and plan to develop in subsequent work) that our two-layer Gibbs approach will also be applicable in many other cases, including those where the MPA has been previously applied, e.g. for open ASEP, and will provide a new and direct route to describe stationary measures and perform asymptotics.

The second reason for our choice of models here is that they permit analytic continuations that allow us to move beyond certain initial limitations on boundary parameters so as to access asymptotics for all choices of boundary parameters. In particular, doing this for the log-gamma polymer and applying an intermediate disorder scaling limit allows us to access the conjectural description from \cite{barraquand2021steady} of the open KPZ equation stationary measure for all choices of Robin boundary condition parameters $u,v\in \RR$  (thus going beyond the restriction $u+v\geq 0$ from \cite{corwin2021stationary}). 
The restrictions in \cite{corwin2021stationary} stem from a similar restriction in the rewriting of the matrix product ansatz in terms of Askey-Wilson processes in \cite{Bryc2017}. The conjecture in \cite{barraquand2021steady} came from the Laplace transform formulas for the open KPZ equation stationary measure derived in \cite{corwin2021stationary} for $u+v\geq 0$ in two steps. First \cite{bryc2021markov,barraquand2021steady} inverted the multipoint Laplace transform formulas to yield a description of the stationary measure in terms of a free-starting point Brownian motion reweighted by a certain exponential functional of its trajectory. This description made sense for $u+v> 0$ but not when $u+v\leq 0$. However, \cite{barraquand2021steady}  then implemented an idea from Liouville quantum mechanics, integrating out the `zero-mode'. The result, remarkably, made sense for all $u,v\in \RR$ and thus led to the conjectured general stationary measure. Justifying the continuation to all $u,v\in \RR$ at the level of the open KPZ equation stationary measure seems quite difficult, e.g. it is unclear how to show that the stationary measure and its conjectural description depend in some manner analytically on the boundary parameters, or that such a dependence results in a unique extension. See \cite{corwin2022some} for more on these works. 

The conditions $u+v> 0$ and $u+v< 0$ define the so-called `fan' and `shock' region in the phase diagram for the open KPZ equation, and likewise arise in the phase diagram for other models like ASEP (along the line $u+v=0$ separating the two phases, stationary measures are Brownian for the KPZ equation, and are generally product measures for discrete models). To our knowledge our work here is the first rigorous derivation of fluctuation limits in the entire phase diagram -- previous work was restricted to the fan region and the product measure line.

The story we develop here puts the reweighted free-starting point Brownian motion stationary measure description from  \cite{bryc2021markov,barraquand2021steady} and the conjectured extension from \cite{barraquand2021steady} into a general context in which the first description follows from the two-layer Gibbs measure structure and preservation property, while the second follows (rigorously) from the uniqueness of analytic continuation of real analytic functions and the analytic dependence of the LPP and the log-gamma stationary measures on boundary parameters.

The stationary measures of geometric LPP or the log-gamma polymer in full space or on a strip with periodic boundary conditions are random walks, i.e. described by a Gibbs measure with only one layer. In full space, this was proved in \cite{balazs2006cube} for exponential LPP and in \cite{seppalainen2012scaling} in the log-gamma case, using elementary identities in distribution. The arguments could be adapted to address periodic boundary conditions. It is remarkable that in presence of more complicated boundary conditions, it suffices to add another layer to the Gibbs measure.

The rest of this introduction is structured as follows. 
Sections \ref{introLPP} and \ref{introloggamma}  respectively introduce the geometric LPP and log-gamma polymer on a strip and Theorems \ref{thm:stationary measure LPP intro} and \ref{thm:stationary measure log-gamma intro}   provide a concise formulas for the stationary measures of these models for all choices of boundary parameters. Section \ref{subsec: stationary measure open KPZ intro} explains how, modulo the convergence (in the spirit of \cite{alberts2014intermediate,wu2018intermediate,parekh2019positive,barraquand2022stationary}) in Conjecture \ref{conj:convergence} of the log-gamma polymer in a strip to the open KPZ equation, we are able to prove in Theorem \ref{thm:KPZ stat} the conjectural open KPZ equation stationary measure formula from \cite{barraquand2021steady} for all boundary parameters. 
Section \ref{subsec:stationary Gibbs line ensembles} explains the Gibbs measure structural mechanism and subsequent analytic continuation argument behind our construction of stationary measures. Section \ref{sec:extensions} discusses potential extension of our approach.

\subsection{Geometric LPP on a strip}\label{introLPP}
We first introduce the geometric last passage percolation (LPP) on a strip and then state the first main theorem of the paper (Theorem \ref{thm:stationary measure LPP intro}) that shows that its stationary measure can be seen as the marginal of a certain reweighting of two independent geometric random walks.

Fix any  $N\in \ZZ_{\geq 1}$ and consider the strip $\left\{(n,m)\in\ZZ^2: 0\leq m\leq n\leq m+N\right\}$ of width $N$ on the integer lattice $\ZZ^2$. Each point $(n,m)$ in the strip is called a vertex of the strip. Vertices $(n,m)$ satisfying $0\leq m<n<m+N$ are called bulk vertices. Vertices $(m,m)$ for $m\geq 0$ are called left boundary vertices and vertices $(m+N,m)$ for $m\geq 0$ are called right boundary vertices. Edges of the lattice $\ZZ^2$ connecting two neighboring vertices of the strip are called edges of the strip.
 
We will use the word `down-right path' to refer to  a path $\hP$ that goes from a left boundary vertex of the strip to a right boundary vertex of the strip, with each step going downwards or rightwards by $1$. We always label vertices and edges of $\hP$ from the up-left start of the path to the down-right end of the path: vertices $\p_0=(n_0,m_0),\dots,\p_N=(n_N,m_N)$ and edges $\e_i$ that connect $\p_{i-1}$ with $\p_i$ for $1\leq i\leq N$.
For $k\in\ZZ_{\geq 0}$, we denote by $\tau_k:\ZZ^2\rightarrow\ZZ^2$ the up-right translation by $(k,k)$, i.e. $\tau_k (x,y) = (x+k,y+k)$. Then the down-right path $\tau_k\hP$ has vertices $\tau_k\p_i$ for $0\leq i\leq N$ and similarly shifted edges.  We denote by $\hP_h$ the horizontal path starting at $(0,0)$, i.e. form by $\p_i=(i,0)$ for $0\leq i\leq N$.

\begin{figure}[h]
    \centering
    \begin{tikzpicture}[scale=1]
\draw[dotted] (0,0)--(3.5,3.5);
\draw[dotted] (6,0)--(9.5,3.5);
\draw[dotted](0,0)--(6,0);
\draw[dotted](1,1)--(7,1);
\draw[dotted] (2,2)--(8,2);
\draw[dotted] (3,3)--(9,3);   
\draw[dotted] (1,0)--(1,1);
\draw[dotted] (2,0)--(2,2);
\draw[dotted](3,0)--(3,3); 
\draw[dotted] (4,0)--(4,3.5);
\draw[dotted] (5,0)--(5,3.5);
\draw[dotted](6,0)--(6,3.5);
\draw[dotted] (7,1)--(7,3.5);
\draw[dotted](8,2)--(8,3.5); 
\draw[dotted] (9,3)--(9,3.5);  
\draw[ultra thick] (1,1)--(3,1)--(3,0)--(6,0);   
\draw[thin] (4,0)--(4,2)--(6,2)--(6,3);
\node[left] at (1,1){$\mathcal{P}$};
\node[above left] at (4,2){$\pi$};
\node  at (0,0) {\small $(0,0)$};
\node  at (1.3,1.15) {\small $G_0(0)$};
\node  at (2.3,1.15) {\small $G_0(1)$};
\node  at (3.3,1.15) {\small $G_0(2)$};
\node at (3.4,0.15) {\small $G_0(3)$};
\node at (4.35,0.15) {\small $G_0(4)$};
\node at (5.35,0.15) {\small $G_0(5)$};
\node at (6.35,0.15) {\small $G_0(6)$};
\node at (4.3,1.1) {\small $\o_{4,1}$};
\node at (4.3,1.9) {\small $\o_{4,2}$};
\node at (5.3,1.9) {\small $\o_{5,2}$};
\node at (6.3,1.9) {\small $\o_{6,2}$};
\node at (7,3.1) {\small $\o_{6,3}\sim \Geom(a^2)$};
\node at (3.8,3.1) {\small $\o_{3,3}\sim \Geom(ac_1)$};
\node at (10,2.9) {\small $\o_{9,3}\sim \Geom(ac_2)$};
\end{tikzpicture}
    \caption{The width 6 strip with a down-right (thick) path $\hP$ and an up-right (thin) path $\pi$ (that intersects $\hP$ only once at $\pi(0)=(4,0)$ and ends at $(6,3)$). The weights $\o$ are shown along the path $\pi$ as well as two boundary weights (with the distributions of the boundary weights and one bulk weight indicated). The initial condition along $\hP$ is given by the $G_0$ random variables labeled there.}
    \label{fig:sample LPP}
\end{figure}
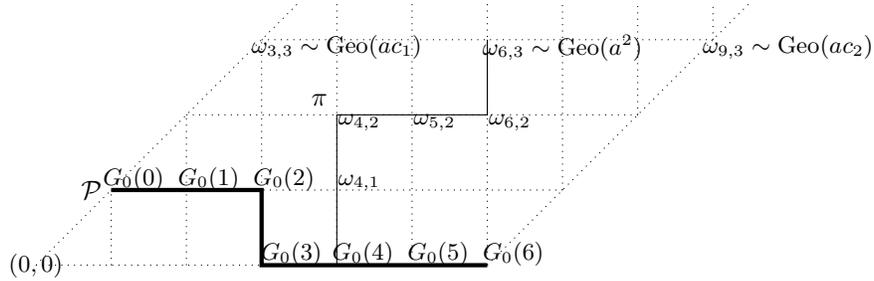 

\begin{definition}[Geometric LPP on a strip] \label{defn:homogeneous geometric LPP}  
A random variable $X\sim\Geom(a)$ has geometric distribution with parameter $a\in(0,1)$ if $\PP(X=k)=(1-a)a^k$ for $k\in\ZZ_{\geq 0}$.
Let $a\in (0,1)$ be a `bulk parameter' and let $c_1,c_2>0$ be `boundary parameters' such that $ac_1<1$ and $ac_2<1$.
Let $(\o_{n,m})_{0\leq m\leq n\leq m+N}$ be a sequence of independent geometric random variables indexed by the vertices of the strip. On a bulk vertex $(n,m)$ we assume that $\o_{n,m}\sim\Geom(a^2)$. On a left boundary vertex $(m,m)$ we assume that $\o_{m,m}\sim\Geom(ac_1)$ and on a right boundary vertex $(m+N,m)$ we assume that $\o_{m+N,m}\sim\Geom(ac_2)$. 
    An `initial condition' is given by a down-right path $\hP$ and a process $\big(G_0(j)\big)_{0\leq j\leq N}$, independent of the geometric variables $(\o_{n,m})_{0\leq m\leq n\leq m+N}$, where below we treat the random variable $G_0(j)$ as the initial condition for the LPP recurrence relation at $j$-th vertex $\p_j=(n_j,m_j)$ of the path for $0\leq j\leq N$. Given such initial condition, define the associated last passage time for all points $(n,m)$ above the path $\hP$ (i.e. all points $(n,m)$ that can be reached from a point $\p_j\in \hP$ via an up-right path of lattice edges in the strip) as
\be\label{eqn:gvar}
G(n,m):=\max_\pi\lb G(\pi(0))+\sum_{i=1}^\ell\o_{\pi(i)}\rb.
\ee
     The maximum is taken over all the up-right lattice paths $\pi$ in the strip that starts from a vertex $\pi(0)\in\hP$ and ends at $\pi(\ell)=(n,m)$ for some length $\ell\geq 0$, and which only touches $\hP$ at  point $\pi(0)$. We have also used the convention that  $G(\p_j)=G_0(j)$ for $0\leq j\leq N$. See Figure \ref{fig:sample LPP} for an illustration.
    The above defined last passage times satisfy the recurrence relation 
    that for all $(n,m)$ above the path $\hP$,
\be\label{eq:recurrence geometric LPP}
G(n,m) = \o_{n,m}+ \begin{dcases*}
\max(G(n-1,m),G(n,m-1)) & if $0<m<n<m+N$,\\
G(n,m-1) & if $0<m=n$,\\
G(n-1,m) & if $0<m<n=m+N$,
\end{dcases*}
\ee
with initial condition that if $(n,m)= \p_j$ for any  $0\leq j\leq N$, then $G(n,m)=G_0(j)$.

Observe from the recurrence \eqref{eq:recurrence geometric LPP} that the expectation of $G(n,n)$ for $n=1,2,\dots$ strictly increases, hence no solution of \eqref{eq:recurrence geometric LPP} is invariant in law under translations $\tau_k$. We instead look at the increments of $G$ along edges.
The law of a process $G_0=\big(G_0(j)\big)_{0\leq j\leq N}$ with $G_0(0)= 0$ is said to be stationary for the geometric LPP on a down-right path $\hP$ if the solution of \eqref{eq:recurrence geometric LPP} with initial condition specified along $\hP$ by $G(\p_j)=G_0(j)$ for $0\leq j\leq N$ has the property that 
the joint distribution of $G_k:=\big(G(\tau_k\p_j)-G(\tau_k\p_0)\big)_{0\leq j\leq N}$ is the same for all $k\in\ZZ_{\geq 0}$, and hence coincides with that of $G_0$. Since we are only concerned with differences, we have normalized $G_0$ to have $G_0(0)\equiv 0$. The process $(G_k)_{k\geq 0}$ is Markov and we say it is ergodic if there exists is a unique stationary measure that is the limit law as $k\to\infty$ for all choices of initial condition $G_0$.
\end{definition}

\begin{definition}[Reweighted geometric random walks]\label{defn:rescaled random walks LPP}
For $a\in(0,1)$, consider two independent random walks $\bL_1=\big(L_1(j)\big)_{1\leq j\leq N}\in \ZZ_{\geq 0}^{N}$ and $\bL_2=\big(L_2(j)\big)_{1\leq j\leq N}\in \ZZ_{\geq 0}^{N}$ starting from  $L_1(0)=L_2(0)=0$ with i.i.d increments distributed for $1\leq j\leq N$ as 
 $$L_1(j)-L_1(j-1) \sim \Geom(a_1),\quad \textrm{and}\quad L_2(j)-L_2(j-1)\sim \Geom(a_2),$$
 and let us denote by $\PP^{a_1,a_2}_{\grw}$ and $\mathbb{E}^{a_1,a_2}_{\grw}$ the associated probability measure and expectation. For $c_1,c_2$ such that $a,ac_1,ac_2\in(0,1)$, we define a new probability measure $\PP_{\stat\lpp}^{a,c_1,c_2}$ by reweighting the measure $\PP^{a,a}_{\grw}$ as 
 \be \label{eq:PstatLPP}
\PP^{a,c_1,c_2}_{\stat\lpp}(\bL_1,\bL_2):= \frac{1}{\hZ^{a,c_1,c_2}_{\lpp}}\,\,(c_1c_2)^{\max_{1\leq j\leq N}\big(L_2(j)-L_1(j-1)\big)}c_2^{L_1(N)-L_2(N)}\PP^{a,a}_{\grw}(\bL_1,\bL_2),
\ee
where $\hZ^{a,c_1,c_2}_{\lpp}$ is a normalizing constant which will be proved to be finite.
\end{definition}

The following is the first main result in this paper. 

\begin{theorem}\label{thm:stationary measure LPP intro}
Assume $a,ac_1,ac_2\in(0,1)$. The law of $\bL_1$ under  $\PP_{\stat\lpp}^{a,c_1,c_2}$ is the (unique) ergodic stationary measure on the horizontal path $\hP_h$ for geometric LPP recurrence relation on a strip (see Definition \ref{defn:homogeneous geometric LPP}).
\end{theorem}

This theorem is a particular case of Theorem \ref{thm:main theorem LPP}, which will be proved in Section \ref{sec:LPP}. When $c_1c_2=1$, it is clear from \eqref{eq:PstatLPP} that the law of $\bL_1$ under $\PP_{\stat\lpp}^{a,c_1,c_2}$ is a $\Geom(ac_2)$ distributed increment random walk starting from $0$.  In this case there are no spatial correlations on the increments and the stationarity measure could have been guessed and verified in a simple way using elementary properties of geometric random variables by adapting the local arguments, e.g. as in \cite[Lemma 2.1]{betea2020stationary}, \cite[Lemma 4.1]{balazs2006cube}.

The restriction to considering  the stationary measure on a horizontal path is unnecessary (though simplifies the description of the measure). Our techniques enable us to determine the stationary measure on any down-right path, see Theorem \ref{thm:stationary measure LPP before sum over zero mode} below which provides a description of the stationary measure for an arbitrary path $\mathcal P$. The stationary measure allows us to define a stationary solution to the recurrence \eqref{eq:recurrence geometric LPP}, but now for all $(n,m)$ in the bi-infinite strip. Specifically, let $G$ be the solution to the recurrence with stationary initial condition from Theorem \ref{thm:stationary measure LPP intro} along the horizontal path from $(0,0)$ to $(N,0)$ and for any $k\in \ZZ_{\geq 0}$ and $(n,m)\in \S^{(k)}_N:=\left\{(n,m)\in\ZZ^2: -k\leq m\leq n\leq m+N\right\}$
define $G^{(k)}(n,m) := G(n+k,m+k)-G(k,k)$. By stationarity, for any $k'>k$, the law of $G^{(k')}$ restricted to $\S^{(k)}_N$ is the same of that of $G^{(k')}$. Thus, we can take a limit $k\to\infty$ to define a bi-infinite stationary solution $G^{(\infty)}$. The increments of this solution along any down-right path coincide in law with the stationary measure on that down-right path from Theorem \ref{thm:stationary measure LPP before sum over zero mode}.

Exponential LPP is a simple scaling limit of geometric LPP and thus its stationary measure is also a reweighting of random walks, now with exponentially distributed jumps, of a very similar nature to \eqref{eq:PstatLPP}. (See also  \cite[Section 3]{barraquand2022stationary}, where a similar limiting procedure is discussed in  the case of LPP in a half-quadrant). Exponential LPP can also be obtained as a (different) scaling limit of the log-gamma polymer model on a strip, which we discuss in the next section. 

There is a well-known mapping \cite{rost1981non} between the exponential LPP and the totally asymmetric simple exclusion process (TASEP) in the quadrant where in the level-lines of LPP record the evolution of the TASEP height function. This mapping survives when the geometry is that of the strip, and the boundary parameters now correspond to the rates at which particles are inserted into TASEP on the left boundary and removed on the right. The stationary measure of this version of TASEP (i.e., open TASEP) can be found using the matrix product ansatz \cite{derrida1993exact}. The open TASEP stationary measure is also implicitly defined from the bi-infinite (as described above) stationary solution to the exponential LPP recurrence. 
%Its height function has the law of the level-line $\big((n,m):G^{(\infty)}(n,m)\leq 0\big)$. 
It is unclear how to obtain an explicit description of the open TASEP stationary measure from this implicit relationship. However, as explained in Section \ref{sec:extensions}, our approach developed here should also be directly applicable to construct the stationary measure for open (T)ASEP, without using such an implicit route.

\subsection{Log-gamma polymer on a strip}\label{introloggamma}
The log-gamma polymer is an exactly solvable directed polymer model on the quadrant $\mathbb{Z}_{>0}^2$, with inverse-gamma distributed weights introduced in \cite{seppalainen2012scaling}. 
Here we introduce a variant, the log-gamma polymer model on a strip and then state the second main theorem of the paper giving its stationary measure in terms of a reweighting of two independent log-gamma random walks.

\begin{definition}[Log-gamma polymer on a strip] \label{defn:homogeneous log gamma polymer}
A random variable $X\sim\Ga^{-1}(\theta)$ has inverse-gamma distribution with parameter $\theta>0$  if it is supported on $\RR_{>0}$ with density 
    $\frac{1}{\Gamma(\theta)}x^{-\theta-1}e^{-1/x}$.
We will write $Y\sim\log\Ga^{-1}(\theta)$ if $e^Y\sim\Ga^{-1}(\theta)$. $Y$ is supported on $\RR$ with density 
\be\label{eq:lgdensity}
\frac{1}{\Gamma(\theta)}\exp(-\theta y-e^{-y}).
\ee
Let $\alpha>0$ be a bulk parameter and let $u,v\in \mathbb R$ be boundary parameters such that $u+\a>0$ and $v+\a>0$.
    Let $(\vo_{n,m})_{0\leq m\leq n\leq m+N}$ be a sequence of independent inverse-gamma random variables indexed by the vertices of the strip. On any bulk vertex $(n,m)$ we assume that $\vo_{n,m}\sim\Gav(2\a)$. On any left boundary vertex $(m,m)$ we assume that $\vo_{m,m}\sim\Gav(u+\a)$ and on any right boundary vertex $(m+N,m)$ we assume that $\vo_{m+N,m}\sim\Gav(v+\a)$. 
    An initial condition is given by a down-right path $\hP$ and a process $\big(z_0(j)\big)_{0\leq j\leq N}$ which is independent of the inverse-gamma variables $(\vo_{n,m})_{0\leq m\leq n\leq m+N}$, where below we treat the random variable $z_0(j)$ as the initial condition for the partition function recurrence relation at the $j$-th vertex $\p_j=(n_j,m_j)$ of the path, for $0\leq j\leq N$. 
Given this, define the associated polymer partition function for all points $(n,m)$ above the path $\hP$  
as
\be\label{eq:defloggammapartitionfunction}
z(n,m):=\sum_\pi\lb z(\pi(0))\prod_{i=1}^\ell\vo_{\pi(i)}\rb.
\ee
     The summation is taken over all the up-right lattice paths $\pi$ in the strip that starts from a vertex $\pi(0)\in\hP$ and ends at $\pi(\ell)=(n,m)$ for some length $\ell\geq 0$, and which only touches $\hP$ at  point $\pi(0)$. We have also used the convention that  $z(\p_j)=z_0(j)$ for $0\leq j\leq N$.
This partition function satisfies the recurrence relation that for all $(n,m)$ above the path $\hP$,
\be\label{eq:recurrence log gamma}
z(n,m) = \vo_{n,m}\times \begin{dcases*}
z(n-1,m)+z(n,m-1) & if $0<m<n<m+N$,\\
z(n,m-1) & if $0<m=n$,\\
z(n-1,m) & if $0<m<n=m+N$,
\end{dcases*}
\ee\
with the initial condition that if $(n,m)=\p_j$ for any $0\leq j\leq N$, then $z(n,m)=z_0(j)$.

As in the LPP case, this recurrence does not have solutions (in law) invariant under translations by $\tau_k$. However, the natural analog of the last passage values here are the free energies $h(n,m):=\log z(n,m)$ along a down-right path. The partition function recurrence relation \eqref{eq:recurrence log gamma} could be rewritten as a recurrence relation for the free energy. The law of a process $h_0=\big(h_0(j)\big)_{0\leq j\leq N}$ with $h_0(0)= 0$ is said to be stationary for the log-gamma polymer on a down-right path $\hP$ if the solution of \eqref{eq:recurrence log gamma} with initial condition specified along $\hP$ by $h(\p_j)=h_0(j)$ for $1\leq j\leq N$ has the property that the joint distribution of $h_k:=\big(h(\tau_k\p_j)-h(\tau_k\p_0)\big)_{0\leq j\leq N}$ is the same for all $k\in\ZZ_{\geq 0}$, and hence coincides with that of $h_0$. As with LPP, $(h_k)_{\geq 0}$ is Markov and the same notion of ergodicity applies. 
\end{definition}
\begin{definition}[Reweighted log-gamma random walks]
For $\theta_1, \theta_2>0$, consider two independent random walks $\bL_1=\big(L_1(j)\big)_{1\leq j\leq N}\in \RR^{N}$ and $\bL_2=\big(L_2(j)\big)_{1\leq j\leq N}\in \RR^{N}$ starting from  $L_1(0)=L_2(0)=0$ with i.i.d increments  distributed for $1\leq j\leq N$ as 
$$L_1(j)-L_1(j-1) \sim \log \Ga^{-1}(\theta_1),\quad \textrm{and}\quad L_2(j)-L_2(j-1) \sim \log \Ga^{-1}(\theta_2),$$
and let us denote by $\PP_{\rm LGRW}^{\theta_1, \theta_2}$ and $\mathbb E_{\rm LGRW}^{\theta_1, \theta_2}$  the associated probability measure and expectation. For $u,v$ such that $\a,u+\a,v+\a>0$, we define a new  probability measure $\PP_{\stat\lgg}^{\alpha, u,v}$
by reweighting the measure  $\PP_{\lgrw}^{\alpha, \alpha}$ as 
\begin{equation} \PP_{\stat\lgg}^{\alpha, u,v}(\bL_1, \bL_2) = \frac{1}{\mathcal Z^{\alpha, u,v}_{\lgg}} \lb\sum_{j=1}^Ne^{L_2(j)-L_1(j-1)}\rb^{-(u+v)}e^{-v(L_1(N)-L_2(N))}\PP_{\lgrw}^{\alpha, \alpha}(\bL_1, \bL_2),
\label{eq:Pstatloggamma}
\end{equation}
where $\hZ^{a,u,v}_{\lgg}$ is a normalizing constant which will be proved to be finite.
\label{def:rescaled random walks}
\end{definition}

The following is the second main result in the paper.
\begin{theorem}\label{thm:stationary measure log-gamma intro}
Assume $\a,u+\a,v+\a\in(0,\infty)$. The law of $\bL_1$ under  $\PP_{\stat\lgg}^{\a,u,v}$ is the (unique) ergodic stationary measure of log-gamma polymer on a strip on the horizontal path $\hP_h$ (see Definition \ref{defn:homogeneous log gamma polymer}).
\end{theorem}

This theorem is a particular case of Theorem \ref{thm:main theorem LG}, which will be proved in Section \ref{sec:LG}. When $u+v=0$, it is clear from \eqref{eq:Pstatloggamma}  that the law of $\bL_1$ under $\PP_{\stat\lgg}^{\a,u,v}$ is a $\log \Ga^{-1}(\alpha+v)$ random walk starting from $0$. In this case, the result could be alternatively proved using properties of inverse gamma random variables, adapting the local arguments in  full-space \cite[Lemma 3.2]{seppalainen2012scaling} and half-space \cite[Proposition 4.5]{barraquand2020halfspace}. As in the LPP case, this stationary measure can be used to define a bi-infinite stationary solution to the polymer recurrence relation. By a `zero-temperature' limit transition, the polymer recurrence relation and stationary measure can be shown to converge to the exponential LPP analogs. Another, considerably more complicated limit transition of the log-gamma polymer is to the open KPZ equation.

\subsection{Stationary measure of the open KPZ equation}
\label{subsec: stationary measure open KPZ intro} 
For $L>0$, the open KPZ equation on $[0,L]$ with Neumann boundary conditions with parameters $u,v\in \mathbb R$ is the stochastic PDE formally written as
\begin{equation}
\begin{cases} 
    \partial_t h(t,x) = \frac{1}{2} \partial_{xx} h(t,x) + \frac{1}{2} \left( \partial_x h(t,x)\right)^2 + \xi(t,x), \\ 
    \partial_x h(t,x)\big\vert_{x=0} = u, \\ 
    \partial_x h(t,x)\big\vert_{x=L} = -v, 
    \end{cases}
    \tag{KPZ$_{u,v}$}
    \label{eq:KPZ}
\end{equation}
where $h$ is a function of  $t\in \mathbb R_+$ and $x\in [0,L]$, and $\xi(t,x)$ is a space-time white noise. Definition \ref{def:Hopf-Cole transform} provides a precise meaning via the Hopf-Cole transform to what it means to solve the open KPZ equation.

Let $C([0,L],\mathbb R)$ denote the space of continuous functions from $[0,L]$ to $\mathbb R$. We will say that the law of a random function $h_0\in C([0,L],\mathbb R)$ with $h_0(0)=0$ is stationary for the open KPZ equation on $[0,L]$ if the solution to \eqref{eq:KPZ} with initial condition $h(0,x)=h_0(x)$ is such the law of $h_t:=\big(h(t,x)-h(t,0)\big)_{x\in [0,L]}$ is the same for all $t\geq 0$, and hence coincides with that of $h_0$. 
In recent works by several authors
\cite{corwin2021stationary,bryc2021markov,bryc2021markov2,barraquand2021steady} (see also the review  \cite{corwin2022some}), the stationary measure for the open KPZ equation was determined when $u+v\geq 0$ and $L=1$, and in the general case, a conjectural formula of the stationary measure appears in \cite{barraquand2021steady} (see below Theorem \ref{thm:KPZ stat} for further discussion and background).

We prove in Proposition \ref{prop:convergencestationarymeasures} that this conjectured open KPZ stationary measure arises as an `intermediate disorder' limit of the log-gamma stationary measure from Theorem \ref{thm:stationary measure log-gamma intro}. 
Intermediate disorder scaling of polymers to the KPZ equation was first introduced in \cite{alberts2014intermediate}, where the convergence was proved for directed polymers in the quadrant $\mathbb Z_{\geq 0}^2$, i.e. for models in `full-space' without boundary (see also \cite{CSZ,corwin2017intermediate} for further developments). The result was later extended in \cite{wu2018intermediate} to discrete directed polymer models in a half-quadrant (`half-space'), whose free energy converge to the KPZ equation on $\mathbb R_{ \geq 0}$ with Neumann boundary condition  at 0 (see also \cite{parekh2019positive, barraquand2022stationary} for generalizations). The methods in those works reduce the convergence to sharp estimates on convergence of random walk to Brownian heat kernels with the specified boundary conditions. Conjecture \ref{conj:convergence} formulates the precise intermediate disorder scaling limit under which the log-gamma polymer in a strip (whose width grows with the scaling) should converge to the open KPZ equation, as a space-time process. 
%We leave the proof of this result to subsequent work. 
Combining this conjecture with the convergence in Proposition \ref{prop:convergencestationarymeasures} of the log-gamma polymer stationary measure under this scale proves the conjectural description from \cite{barraquand2021steady} of the stationary measures of open KPZ equation for general boundary parameters $u,v$ and arbitrary interval length $L$. We state this result as Theorem \ref{thm:KPZ stat} below.

The following process is the intermediate disorder limit (see Proposition \ref{prop:convergencestationarymeasures}) of the reweighted log-gamma random walk $\PP_{\stat\lgg}^{\a,u,v}$ from Definition \ref{def:rescaled random walks} and the conjectured open KPZ stationary measure from \cite{barraquand2021steady}.

\begin{definition}[Hariya-Yor process]\label{def:KPZstat}
For $v_1,v_2\in \RR$ let  $\PP_{\BM}^{v_1,v_2}$  be the probability measure on $C([0,L], \mathbb R)^2$  of  two independent standard Brownian motions $B_1,B_2$ on $[0,L]$ with drifts $v_1$ and $v_2$ respectively.
For $u,v\in \RR$ define the probability measure  $\PP_{\KPZstat}^{u,v}$  on $C([0,L], \mathbb R)$ through its Radon-Nikodym derivative by
\begin{equation}
\label{eq:PstatKPZ}
\frac{\d \PP_{\KPZstat}^{u,v}}{\d\PP_{\BM}^{0,0}} (B_1,B_2) =  \frac{1}{\mathcal Z_{\KPZstat}^{u,v}}\left( \int_0^L \d s e^{-(B_1(s)-B_2(s))}\right)^{-u-v}e^{-v(B_1(L)-B_2(L))}. 
\end{equation}
By the Cameron-Martin theorem, we may also include drifts in the reference measure and define $\PP_{\KPZstat}^{u,v}$ as 
\begin{equation}
\label{eq:PstatKPZbis}
\frac{\d \PP_{\KPZstat}^{u,v}}{\d\PP_{\BM}^{-v,v}} (B_1,B_2) =  \frac{e^{v^2L}}{\mathcal Z_{\KPZstat}^{u,v}}\left( \int_0^L \d s e^{-(B_1(s)-B_2(s))}\right)^{-u-v}. 
\end{equation}
\end{definition}
\begin{theorem}[conjectured in \cite{barraquand2021steady} --  proved here modulo Conjecture \ref{conj:convergence}]
\label{thm:open KPZ stationary measure}
For any $u,v\in \mathbb R$, the law of $B_1$ under $\PP_{\KPZstat}^{u,v}$ is the (unique) ergodic stationary measure for \eqref{eq:KPZ}. 
\label{thm:KPZ stat}
\end{theorem}

The measure $\PP_{\KPZstat}^{u,v}$ was first considered in \cite{hariya2004limiting} in a completely different context, motivated by considerations around the Matsumoto-Yor identity \cite{matsumoto2001relationship, matsumoto2005exponential}. The normalization constant $\mathcal Z_{\KPZstat}^{u,v}$ was computed explicitly for $u,v>0$  in \cite[Proposition 4.1]{donati2000striking} and a formula (the analytic continuation of the formula for $u,v>0$) is given in \cite{bryc2021markov2} for all $u,v$ such that $u+v>0$. The fact that the normalizing constant $\mathcal Z_{\KPZstat}^{u,v}$ is finite for all $u,v\in \mathbb R$ was implicit in \cite{hariya2004limiting}. It is equivalent to the fact that when $B$ is a Brownian motion,  
\begin{equation}
	\mathbb E\left[ \left(\int_0^t e^{-B(s)-\mu s}\mathsf ds\right)^m\right] <\infty
	\label{eq:finitenessZKPZ}
\end{equation}
for any $t>0$ and $\mu, m\in \mathbb R$. When $m>1$ or $m<0$ Jensen's inequality and the convexity of the function $x \mapsto x^m$ imply that 
$$ \mathbb E\left[ \left(\int_0^t e^{-B(s)-\mu s}\mathsf ds\right)^m\right] \leq t^{m-1} \int_0^t \mathbb E\left[e^{-mB(s)-m\mu s} \right] \mathsf ds <\infty.$$
When $m\in (0,1)$, \eqref{eq:finitenessZKPZ}  follows from the case $m=1$, which is easy to check.

The stationary measure for the open KPZ equation was first explicitly described in \cite{corwin2021stationary} via multi-point Laplace transform formulas. As mentioned in the preface, this result was limited to  $u+v>0$ and $L=1$. The $u+v>0$ restriction is serious, coming from limitations of the Askey-Wilson process  \cite{Uchiyama_2004,Bryc2017} method used in \cite{corwin2021stationary}; the $L=1$ restriction can likely be lifted. By inverting the Laplace transform formula, an explicit probabilistic description was given in \cite{bryc2021markov, barraquand2021steady},  showing that the stationary measure (constructed in \cite{corwin2021stationary}) has the law of the process 
$W(x)+ U(x)-U(0)$, where $W$ is a Brownian motion with diffusion coefficient $1/2$ and $W(0)=0$, and $U(x)$ is independent from $W$. The law of $U$ is absolutely continuous with respect to the free Brownian (infinite) measure with Lebesgue measure for $U(0)$ and Brownian law from there with diffusion coefficient $1/2$, with Radon-Nikodym derivative proportional to 
\begin{equation}\exp\left( -2 uU(0) -2 vU(L) -\int_0^L \d s e^{-2 U(s)} \right).
\label{eq:Liouvillereweighting}
\end{equation} 
This is reminiscent of path integrals encountered in Liouville quantum mechanics (see references in \cite{barraquand2021steady}) and is a continuum version of the two-layer Gibbs measure that we introduce below in Section \ref{subsec:stationary Gibbs line ensembles}. 

It was further observed in \cite{barraquand2021steady} (see also Section 5 of \cite{corwin2022some} for a summary of this calculation) that it is fruitful to explicitly average over $U(0)$ in the probability measure above, as is a standard procedure in Liouville quantum mechanics. Letting $X(x)=U(x)-U(0)$ and performing this integration out over $U(0)$, the stationary measure has the law of the process $\big((W(x)+X(x)\big)_{x\in [0,L]}$, where, again, $X$ is independent from $W$, and its law  is absolutely continuous with respect to the Brownian measure with diffusion coefficient $1/2$ and $X(0)=0$, with Radon-Nikodym derivative proportional to 
\begin{equation}  \left( \int_0^L \d s e^{-2 X(s)}\right)^{-u-v} e^{-2v X(L)}.
\label{eq:RNDBarraquandLeDoussal}
\end{equation}
Via the change of variables $W=\frac{B_1+B_2}{2}$, $X=\frac{B_1-B_2}{2}$, it is easy to check that the process $B_1=W+X$ is exactly distributed as  $\PP_{\KPZstat}^{u,v}$ defined above.
This shows that Theorem \ref{thm:KPZ stat} was already proven (combining results in \cite{bryc2021markov, bryc2021markov2, barraquand2021steady}) when $u+v\geq 0$ and $L=1$.

It was observed in \cite{barraquand2021steady} that the law of $B_1=W+X$ defined above makes sense for any $u,v\in \mathbb R$. Thus they conjectured that these measures remain stationary for the open KPZ equation for any $u,v\in \mathbb R$. More precisely, it was argued in \cite{barraquand2021steady} that the result in the $u+v>0$ phase should extend to all $u,v\in \mathbb R$ by analytic continuation. However, it is not at all clear how to make this argument mathematically rigorous. In particular, one would need to define what it means for a distribution on $C([0,L],\RR)$ to be analytic, that the open KPZ stationary measure and the law of $B_1$ both depend in this manner on the boundary parameters,  and that this notion of analyticity implies uniqueness. Our Theorem \ref{thm:KPZ stat} confirms this conjecture by essentially developing such an argument at the level of the pre-limiting log-gamma polymer. As explained above, Theorem \ref{thm:KPZ stat} is a limit of Theorem \ref{thm:stationary measure log-gamma intro}, which is proven from the analytic continuation of the $u+v>0$ phase to all $u,v\in \mathbb R$, but in the case of the log-gamma polymer, we are able to justify rigorously this analytic continuation procedure.

We finally remark that ergodicity of the open KPZ equation and uniqueness of its stationary measure included in Theorem \ref{thm:open KPZ stationary measure} follows directly  from recent work of \cite{knizel2022strong, parekh2022ergodicity}, using ideas from \cite{hairer2018strong}. Thus, the content of the theorem is the confirmation of the conjectured form of the stationary measure for all $u,v\in \RR$ and particular for $u+v<0$. %In particular, \cite{parekh2022ergodicity} applies for all $u,v\in \RR$ and proves the ergodicity.
\subsection{Universality}
The stationary measures of geometric LPP \eqref{eq:PstatLPP}, the log-gamma polymer  \eqref{eq:Pstatloggamma}, and the open KPZ equation \eqref{eq:PstatKPZ} converge to the same large-scale limit,  defined as follows. 
For $\tilde u,\tilde v\in \RR$ we define the probability measure  $\PP_{\infty}^{\tilde u, \tilde v}$  on $C([0,1], \mathbb R)$ through its Radon-Nikodym derivative with respect to $\PP_{\BM}^{-\tilde v,\tilde v}$ by
\begin{equation*}
%label{eq:PstatFP}
\frac{\d \PP_{\infty}^{\tilde u,\tilde v}}{\d\PP_{\BM}^{-\tilde v,\tilde v}} (B_1,B_2) =  \frac{1}{\mathcal Z_{\infty}^{\tilde u,\tilde v}}
e^{(\tilde u+\tilde v) \min_{x\in [0,1]}\lbrace B_1(x)-B_2(x)\rbrace}. 
\end{equation*}

For the KPZ equation, recall the probability measure $\PP_{\KPZstat}^{u,v}$ that implicitly depends on a length $L$ which we set to be $L=\eps^{-1}$. For $i=1,2$ we introduce rescaled processes $B_i^{(\eps)}(x)= \eps^{1/2} B_i(\eps^{-1}x)$ and scale boundary parameters as $u=\eps^{1/2}\tilde u, v=\eps^{1/2}\tilde v$. Then $(B_1^{(\eps)}, B_2^{(\eps)})$ weakly converges as $\eps$ goes to $0$  to  $\PP_{\infty}^{\tilde u, \tilde v}$ \cite{barraquand2021steady}. 

For geometric LPP, consider the measure $\PP^{a,c_1,c_2}_{\stat\lpp}(\bL_1,\bL_2)$ and let $N=\eps^{-1}$ and $$B_i^{(\eps)}(x)= \frac{\eps^{1/2}}{\sigma}\left(L_i(\eps^{-1}x)-m\eps^{-1}x\right),$$ where $m=a/(1-a)$ and $\sigma^2=a/(1-a)^2$ are the mean and variance $\mathrm{Geom}(a)$, and scale boundary parameters as $c_1=\exp(-\tilde u \eps^{1/2}/\sigma), c_2=\exp(-\tilde v \eps^{1/2}/\sigma)$. Then, $(B_1^{(\eps)}, B_2^{(\eps)})$ weakly converges as $\eps\to 0$ to  $\PP_{\infty}^{\tilde u, \tilde v}$. 

For the log-gamma polymer, consider the measure $\PP^{\alpha, u,v}_{\stat\lgg}(\bL_1,\bL_2)$ and let $N=\eps^{-1}$. The rescaled processes are defined as in the geometric LPP case above except that  now $m=-\psi(\alpha)$ and $\sigma^2=\psi'(\alpha)$ (where $\psi(z)=\partial_z \log \Gamma(z)$ is the digamma function) are the $\log \Ga^{-1}(\alpha)$ mean and variance, and now we scale boundary parameters as $u=\tilde u \eps^{1/2}/\sigma, v=\tilde v \eps^{1/2}/\sigma$. Again, $(B_1^{(\eps)}, B_2^{(\eps)})$ weakly converges as $\eps \to 0$ to  $\PP_{\infty}^{\tilde u, \tilde v}$ (to see that, it is more convenient to use the alternative definition of $\PP^{\alpha, u,v}_{\stat\lgg}(\bL_1,\bL_2)$ from Lemma \ref{lem:alternative reweighted random walk} below).

The fact that the stationary measures of the three models converge to the same limit (as implied by the above convergence results) is not a coincidence, but a sign of universality. Indeed, the three models are expected to converge at large scale to a universal space-time process defined on the strip $[0,1]\times\mathbb R_{\geq 0}$, depending on two boundary parameters  $\tilde u, \tilde v$: the open KPZ fixed point.  The law of $B_1$ under the probability measure $\PP_{\infty}^{\tilde u, \tilde v}$ should be its (unique for each choice of $(\tilde u,\tilde v)$) stationary measure, as predicted in \cite{barraquand2021steady}. Let us mention that the KPZ fixed point has been constructed only for models without boundaries \cite{matetski2021kpz}, and its construction on domains with boundary is an open problem.  Let us also remark that stationary measures of ASEP converge as well to the same limit \cite{bryc2022asymmetric}, and that by scaling boundary parameters slightly differently than above, one could recover all limiting processes studied in \cite{derrida2004asymmetric, bryc2019limit}. To our knowledge all previous fluctuation limit results have been restricted to the fan region where $u+v\geq 0$ while we access the full range of $u,v\in \RR$ (here for the geometric LPP and log-gamma polymer) by our methods.

\subsection{Stationary Gibbs line ensembles}
\label{subsec:stationary Gibbs line ensembles}
In this section, we will explain a general mechanism behind our construction of stationary measures for the geometric LPP and log-gamma polymer models. Though the stationary measures will generally not be of product form, they will have a relatively simple structure -- namely, they will arise as the law of the top layer of certain two-layer Gibbs measures. As we will explain in Section \ref{sec:stationarymeasuresstrip}, the nature of these Gibbs measures and idea behind showing their invariance under the LPP or polymer recurrence dynamics comes from the study of dynamics on variants and generalizations of Schur processes. Our initial motivation for recognizing this structure came from a comparison between the open KPZ equation stationary measure (for $u+v>0$) and the Gibbs property of the KPZ line ensemble. With that in mind, we first relate that motivation and then introduce our two-layer Gibbs measures. This motivation can, in principle, be skipped and one can go immediately to Section \ref{sec:stationarymeasuresstrip}

\subsubsection{Gibbs line ensembles and stationary measures}
Gibbsian line ensembles are useful to describe our guiding principle and understand its potential for applications beyond the scope of the present paper. The distribution of the process $U$ in \eqref{eq:Liouvillereweighting} is closely related to the interaction potential appearing in the definition of the KPZ line ensemble \cite{corwin2016kpz}. We will not review here the exact definition of this object, but simply mention that it is a probability measure on a family of processes $\left(\Lambda_i(x)\right)_{x\in \mathbb R, i\in \mathbb Z_{>0}}$  such that 
\begin{itemize}
    \item The first curve $\Lambda_1(x)$ has the law of the solution to the KPZ equation on $\mathbb R$, with narrow-wedge initial condition, at a fixed time $t$ (the line ensemble may be constructed for any fixed $t>0$); 
    \item If we fix a subset $[a,b]\times \llbracket 1,n\rrbracket \subset \mathbb R\times \mathbb Z_{>0}$, the law of $\left(\Lambda_i(x)\right)_{x\in [a,b], i\in \llbracket 1,n\rrbracket}$, conditionally on the law of $\left(\Lambda_i(x)\right)$ on the exterior of the subdomain, is absolutely continuous with respect to the law of Brownian bridges joining $\Lambda_i(a)$ to $\Lambda_i(b)$, with a Radon-Nikodym derivative proportional to 
    \begin{equation}
        \exp\left(-\sum_{i=1}^{n} \int_a^b e^{-(\Lambda_i(s)-\Lambda_{i+1}(s))}\d s\right).
        \label{eq:interactionpotential}
    \end{equation} 
% A similar property holds for arbitrary compact subdomains of $\mathbb R \times \mathbb Z_{>0}$. 
\end{itemize}
Line ensembles corresponding to various types of initial conditions for the KPZ equation can be considered as well (e.g. the classes of initial data considered in \cite{BCF15,Borodin2015}). The stationary initial condition is somewhat singular, but it can be seen that if the initial condition approaches a Brownian motion, the first line $\Lambda_1(x)$ gets shifted upwards to $+\infty$ , away  from the other lines $\Lambda_i(x)$ for $i\geq 2$. Due to the form of the interaction \eqref{eq:interactionpotential}, in this limit the first curve becomes independent from the others so that $\Lambda_1$ is distributed as a Brownian motion and we recover that Brownian motion is stationary for the KPZ equation on $\mathbb R$ \cite{bertini1997stochastic}.  

Gibbsian line ensembles can be constructed for various models. They were originally introduced in the context of Brownian last passage percolation \cite{corwin2014brownian}, and can be defined as well for discrete models such as geometric LPP \cite{PS02,Spohn05,CLW16,serio2023tightness} and log-gamma polymer models  \cite{corwin2014tropical,johnston2020scaling, wu2019tightness, dimitrov2021tightness, barraquand2021spatial} in the quadrant $\mathbb Z_{>0}^2$. In the case of lattice models,  the analog of the curves $\Lambda_i(x)$ are no longer absolutely continuous with respect to the Brownian bridge measure, but with respect to random walk bridges with geometric or log-gamma distributed increments respectively. The curves are no longer indexed by $x\in \mathbb R$, but rather by the vertices along a down-right path in the quadrant $\mathbb Z_{>0}^2$, such that the distribution of $\Lambda_1(x)$ coincides with the distribution of passage times and free energies, respectively, along the path. The interaction potential \eqref{eq:interactionpotential} needs to be replaced by some discrete analogue which takes the form of a product of Boltzmann weights, as in a usual Gibbs measure. The structure of Boltzmann weights for geometric LPP or the log-gamma polymer comes respectively from the definition of the Schur process \cite{okounkov2003correlation, borodin2014anisotropic, borodin2011schur} and the Whittaker process \cite{borodin2014macdonald, corwin2014tropical} (in particular, their branching rules) and the relation between these models and Schur/Whittaker processes comes from Markovian dynamics that map between these processes and have marginals that recover the LPP or polymer recurrence relations. These dynamics ultimately relate to the skew Cauchy identity satisfied by the Schur or Whittaker symmetric functions. Note that, as in the continuous case, stationary measures of both models can be understood as the law of the first line of the line ensemble after it has detached from the other lines.  

Line ensembles can also be constructed for models with a single boundary, in particular the log-gamma polymer model in a half-quadrant \cite{barraquand2023kpz}. The specific form of interaction between the lines comes from the half-space Whittaker process \cite{o2014geometric, barraquand2020half, bisi2019point} and ultimately from the skew Littlewood identities. 
The half-space KPZ line ensemble $\left(\Lambda_i(x)\right)_{x\in \mathbb R_{\geq 0}, i\in \mathbb Z_{>0}}$ has not been constructed yet in the literature, but we expect that it can be constructed as a limit of the half-space log-gamma line ensemble introduced in \cite{barraquand2023kpz}. Anticipating a bit, we expect that the half-space KPZ line ensemble would satisfy the following Gibbs property:
the law of  $\left(\Lambda_i(x)\right)_{x\in [0,a], i\in \llbracket 1,n\rrbracket}$, conditionally on the law of $\left(\Lambda_i(x)\right)$ on the exterior of the subdomain, is absolutely continuous with respect to Brownian motions on $[0,a]$ terminating at $\Lambda_i(a)$, with a Radon-Nikodym derivative proportional to 
\begin{equation}
        \exp\left( \sum_{i=1}^{n} (-1)^{i}u\Lambda_i(0)-\int_a^b e^{-(\Lambda_i(s)-\Lambda_{i+1}(s))}\d s\right).
        \label{eq:interactionpotentialhalfspace}
    \end{equation} 
We emphasize that there is now a boundary interaction potential of the form $-u(\Lambda_1(0)-\Lambda_2(0)+\Lambda_3(0)-\dots)$, where $u$ is the boundary parameter of the KPZ equation on $\mathbb R_{\geq 0}$. A similar boundary interaction potential appears in the Radon-Nikodym derivative \eqref{eq:Liouvillereweighting} above, upon identifying $2U$ with $\Lambda_1-\Lambda_2$. Stationary measures for the log-gamma polymer and the KPZ equation in a half-space have been studied in \cite{barraquand2021steady, barraquand2022stationary}. The present paper is largely inspired from the observation that for such initial condition, not only the first line of the line ensemble should separate from the bulk of the line ensemble, but the first two lines together, and the stationary measure for both models should then relate to the distribution of half-space KPZ or log-gamma line ensembles with only two lines. This phenomenon seems to be quite general, as it can be shown that stationary measures of open ASEP also present a similar structure \cite{barraquand2023stationary}. Note that the work \cite{barraquand2022stationary} constructs stationary measures using different ideas that do not involve Gibbs line ensembles, but the framework used in \cite{barraquand2022stationary} (half-space Macdonald processes) is not available in the context of models with two boundaries, which are the focus of the present paper. 

Let us now turn to models with two boundaries. When $u+v>0$, the stationary measure of the KPZ equation on $[0,L]$  \eqref{eq:KPZ} can be seen as the distribution of $\Lambda_1(x)-\Lambda_1(0)$, where the law of the pair of continuous processes $(\Lambda_1,\Lambda_2)$ is absolutely continuous with respect to the product of two Brownian measures on $[0,L]$ (where $\Lambda_1$ and $\Lambda_2$ have starting value distributed as Lebesgue measure and then Brownian law from there with diffusivity $1$), with Radon-Nikodym derivative proportional to
\begin{equation}
     \exp\left(-u(\Lambda_1(0)-\Lambda_2(0))-v(\Lambda_1(L)-\Lambda_2(L))-\int_0^L e^{-(\Lambda_1(s)-\Lambda_{2}(s))}\d s\right).
     \label{eq:twolineGibbsmeasure}
\end{equation} 
This is just a restatement of \eqref{eq:Liouvillereweighting}. It becomes clear now that this description matches the Gibbs property of the KPZ line ensemble on $[0,L]$ restricted to two lines with the interaction potential at both boundaries of the form as in half-space models. Strictly speaking, the probability measure defined by \eqref{eq:twolineGibbsmeasure} is invariant by translation, so that it is an infinite measure. However, since we are eventually interested in the spatial increments $\Lambda_1(x)-\Lambda_1(0)$, we may fix the value of any point, e.g. $\Lambda_2(0)=0$. Provided $u+v>0$, doing this results in a finite measure that can be normalized to be a probability measure.

\subsubsection{Gibbs measures on two-layer graphs and stationary measures on a strip}
\label{sec:stationarymeasuresstrip}
Motivated by the discussion above, we formulate here the stationary measures for the geometric LPP and log-gamma polymer model on a strip in terms of the top curve of discrete analogues of the two-line Gibbsian line ensemble \eqref{eq:twolineGibbsmeasure}. The construction is very similar for both models, and we describe below the main ideas. 

As explained above, in the discrete setting, the Gibbsian line ensemble should be indexed by the vertices along a down-right path. Let us consider some  down-right path $\mathcal P$ in the strip (see Figure \ref{fig:strip and two layer graph in intro} (a)). To such a down-right path, we associate a two-layer graph $\PiP$, depicted in Figure \ref{fig:strip and two layer graph in intro} (b), which consists of two copies of the path $\mathcal P$, rotated counter-clockwise by $\pi/4$, with certain interaction edges between the two paths. We refer to Definition \ref{defn:two path diagram and configuration LPP} below for a precise definition.

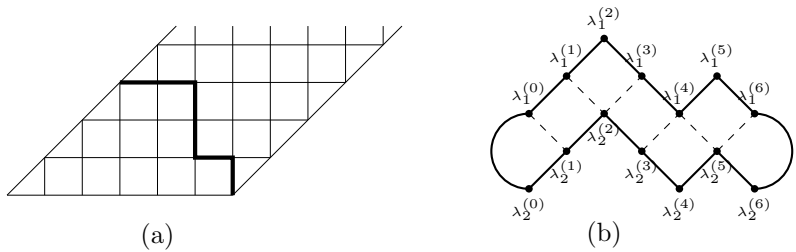
\begin{figure}[h]
\raisebox{-1pt}{\begin{tikzpicture}[scale=0.5]
\draw[very thin] (0,0)--(4.5,4.5);
\draw[very thin] (6,0)--(10.5,4.5);
\draw[very thin] (0,0)--(6,0);
\draw[very thin] (1,1)--(7,1);
\draw[very thin] (2,2)--(8,2);
\draw[very thin] (3,3)--(9,3);  
\draw[very thin] (4,4)--(10,4);  
\draw[very thin] (1,0)--(1,1);
\draw[very thin] (2,0)--(2,2);
\draw[very thin] (3,0)--(3,3);
\draw[very thin] (4,0)--(4,4);
\draw[very thin] (5,0)--(5,4.5);
\draw[very thin] (6,0)--(6,4.5);
\draw[very thin] (7,1)--(7,4.5);
\draw[very thin] (8,2)--(8,4.5); 
\draw[very thin] (9,3)--(9,4.5); 
\draw[very thin] (10,4)--(10,4.5); 
\draw[ultra thick] (3,3)--(5,3)--(5,1)--(6,1)--(6,0);  
\node [below] at (4,-0.5) {(a)};
\end{tikzpicture}}
\quad\quad\quad
\begin{tikzpicture}[scale=0.5]
			\draw[thick] (0,0) -- (2,2) -- (4,0) -- (5,1) -- (6,0);
			\draw[thick] (0,2) -- (2,4) -- (4,2) -- (5,3) -- (6,2);
			\draw[thick] (0,0) arc(270:90:1);
			\draw[thick] (6,2) arc(90:-90:1);
			\fill (0,0) circle(0.1);\node[below] at (0,0) {\fontsize{1}{5}\selectfont $\la_2^{(0)}$};
			\fill (1,1) circle(0.1);\node[below] at (1,1) {\fontsize{1}{5}\selectfont $\la_2^{(1)}$};
			\fill (2,2) circle(0.1);\node[below] at (2,2) {\fontsize{1}{5}\selectfont $\la_2^{(2)}$};
			\fill (3,1) circle(0.1);\node[below] at (3,1) {\fontsize{1}{5}\selectfont $\la_2^{(3)}$};
			\fill (4,0) circle(0.1);\node[below] at (4,0) {\fontsize{1}{5}\selectfont $\la_2^{(4)}$};
			\fill (5,1) circle(0.1);\node[below] at (5,1) {\fontsize{1}{5}\selectfont $\la_2^{(5)}$};
			\fill (6,0) circle(0.1);\node[below] at (6,0) {\fontsize{1}{5}\selectfont $\la_2^{(6)}$}; 
\fill (0,2) circle(0.1);\node[above] at (0,2) {\fontsize{1}{5}\selectfont $\la_1^{(0)}$};
\fill (1,3) circle(0.1);\node[above] at (1,3) {\fontsize{1}{5}\selectfont $\la_1^{(1)}$};
\fill (2,4) circle(0.1);\node[above] at (2,4) {\fontsize{1}{5}\selectfont $\la_1^{(2)}$};
\fill (3,3) circle(0.1);\node[above] at (3,3) {\fontsize{1}{5}\selectfont $\la_1^{(3)}$};
\fill (4,2) circle(0.1);\node[above] at (4,2) {\fontsize{1}{5}\selectfont $\la_1^{(4)}$};
\fill (5,3) circle(0.1);\node[above] at (5,3) {\fontsize{1}{5}\selectfont $\la_1^{(5)}$};
\fill (6,2) circle(0.1);\node[above] at (6,2) {\fontsize{1}{5}\selectfont $\la_1^{(6)}$}; 
\draw[dashed] (0,2) -- (1,1);
\draw[dashed] (1,3) -- (2,2);
\draw[dashed] (2,2) -- (3,3);
\draw[dashed] (3,1) -- (4,2);
\draw[dashed] (4,2) -- (5,1);
\draw[dashed] (5,1) -- (6,2); 
\node [below] at (2,-0.6) {(b)};
		\end{tikzpicture}
  \caption{(a) Strip $\S_N$ for $N=6$, with a down-right (thick) path $\hP$ depicted. (b) Two-layer graph $\PiP$ associated to the path $\hP$.} 
  \label{fig:strip and two layer graph in intro} 
\end{figure}

A two-layer configuration is an assignment of numbers $\la_i^{(j)}$, $i=1,2$, $1\leq j\leq N$ to each vertex of two-layer graph $\PiP$, see Figure \ref{fig:strip and two layer graph in intro} (b). These  take integer values in the case of geometric LPP and real values in the case of log-gamma polymer. We will use the shorthand $\bl_i:=(\la_i^{(0)},\dots,\la_i^{(N)})$ for $i=1,2$ and $\bl=(\bl_1,\bl_2)$. 

We assign to any two-layer configuration on $\PiP$ a weight $\wt^{\PiP}(\bl)$ given by the product of (model dependent) Boltzmann weights over all edges in the graph. The specific form of Boltzmann weights (see \eqref{eq:weightsLPP} and \eqref{eq:weightsLG} below) comes from the structure arising in discrete half-space Gibbs line ensemble \cite{baik2018pfaffian,barraquand2023kpz} discussed in the previous section. 
We will view the weight of a two-layer configuration as an infinite measure over all configurations on $\PiP$ --  referred as the two-layer Gibbs measure below --  which is translation invariant under adding the same constant to all $\la_i^{(j)}$. It may be thought of as a discrete analogue of the measure \eqref{eq:twolineGibbsmeasure}. 
 
It turns out that for both geometric LPP and the log-gamma polymer, the dynamics of the height function can be seen as a marginal of more general Markov dynamics that map the two layer Gibbs measure associated to one path to the two-layer Gibbs measure associated to another path. These Markov dynamics are constructed as follows. First, we notice that any down-right path $\hP$ can be updated to any down-right path $\hQ$ above it by sequentially performing the following three types of `local moves':

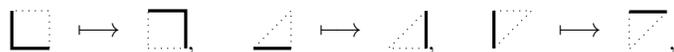
$$  
\begin{tikzpicture}[scale=0.5]
		\draw[dotted] (0,0) -- (1,0)--(1,1)--(0,1)--(0,0);
		\draw[very thick] (0,1) -- (0,0) -- (1,0);
		\end{tikzpicture}
  \quad 
\raisebox{5pt}{$\longmapsto$}
\quad 
\begin{tikzpicture}[scale=0.5]
		\draw[dotted] (0,0) -- (1,0)--(1,1)--(0,1)--(0,0);
		\draw[very thick] (0,1) -- (1,1) -- (1,0);
		\end{tikzpicture}
,\quad\quad
\begin{tikzpicture}[scale=0.5]
		\draw[dotted] (0,0) -- (0,1)--(-1,0)--(0,0);
            \draw[very thick] (0,0) -- (-1,0);
		\end{tikzpicture}
  \quad 
\raisebox{5pt}{$\longmapsto$}
\quad 
\begin{tikzpicture}[scale=0.5]
		\draw[dotted] (0,0) -- (0,1)--(-1,0)--(0,0);
		\draw[very thick] (0,0) -- (0,1);
		\end{tikzpicture}
,\quad\quad
\begin{tikzpicture}[scale=0.5]
		\draw[dotted] (0,0) -- (1,0)--(0,-1)--(0,0);
            \draw[very thick] (0,0) -- (0,-1);
		\end{tikzpicture}
  \quad 
\raisebox{5pt}{$\longmapsto$}
\quad 
\begin{tikzpicture}[scale=0.5]
		\draw[dotted] (0,0) -- (1,0)--(0,-1)--(0,0);
		\draw[very thick] (0,0) -- (1,0);
		\end{tikzpicture},
$$
which we respectively refer to as the bulk/left boundary/right boundary local move. 
We  associate to each local move $\hP\mapsto\widetilde{\hP}$ of a down-right path the corresponding local move of the two-layer graph $\PiP\mapsto\mathcal{G}\widetilde{\hP}$. 

We will define three types of Markov kernels $\U^{\rellcorner}$, $\U^{\reacuteangle}$ and $\U^{\rerotateangle}$ that map two-layer configurations $\bl$ on $\PiP$ to two-layer configurations $\bl'$ on $\mathcal{G}\widetilde{\hP}$. These Markov kernels are local, in the sense that they only update the configuration on the vertex which moves during the local move, and the update only depends on the configuration on neighbouring vertices. 
The bulk kernels $\U^{\rellcorner}$ that we chose below are reminiscent to so-called push-block dynamics, originally introduced in the context of the Schur process \cite{borodin2014anisotropic, warren2009some, borodin2011schur} building on ideas from \cite{diaconis1990strong}. It is possible that other type of dynamics, such as the so-called RSK-type  dynamics \cite{borodin2016nearest, matveev2016q, bufetov2018hall, bufetov2019yang} could be used as well, though the push-block type kernels seemed to be the most convenient choice for our purposes. Likewise, our boundary kernels $\U^{\reacuteangle}$ and $\U^{\rerotateangle}$ are similar with push-block type dynamics introduced in the context of half-space Schur and Macdonald processes  \cite{baik2018pfaffian,barraquand2020half}.  The existence of these kernels follows from certain identities that we show about the local weights in the bulk and boundary for the two-layer Gibbs measures. For instance, in the bulk we have 
\begin{equation}
\label{eq:skewCauchypictorial}
\sum_{\k_1,\k_2}\wt \lb\raisebox{-25pt}{\begin{tikzpicture}
        \draw[thick] (-0.4,0.4)--(0,0)--(0.4,0.4);
        \draw[thick] (-0.4,-0.4)--(0,-0.8)--(0.4,-0.4);
        \draw[dashed] (-0.4,-0.4)--(0,0)--(0.4,-0.4);
        \node[above right] at (0.3,0.3) {\tiny $\mu_1$};
        \node[above right] at (0.3,-0.5) {\tiny  $\mu_2$};
        \node[above left] at (-0.3,0.3) {\tiny  $\la_1$};
        \node[above left] at (-0.3,-0.5) {\tiny  $\la_2$};
       \node[above] at (0.2,0.2) {\textcolor{red}{\fontsize{1}{5}\selectfont $b$}};
        \node[above] at (0.2,-0.6) {\textcolor{red}{\fontsize{1}{5}\selectfont $b$}};
        \node[above] at (-0.2,0.2) {\textcolor{red}{\fontsize{1}{5}\selectfont $a$}};
        \node[above] at (-0.2,-0.6) {\textcolor{red}{\fontsize{1}{5}\selectfont $a$}};
        \node[below] at (0,0) {\tiny $\k_1$};
        \node[below] at (0,-0.8) {\tiny $\k_2$};
\end{tikzpicture}}\rb
=\sum_{\pi_1,\pi_2}\wt \lb\raisebox{-25pt}{\begin{tikzpicture}
    \draw[thick] (-0.4,0.4)--(0,0.8)--(0.4,0.4);
    \draw[dashed] (-0.4,0.4)--(0,0)--(0.4,0.4);
    \draw[thick] (-0.4,-0.4)--(0,0)--(0.4,-0.4);
    \node[below right] at (0.3,0.55) {\tiny $\mu_1$};
    \node[below right] at (0.3,-0.25) {\tiny  $\mu_2$};
    \node[below left] at (-0.3,0.55) {\tiny $\la_1$};
    \node[below left] at (-0.3,-0.25) {\tiny $\la_2$};
    \node[above] at (0,0) {\tiny $\pi_2$};
    \node[above] at (0,0.8) {\tiny $\pi_1$};
    \node[above] at (0.25,-0.25){\textcolor{red}{\fontsize{1}{5}\selectfont $a$}};
    \node[above] at (0.25,0.55){\textcolor{red}{\fontsize{1}{5}\selectfont $a$}};
    \node[above] at (-0.25,-0.25){\textcolor{red}{\fontsize{1}{5}\selectfont $b$}};
    \node[above] at (-0.25,0.55){\textcolor{red}{\fontsize{1}{5}\selectfont $b$}};
\end{tikzpicture}}\rb, 
\end{equation}
where on each side of the equation, the weights are products of Boltzmann weights over all the edges of the subgraphs depicted above. 
We have only stated here the identity corresponding to a bulk local move, and we refer to Propositions \ref{prop:Cauchy identity LPP} and \ref{prop:Littlewood identity LPP} (geometric LPP case) and Lemmas \ref{lem:Cauchy LG} and \ref{lem:Littlewood LG} (log-gamma polymer case)  for details. As we explain in Remark \ref{rem:remarkskewCauchy}, these identities may be seen as variations of the (skew) Cauchy and Littlewood identities in symmetric function theory. 

The kernels $\U^{\rellcorner}$, $\U^{\reacuteangle}$ and $\U^{\rerotateangle}$ also have the property that the first layer marginal of the two-layer dynamics coincides with the recurrence relations \eqref{eq:recurrence geometric LPP} and \eqref{eq:recurrence log gamma} satisfied respectively by  geometric LPP passage times  and by the log-gamma polymer partition functions.  
Hence, by construction, those marginal distributions of two-layer Gibbs measures $\wt^{\PiP}(\bl)$ are precisely the stationary measures we are after.

In order to arrive at a simple description for the stationary measures, we perform several additional steps, by adapting to the discrete setting the procedure that transforms the measure \eqref{eq:twolineGibbsmeasure} to the one in Definition \ref{def:KPZstat}. 
First of all, the two-layer Gibbs measures are translation invariant infinite measures. To normalize them into probability measures, we notice that under certain restrictions of the range of boundary parameters (analogous to $u+v>0$ in the case of the open KPZ equation),
%: $c_1c_2<1$ in geometric LPP and $u+v>0$ in log-gamma polymer, 
the mass of these infinite measures with fixed value $\la_1^{(0)}$ is finite. 
%We get stationary probability measures $\mathrm{P}_{\lpp}^{\hP}$ and $\mathrm{P}_{\lgg}^{\hP}$ under these conditions. 
Therefore one obtains bona fide stationary probability measures under those conditions. 

To get rid of the contraint on  boundary parameters, we finally average the aforementioned (stationary) probability measures over the `zero mode' $\Delta=\la_1^{(0)}-\la_2^{(0)}$. Remarkably, this yields probability measures on $\mathbb Z^N$ or $\mathbb R^N$ that are well-defined without any restrictive constraint on boundary parameters. This procedure is analogous to the one performed non-rigorously in \cite{barraquand2021steady} for stationary measures of the open KPZ equation, as we have mentioned below Theorem \ref{thm:open KPZ stationary measure}. Eventually, the proofs of stationarity of these measures to the full-range of boundary parameters follow analytic continuation arguments. More precisely, let us denote the stationary measure by  $\mathrm P: \mathbb K^N \to [0,1]$, where $\mathbb K$ is $\mathbb Z$ or $\mathbb R$. The dynamics of each model can be encoded by a Markov transition kernel $\mathbf U(\mathbf L_1'\vert \mathbf L_1)$ for $\mathbf L_1,\mathbf L_1'\in \mathbb K^N$, where $\mathbf L_1$ denote the sequence of last passage times or free energies along a path centered by the value on the left-boundary, and $\mathbf L_1'$ denote the same sequence along the same path translated by $(1,1)$, i.e., by $\tau_1$. Assume that under some restriction on boundary parameters (say $u+v>0$),  for  all $\mathbf L_1'\in \mathbb K^N$, we have an equation of the form 
\be 
\label{eq:genericstationarity}
 \sum_{\mathbf L_1'\in \mathbb K^N} \mathbf U(\mathbf L_1'\vert \mathbf L_1)\mathrm P(\mathbf L_1)= \mathrm P(\mathbf L_1'),
\ee
such as \eqref{eq:compatibility with probability measure after analytic continuatio LPP} and \eqref{eq:compatibility with probability measure after analytic continuatio LG} below (the summation is an integral when $\mathbb K=\R$), where $\mathbf U$ and $\mathrm P$ depend on boundary parameters. The equation may be analytically extended to a larger range of  $u,v$, provided both sides of \eqref{eq:genericstationarity} are real analytic functions of the variable  $u$. In the case of geometric LPP, we show that the RHS is analytic by direct inspection of an explicit formula, and we prove that the LHS is a power series in the boundary parameter with appropriate radius of convergence. In the log-gamma case the argument is considerably more involved as the sum in \eqref{eq:genericstationarity} becomes an integral. In Section \ref{sec:proof of real analyticity}, we demonstrate a stronger result than real analyticity. Using Morera's theorem and bounds on integrability we show that both sides are holomorphic functions in $u$ in a suitable open set.
\begin{remark}  
For open ASEP a simple argument shows that its stationary measure depends real analytically on bulk and boundary parameters $q,\alpha,\beta,\gamma,\delta$, in the region of the parameter space when the stationary measure is unique (we refer to \cite{derrida1993exact} for a precise definition of the model). 
This actually holds for a  wide class of Markov chains with finite state space.  Suppose $\mathcal{C}_{\boldsymbol{\theta}}$ is a family of continuous-time Markov chains with the same state space $\mathcal{S}$ (where $|\mathcal{S}|<\infty$) and with infinitesimal generator matrices $L_{\boldsymbol{\theta}}$ depending real analytically on a set of parameters $\boldsymbol{\theta}=(\theta_1,\dots,\theta_m)\in\Omega$, where $\Omega$ is an open subset of $\RR^m$. We also assume that the stationary measure of $\mathcal{C}_{\boldsymbol{\theta}}$, denoted $\mu_{\boldsymbol{\theta}}$,   is unique. Then $\mu_{\boldsymbol{\theta}}$ depend real 
analytically on $\boldsymbol{\theta}\in\Omega$. 

Indeed, by definition, the column vector $\mu_{\boldsymbol{\theta}}\in\RR^{|\mathcal{S}|}$ is the unique solution of equations $L_{\boldsymbol{\theta}}^*\mu_{\boldsymbol{\theta}}=0$ and $[1,\dots,1]\mu_{\boldsymbol{\theta}}=1$. 
The system can be solved by Gaussian elimination, so that each component of  $\mu_{\boldsymbol{\theta}}$  can be expressed as a rational function of the coefficients in $L_{\boldsymbol{\theta}}^*$, hence $\mu_{\boldsymbol{\theta}}$ depends real analytically on $\boldsymbol{\theta}\in\Omega$.

This argument does not apply for the geometric LPP and log-gamma polymer models, which have infinite state spaces. It may, however apply to some of the models that we mention now in Section \ref{sec:extensions}
\end{remark} 

\subsection{Extensions of our method}
\label{sec:extensions} Our methods should be extendable to a broad class of models on a strip from integrable probability that satisfy Cauchy and Littlewood type summation identities. From these it should be possible to construct two-layer Gibbs measures and Markovian dynamics that preserve them as well as project on the top layer to the models in question. The class of models we anticipate being able to approach includes those coming from symmetric functions with explicit and local branching relations in the Macdonald hierarchy as well as the symmetric functions coming from stochastic vertex models.

As explained in Section \ref{subsec:Skew Schur functions indexed by signatures LPP}, our construction of geometric LPP stationary measures is closely related to Schur functions.  Based on \cite{o2014geometric}, and the form of the weights \eqref{eq:weightsLG} below, there exists a similar relation between the log-gamma polymer on a strip  and $\mathfrak{gl}_n(\RR)$-Whittaker functions. We have not made this connection fully explicit in Section \ref{sec:LG} below, in order to keep the proofs elementary and self-contained, see Remark \ref{rem:Whittaker}.  
Since the family of $q$-Whittaker functions interpolates between the Schur and Whittaker functions, our methods should be adaptable to the $q$-Whittaker case. Models connected to the $q$-Whittaker process, such as the $q$-pushTASEP are discussed in \cite{matveev2016q, borodin2013discrete} in full-space and in \cite{barraquand2020half,imamura2022solvable} in half-space.  
It would also be interesting to adapt our methods to models related to Hall-Littlewood functions, in particular the stochastic six-vertex model and ASEP \cite{borodin2016stochastic, borodin2016between, he2023boundary}. In this case, the Gibbs line ensemble structure has been already studied in \cite{corwin2018transversal} (in the full-space setting), and it turns out that the summation identity of the form \eqref{eq:skewCauchypictorial} has already been proved in \cite{bufetov2018hall} (see more details in Remark \ref{rem:remarkskewCauchy} below). Hence, provided the appropriate Littlewood type summation identity holds as well, it seems very likely that the method of the present paper can be adapted.

Beyond the class of Macdonald symmetric functions, it would be interesting to consider models associated with symmetric rational functions built as partition functions of stochastic vertex models \cite{corwin2015stochastic, borodin2018higher}, in particular the spin Hall-Littlewood functions \cite{borodin2017family} and spin $q$-Whittaker functions \cite{borodin2021spin}. Again, the summation identity \eqref{eq:skewCauchypictorial} has also been proved for the spin Hall-Littlewood functions, using a Yang-Baxter ``zipper'' argument which is likely to be applicable to many other cases, including non-symmetric functions that are related to colored vertex models and interacting particle systems  with particles of multiple types  \cite{borodin2018coloured, aggarwal2021colored}. Skew Littlewood type identities for spin Hall-Littlewood functions  have been also proved in \cite{chen2021stable, gavrilova2021refined}, and although these are slightly different from what we need to build stationary measures (we need to consider summations over all integer signatures instead of summations over nonnegative signatures), we believe that similar arguments would yield the appropriate identities.

There exists another direction in which our method could be extended.  In the case of the geometric LPP model, our two-layer Gibbs measures can be viewed as  products of skew Schur functions indexed by  integer signatures (not necessarily nonnegative) of length $2$. We refer to Section \ref{subsec:Skew Schur functions indexed by signatures LPP} for background on Schur functions and signatures. 
Thus, the two-layer Gibbs measures have a structure similar to  Schur processes \cite{okounkov2003correlation} and their variants with boundaries  \cite{borodin2005eynard, betea2018free}, modulo the important difference that our measures live on sequences of integer signatures  instead of sequences of partitions, and we are considering an infinite measure instead of a probability measure. We refer to Remark \ref{rem:comparisonwithSchurmeasures} for more details on Schur processes. Based on this analogy, our construction may be generalized in the following directions.
%Thus, the two-layer Gibbs measures may  be seen as a variant of the Schur process, or more precisely a variant of the two-boundary Schur process \cite{betea2018free}, such that our measures live on sequences of integer signatures  instead of sequences of partitions, and we are considering an infinite measure instead of a probability measure. Our construction may be generalized in several directions. 
As the summation identities from Section \ref{subsec:Skew Schur functions indexed by signatures LPP} hold for an arbitrary number of variables, they must hold in the ring of symmetric functions, and hence one could study the Gibbs measures associated to other types of specializations of the ring of symmetric functions (see \cite{matveev2016q} for an illustration of how different specializations are related to different stochastic models). Moreover, these summation identities (Proposition \ref{prop:Cauchy identity LPP} and Proposition \ref{prop:Littlewood identity LPP}) also hold for signatures of arbitrary (even) length. Hence one can construct $n$-layer Gibbs measures and associated  Markov dynamics  on sequences of signatures of length $n$. It would be interesting to understand the probabilistic information that is contained in those (infinite) measures. 

\subsection{Outline of the paper}
In Section \ref{sec:LPP}, we consider an inhomogeneous generalization of the geometric LPP model from Definition \ref{defn:inhomogeneous geometric LPP and specializations} and prove Theorem \ref{thm:main theorem LPP} which generalizes Theorem \ref{thm:stationary measure LPP intro}.
Section \ref{sec:LG} contains an inhomogeneous generalization of the log-gamma polymer model from Definition \ref{defn:inhomogeneous LG and specializations} and proves Theorem \ref{thm:main theorem LG} which generalizes Theorem \ref{thm:stationary measure log-gamma intro}.
Section \ref{sec:stationary measure KPZ} consider the intermediate disorder scaling of log-gamma polymer model on a strip (Conjecture \ref{conj:convergence}). Proposition \ref{prop:convergencestationarymeasures} shows that the stationary measure for the log-gamma polymer in Theorem \ref{thm:stationary measure log-gamma intro} converges to the open KPZ stationary measure. Hence, modulo Conjecture \ref{conj:convergence} on convergence of the model, we prove Theorem \ref{thm:KPZ stat}, verifying the conjectured \cite{barraquand2021steady} open KPZ stationary measure.

\subsection*{Notation}
We write $\ZZ_{\geq a}:= \ZZ\cap [a,\infty)$. Bold face letters such as $\bX$ are used to denote vectors. For a distribution $D$, we write $X \sim D$ to mean that $X$ is a random variable with distribution $D$ and generally assume $X$ to be independent of other random variables. We write $U(x'|x)$ (or $U$ in various fonts) for transition probabilities from $x$ to $x'$. When the $x,x'$ take discrete values (i.e., $\ZZ$ or products of $\ZZ$) as in the `geometric' case, $U$ is a transition probability;  when $x,x'$ take continuous values (i.e., $\RR$ or products of $\RR$) as in the `log-gamma' case, $U$ is a transition probability density. In that case, we will also use Dirac delta functions if we want to indicate that $U$ acts as the identity on certain coordinates. We will use a subscript $\lpp$ and $\lgg$ to distinguish between the `geometric' and `log-gamma' cases of various notation. A summary of much of the notation used is contained in the following table.
$$
\begin{tabular}{ |p{4.2cm}|p{2cm}|p{3.2cm}|p{2cm}|p{3.2cm}|  }
\hline 
 & \multicolumn{2}{|c|}{Geometric LPP} & \multicolumn{2}{|c|}{Log-gamma polymer}\\
\hline 
 & Two-layer & First-layer marginal & Two-layer & First-layer marginal \\ 
\hline
Random walk measures & $\PP^{\aa,\bb}_{\grw}$ & & $\PP^{\aaa,\bbb}_{\lgrw}$ & \\
Reweighting functions & $V^{c_1,c_2}_{\lpp}$ & & $V^{u,v}_{\lgg}$ & \\
Reweighted random walks & $\PP^{\bb,c_1,c_2}_{\lpp}$ & $\mathrm{P}^{\bb,c_1,c_2}_{\lpp}$ & $\PP^{\bbb,u,v}_{\lgg}$ & $\mathrm{P}^{\bbb,u,v}_{\lgg}$ \\
Gibbs measures & $\wt_{\lpp}^\PiP$ & $\wt_{\lpp}^{\hP}$& $\wt_{\lgg}^\PiP$& $\wt_{\lgg}^\hP$\\
 Gibbs probability measures & & $\mathrm{P}_{\lpp}^{\hP}$ & & $\mathrm{P}_{\lgg}^{\hP}$\\
Markov Dynamics & $\mathcal{U}_{\lpp}^{\mathcal{GP},\mathcal{GQ}}$ & $\mathbf{U}_{\lpp}^{\mathcal{P},\mathcal{Q}}$ & $\mathcal{U}_{\lgg}^{\mathcal{GP},\mathcal{GQ}}$ &
$\mathbf{U}_{\lgg}^{\mathcal{P},\mathcal{Q}}$\\
Bulk Markov operators & $\mathcal{U}_{\lpp}^{\rellcorner}$ & $\mathbf{U}_{\lpp}^{\rellcorner}$ & $\mathcal{U}_{\lgg}^{\rellcorner}$ & $\mathrm{U}_{\lgg}^{\rellcorner}$\\
Left boundary operators & $\mathcal{U}_{\lpp}^{\reacuteangle}$ & $\mathbf{U}_{\lpp}^{\reacuteangle}$ & $\mathcal{U}_{\lgg}^{\reacuteangle}$ & $\mathrm{U}_{\lgg}^{\reacuteangle}$\\
Right boundary operators & $\mathcal{U}_{\lpp}^{\rerotateangle}$ & $\mathbf{U}_{\lpp}^{\rerotateangle}$ & $\mathcal{U}_{\lgg}^{\rerotateangle}$ & $\mathrm{U}_{\lgg}^{\rerotateangle}$\\
\hline
\end{tabular} 
$$

\subsection*{Acknowledgements} 
G.B. and I.C. would like to acknowledge the `Universality and Integrability in Random Matrix Theory and Interacting Particle Systems' semester program at MSRI in Fall 2021 (NSF grant DMS:1928930) during which discussions related to this paper were initialized.
G.B. was partially supported by Agence Nationale de la Recherche through grants ANR-21-CE40-0019 and ANR-23-ERCB-0007. 
I.C. was partially supported by the NSF through grants DMS:1811143 and DMS:1937254, through the Simons Foundation through a Fellowship in Mathematics grant (\#817655) and an Investigator in Mathematics grant (\#929852), and through the W.~M.~Keck Foundation Science and Engineering Grant on Extreme Diffusion.
Z.Y. was partially supported by I.C.’s NSF grant DMS-1811143 as well as the Fernholz Foundation’s `Summer Minerva Fellows' program.

\section{Stationary measure for geometric LPP on a strip}
\label{sec:LPP}
\subsection{LPP with inhomogeneous weights} 
\label{subsec:LPP with inhomogeneous weights}
We define the geometric LPP with inhomogeneous weights and then state the main theorem of this section constructing its stationary measure on a horizontal path, which generalizes Theorem \ref{thm:stationary measure LPP intro} in the introduction.

\begin{definition}[Inhomogeneous geometric LPP]\label{defn:inhomogeneous geometric LPP and specializations} 
Let $a_1,\dots, a_N>0$ and $c_1,c_2>0$ be such that:
\be\label{eq:condition for inhomogeneous geometric LPP}
a_ia_j<1,\quad a_ic_1<1,\quad a_ic_2<1,\quad\forall 1\leq i,j\leq N.
\ee 
We will always assume these conditions in this paper as they are necessary for the geometric random variables below to be defined. We call the $a_i$ `bulk parameters' and the $c_1,c_2$ `boundary parameters'.
Define $a_{j+kN}=a_j$ for $k\in\mathbb{Z}$ and $1\leq j\leq N$.
We define the inhomogeneous version of geometric LPP on a strip by the same recurrence \eqref{eq:recurrence geometric LPP} and initial condition as the homogeneous model, but now with $\omega_{n,m}\sim\Geom(a_na_m)$ in the bulk $0\leq m<n<m+N$, $\omega_{m,m}\sim\Geom(a_mc_1)$ on the left boundary and $\omega_{m+N,m}\sim\Geom(a_mc_2)$ on the right boundary.  
We label each edge of the strip by a number, which will be needed later to define Gibbs  measures. In particular, we label the horizontal edge $(n-1,m)\rightarrow(n,m)$ by $a_n$ and the vertical edge $(n,m-1)\rightarrow(n,m)$ by $a_m$. See Figure \ref{fig:strip for N=5} for an illustration.
\end{definition}
\begin{figure}[h]
    \centering
    \begin{tikzpicture}[scale=1.5]
\draw[dotted] (0,0)--(3.5,3.5);
\draw[dotted] (5,0)--(8.5,3.5);
\draw[dotted] (0,0)--(5,0);
\draw[dotted] (1,1)--(6,1);
\draw[dotted] (2,2)--(7,2);
\draw[dotted] (3,3)--(8,3);
\draw[dotted] (1,0)--(1,1);
\draw[dotted] (2,0)--(2,2);
\draw[dotted] (3,0)--(3,3);
\draw[dotted] (4,0)--(4,3.5);
\draw[dotted] (5,0)--(5,3.5);
\draw[dotted] (6,1)--(6,3.5);
\draw[dotted] (7,2)--(7,3.5);
\draw[dotted] (8,3)--(8,3.5);
\node at (0,0) {\small $c_1a_5$};
\node at (1,0) {\small $a_1a_5$};
\node at (2,0) {\small $a_2a_5$};
\node at (3,0) {\small $a_3a_5$};
\node at (4,0) {\small $a_4a_5$};
\node at (5,0) {\small $c_2a_5$};
\node at (1,1) {\small $c_1a_1$};
\node at (2,1) {\small $a_2a_1$};
\node at (3,1) {\small $a_3a_1$};
\node at (4,1) {\small $a_4a_1$};
\node at (5,1) {\small $a_5a_1$};
\node at (6,1) {\small $c_2a_1$};
\node at (2,2) {\small $c_1a_2$};
\node at (3,2) {\small $a_3a_2$};
\node at (4,2) {\small $a_4a_2$};
\node at (5,2) {\small $a_5a_2$};
\node at (6,2) {\small $a_1a_2$};
\node at (7,2) {\small $c_2a_2$};
\node at (3,3) {\small $c_1a_3$};
\node at (4,3) {\small $a_4a_3$};
\node at (5,3) {\small $a_5a_3$};
\node at (6,3) {\small $a_1a_3$};
\node at (7,3) {\small $a_2a_3$};
\node at (8,3) {\small $c_2a_3$};
\node at (0.5,0) {\textcolor{red}{\small $a_1$}};
\node at (1.5,0) {\textcolor{red}{\small $a_2$}};
\node at (2.5,0){\textcolor{red}{\small $a_3$}};
\node at (3.5,0) {\textcolor{red}{\small $a_4$}};
\node at (4.5,0) {\textcolor{red}{\small $a_5$}};
\node at (1.5,1) {\textcolor{red}{\small $a_2$}};
\node at (2.5,1) {\textcolor{red}{\small $a_3$}};
\node at (3.5,1) {\textcolor{red}{\small $a_4$}};
\node at (4.5,1) {\textcolor{red}{\small $a_5$}};
\node at (5.5,1) {\textcolor{red}{\small $a_1$}};
\node at (2.5,2) {\textcolor{red}{\small $a_3$}};
\node at (3.5,2) {\textcolor{red}{\small $a_4$}};
\node at (4.5,2) {\textcolor{red}{\small $a_5$}};
\node at (5.5,2){\textcolor{red}{\small $a_1$}};
\node at (6.5,2) {\textcolor{red}{\small $a_2$}};
\node at (3.5,3) {\textcolor{red}{\small $a_4$}};
\node at (4.5,3) {\textcolor{red}{\small $a_5$}};
\node at (5.5,3) {\textcolor{red}{\small $a_1$}};
\node at (6.5,3) {\textcolor{red}{\small $a_2$}};
\node at (7.5,3) {\textcolor{red}{\small $a_3$}};
\node at (1,0.5){\textcolor{red}{\small $a_1$}};
\node at (2,0.5){\textcolor{red}{\small $a_1$}};
\node at (3,0.5){\textcolor{red}{\small $a_1$}};
\node at (4,0.5) {\textcolor{red}{\small $a_1$}};
\node at (5,0.5) {\textcolor{red}{\small $a_1$}};
\node at (2,1.5) {\textcolor{red}{\small $a_2$}};
\node at (3,1.5) {\textcolor{red}{\small $a_2$}};
\node at (4,1.5) {\textcolor{red}{\small $a_2$}};
\node at (5,1.5) {\textcolor{red}{\small $a_2$}};
\node at (6,1.5) {\textcolor{red}{\small $a_2$}};
\node at (3,2.5){\textcolor{red}{\small $a_3$}};
\node at (4,2.5) {\textcolor{red}{\small $a_3$}};
\node at (5,2.5) {\textcolor{red}{\small $a_3$}};
\node at (6,2.5) {\textcolor{red}{\small $a_3$}};
\node at (7,2.5){\textcolor{red}{\small $a_3$}};
\end{tikzpicture}
    \caption{Definition of the model for $N=5$. The numbers on vertices are parameters of the geometric random variables. The numbers labelling the edges are in red.}
    \label{fig:strip for N=5}
\end{figure}
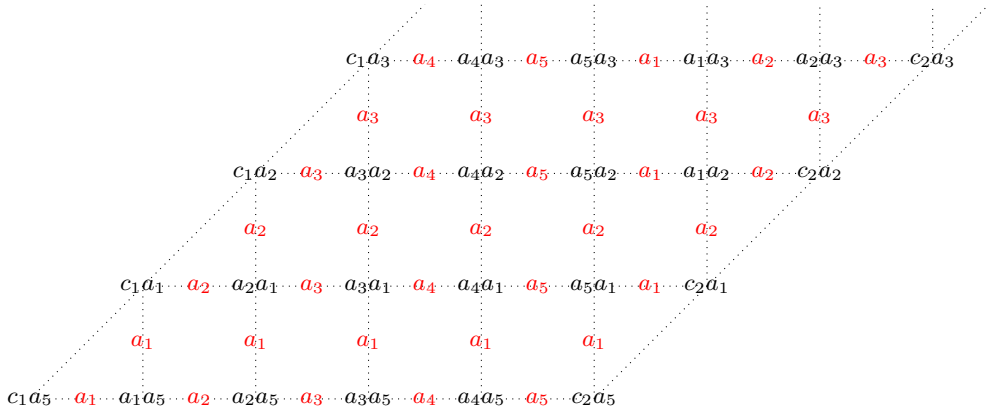

Next we introduce the inhomogeneous version of reweighted geometric random walk generalizing Definition \ref{defn:rescaled random walks LPP}. We then introduce the main theorem of this section that this reweighted random walk is stationary under inhomogeneous geometric LPP, which generalizes Theorem \ref{thm:stationary measure LPP intro}. 

\begin{definition}[Reweighted inhomogeneous geometric random walks] \label{defn:rescaled random walks LPP inhomogeneous}
Assume that $(a_i)_{1\leq i \leq N}\in (0,1)^N$ and $c_1,c_2$ satisfy \eqref{eq:condition for inhomogeneous geometric LPP}.
Suppose $\hP$ is a horizontal path with labels on edges $\bb=(b_1,\dots,b_N)$ from left to right ($\bb$ must be a cyclic shift of $\aa=(a_1,\dots,a_N)$).
Consider two independent random walks  $\bL_1=\big(L_1(j)\big)_{1\leq j\leq N}\in \ZZ_{\geq 0}^{N}$ and $\bL_2=\big(L_2(j)\big)_{1\leq j\leq N}\in \ZZ_{\geq 0}^{N}$ starting from  $L_1(0)=L_2(0)=0$ with independent increments distributed for $1\leq j\leq N$ as
$$L_1{(j)}-L_1{(j-1)} \sim \Geom(b_j),\quad\textrm{and}\quad L_2{(j)}-L_2{(j-1)}\sim \Geom(b_j).$$
We use shorthand $\bL=(\bL_1,\bL_2)$ and denote by $\PP^{\bb,\bb}_{\grw}\lbe\bL\rbe$  and $\mathbb{E}^{\bb,\bb}_{\grw}$  the associated probability measure and expectation.
We define a new probability measure $\PP_{\stat\lpp}^{\bb,c_1,c_2}$ by reweighting the measure $\PP^{\bb,\bb}_{\grw}$ as 
 \be \label{eq:PstatLPP inhomogeneous}
\PP^{\bb,c_1,c_2}_{\stat\lpp}(\bL):= \frac{V^{c_1,c_2}_{\lpp}(\bL)\PP^{\bb,\bb}_{\grw}(\bL)}{\hZ^{\bb,c_1,c_2}_{\lpp}}\quad\textrm{where}\quad V^{c_1,c_2}_{\lpp}(\bL):=(c_1c_2)^{\max_{1\leq j\leq N}(L_2{(j)}-L_1{(j-1)})}c_2^{L_1{(N)}-L_2{(N)}},\ee 
and the normalizing constant 
$\hZ^{\bb,c_1,c_2}_{\lpp}:=\mathbb{E}^{\bb,\bb}_{\grw}\lbE V^{c_1,c_2}_{\lpp}(\bL)\rbE$
 will be proved to be finite in Proposition \ref{prop:finiteness of weights after summing zero mode LPP}. Let $\mathrm{P}^{\bb,c_1,c_2}_{\stat\lpp}(\bL_1)$ denote the marginal law of $\bL_1$ for the probability measure $\PP^{\bb,c_1,c_2}_{\stat\lpp}(\bL)$ where $\bL=(\bL_1,\bL_2)$.
\end{definition}

The following is the main theorem in this section. 

 \begin{theorem}\label{thm:main theorem LPP}
Assume that $(a_i)_{1\leq i \leq N}\in (0,1)^N$ and $c_1,c_2>0$ satisfy \eqref{eq:condition for inhomogeneous geometric LPP}.
Suppose $\bb=(b_1,\dots,b_N)$ are labels on some horizontal path $\hP$, and $\tau_1\bb=(b_2,\dots,b_N,b_1)$ are labels on the translated path $\tau_1\hP$. 
Consider the inhomogeneous geometric LPP model from Definition \ref{defn:inhomogeneous geometric LPP and specializations} starting from an initial condition along $\hP$ given by $(G(\p_j))_{0\leq j\leq N}$ where $(G(\p_j)-G(\p_0))_{1\leq j\leq N}$ is distributed according to $\mathrm{P}^{\bb,c_1,c_2}_{\stat\lpp}$.
Then the distribution of $(G(\tau_1\p_j)-G(\tau_1\p_0))_{1\leq j\leq N}$ coincides with
$\mathrm{P}^{\tau_1\bb,c_1,c_2}_{\stat\lpp}$.
For the homogeneous (i.e. $\aa=(a,\ldots, a)$) case,
$\mathrm{P}^{\aa,c_1,c_2}_{\stat\lpp}$ is the (unique) ergodic stationary measure on the horizontal path $\hP_h$ for the geometric LPP recurrence relation (Definition \ref{defn:homogeneous geometric LPP}).
\end{theorem} 
The proof of this theorem is given in Section \ref{subsec:proof of main theorem LPP}. In Section \ref{subsec:Two-layer Gibbs-like measure} we construct the two-layer Gibbs measures indexed by down-right paths. 
We prove certain summation identities of symmetric functions in Section \ref{subsec:Skew Schur functions indexed by signatures LPP}, based on which we will construct in Section \ref{subsec:local Markov operators LPP} a Markov dynamics under which the two-layer Gibbs measures are stationary. Under $c_1c_2<1$, in Section \ref{subsec:Marginal distributions LPP} we will turn the two-layer Gibbs measures into probability measures which are stationary under geometric LPP. To obtain stationarity outside $c_1c_2<1$, in Section \ref{subsec:proof of main theorem LPP} we argue by the uniqueness of analytic continuation of real analytic functions.  
Let us also note here that though this theorem is formulated along only horizontal paths, it is possible to formulate and prove such a result along any down-right path in a similar manner.

\subsection{Two-layer Gibbs measure}\label{subsec:Two-layer Gibbs-like measure}
\begin{definition}[Two-layer graph]\label{defn:two path diagram and configuration LPP}
For any down-right path $\hP$ on the strip we create a two-layer graph $\PiP$ by the following steps. We first rotate $\hP$ counter-clockwise by $\pi/4$ to get the upper layer (i.e., all $\downarrow$ edges in $\hP$ become $\searrow$ in $\PiP$ and $\rightarrow$ edges in $\hP$ become $\nearrow$ in $\PiP$). 
The vertices of the upper layer are denoted by $\p_1^{(0)},\dots,\p_1^{(N)}$ and edges $\e_1^{(1)},\dots,\e_1^{(N)}$ from the left to the right, where $\e_1^{(j)}$ connects $\p_1^{(j-1)}$ and $\p_1^{(j)}$, for $1\leq j\leq N$.
The lower layer is the downwards translation of the upper layer by $\sqrt{2}$, with vertices  $\p_2^{(0)},\dots,\p_2^{(N)}$ and edges $\e_2^{(1)},\dots,\e_2^{(N)}$ from left to  right.
We draw dotted edges of slope $\pi/4$ or $-\pi/4$ to connect those pairs of vertices with distance $1$ that have not been connected by the edges in the layers (i.e., the edges between the layers). Finally, we draw two solid arcs on the left and right boundaries, connecting respectively $\p_1^{(0)}$ with $\p_2^{(0)}$ and $\p_1^{(N)}$ with $\p_2^{(N)}$. This is the graph $\PiP$.
Next we label all the solid edges and arcs real numbers (i.e., one-variable `specializations'). For $1\leq j\leq N$ the label for edges $\e_1^{(j)}$ in the upper layer and $\e_2^{(j)}$ in the lower layer is that of the edge $\e_j$ in the down-right path $\hP$ in the strip. The left boundary arc is labelled by $c_1$ and the right boundary arc is labelled by $c_2$.
 
A two-layer configuration $\bl=(\la_1^{(0)},\dots,\la_1^{(N)},\la_2^{(0)},\dots,\la_2^{(N)})$ is an assignment of $\la_i^{(j)}\in\ZZ$ to each vertex $\p_i^{(j)}$ of $\PiP$, for $i=1,2$ and $j=0,\dots,N$. See Figure \ref{fig:two-path diagram} for an illustration of these definitions.
\end{definition}

\begin{figure}[h]
\raisebox{-1pt}{\begin{tikzpicture}[scale=0.9]
\draw[dotted] (0,0)--(3.5,3.5);
\draw[dotted]  (6,0)--(9.5,3.5);
\draw[dotted]  (0,0)--(6,0);
\draw[dotted]  (1,1)--(7,1);
\draw[dotted]  (2,2)--(8,2);
\draw[dotted]  (3,3)--(9,3);   
\draw[dotted]  (1,0)--(1,1);
\draw[dotted]  (2,0)--(2,2);
\draw[dotted]  (3,0)--(3,3);
\draw[dotted]  (4,0)--(4,3.5);
\draw[dotted]  (5,0)--(5,3.5);
\draw[dotted]  (6,0)--(6,3.5);
\draw[dotted]  (7,1)--(7,3.5);
\draw[dotted]  (8,2)--(8,3.5); 
\draw[dotted]  (9,3)--(9,3.5);  
\node at (0.5,0) {\textcolor{red}{\small $a_1$}};
\node at (1.5,0) {\textcolor{red}{\small $a_2$}};
\node at (2.5,0){\textcolor{red}{\small $a_3$}};
\node at (3.5,0) {\textcolor{red}{\small $a_4$}};
\node at (4.5,0) {\textcolor{red}{\small $a_5$}};
\node at (5.5,0) {\textcolor{red}{\small $a_6$}};
\node at (1.5,1) {\textcolor{red}{\small $a_2$}};
\node at (2.5,1) {\textcolor{red}{\small $a_3$}};
\node at (3.5,1) {\textcolor{red}{\small $a_4$}};
\node at (4.5,1) {\textcolor{red}{\small $a_5$}};
\node at (5.5,1) {\textcolor{red}{\small $a_6$}};
\node at (6.5,1) {\textcolor{red}{\small $a_1$}};
\node at (2.5,2) {\textcolor{red}{\small $a_3$}};
\node at (3.5,2) {\textcolor{red}{\small $a_4$}};
\node at (4.5,2) {\textcolor{red}{\small $a_5$}};
\node at (5.5,2){\textcolor{red}{\small $a_6$}};
\node at (6.5,2) {\textcolor{red}{\small $a_1$}};
\node at (7.5,2) {\textcolor{red}{\small $a_2$}};
\node at (3.5,3) {\textcolor{red}{\small $a_4$}};
\node at (4.5,3) {\textcolor{red}{\small $a_5$}};
\node at (5.5,3) {\textcolor{red}{\small $a_6$}};
\node at (6.5,3) {\textcolor{red}{\small $a_1$}};
\node at (7.5,3) {\textcolor{red}{\small $a_2$}};
\node at (8.5,3) {\textcolor{red}{\small $a_3$}};
\node at (1,0.5){\textcolor{red}{\small $a_1$}};
\node at (2,0.5){\textcolor{red}{\small $a_1$}};
\node at (3,0.5){\textcolor{red}{\small $a_1$}};
\node at (4,0.5) {\textcolor{red}{\small $a_1$}};
\node at (5,0.5) {\textcolor{red}{\small $a_1$}};
\node at (6,0.5) {\textcolor{red}{\small $a_1$}};
\node at (2,1.5) {\textcolor{red}{\small $a_2$}};
\node at (3,1.5) {\textcolor{red}{\small $a_2$}};
\node at (4,1.5) {\textcolor{red}{\small $a_2$}};
\node at (5,1.5) {\textcolor{red}{\small $a_2$}};
\node at (6,1.5) {\textcolor{red}{\small $a_2$}};
\node at (7,1.5) {\textcolor{red}{\small $a_2$}};
\node at (3,2.5){\textcolor{red}{\small $a_3$}};
\node at (4,2.5) {\textcolor{red}{\small $a_3$}};
\node at (5,2.5) {\textcolor{red}{\small $a_3$}};
\node at (6,2.5) {\textcolor{red}{\small $a_3$}};
\node at (7,2.5){\textcolor{red}{\small $a_3$}};
\node at (8,2.5){\textcolor{red}{\small $a_3$}};
\draw[thick] (3,3)--(5,3)--(5,1)--(6,1)--(6,0);
\node [below] at (4,-0.6) {(a)};
\end{tikzpicture}}
%\quad 
\begin{tikzpicture}[scale=0.8]
			\draw[thick] (0,0) -- (2,2) -- (4,0) -- (5,1) -- (6,0);
			\draw[thick] (0,2) -- (2,4) -- (4,2) -- (5,3) -- (6,2);
			\draw[thick] (0,0) arc(270:90:1);
			\draw[thick] (6,2) arc(90:-90:1);
			\fill (0,0) circle(0.1);\node[below] at (0,0) {\small $\la_2^{(0)}$};
			\fill (1,1) circle(0.1);\node[below] at (1,1) {\small $\la_2^{(1)}$};
			\fill (2,2) circle(0.1);\node[below] at (2,2) {\small $\la_2^{(2)}$};
			\fill (3,1) circle(0.1);\node[below] at (3,1) {\small $\la_2^{(3)}$};
			\fill (4,0) circle(0.1);\node[below] at (4,0) {\small $\la_2^{(4)}$};
			\fill (5,1) circle(0.1);\node[below] at (5,1) {\small $\la_2^{(5)}$};
			\fill (6,0) circle(0.1);\node[below] at (6,0) {\small $\la_2^{(6)}$}; 
\fill (0,2) circle(0.1);\node[above] at (0,2) {\small $\la_1^{(0)}$};
\fill (1,3) circle(0.1);\node[above] at (1,3) {\small $\la_1^{(1)}$};
\fill (2,4) circle(0.1);\node[above] at (2,4) {\small $\la_1^{(2)}$};
\fill (3,3) circle(0.1);\node[above] at (3,3) {\small $\la_1^{(3)}$};
\fill (4,2) circle(0.1);\node[above] at (4,2) {\small $\la_1^{(4)}$};
\fill (5,3) circle(0.1);\node[above] at (5,3) {\small $\la_1^{(5)}$};
\fill (6,2) circle(0.1);\node[above] at (6,2) {\small $\la_1^{(6)}$}; 
\draw[dashed] (0,2) -- (1,1);
\draw[dashed] (1,3) -- (2,2);
\draw[dashed] (2,2) -- (3,3);
\draw[dashed] (3,1) -- (4,2);
\draw[dashed] (4,2) -- (5,1);
\draw[dashed] (5,1) -- (6,2); 
\node[left] at (-1,1) {\textcolor{red}{\small $c_1$}};
\node[right] at (7,1) {\textcolor{red}{\small $c_2$}};
\node[above] at (0.5,2.5) {\textcolor{red}{\small $a_4$}};
\node[above] at (1.5,3.5) {\textcolor{red}{\small $a_5$}};
\node[above] at (2.5,3.5) {\textcolor{red}{\small $a_3$}};
\node[above] at (3.5,2.5) {\textcolor{red}{\small $a_2$}};
\node[above] at (4.5,2.5) {\textcolor{red}{\small $a_6$}};
\node[above] at (5.5,2.5) {\textcolor{red}{\small $a_1$}};
\node[below] at (0.5,0.5) {\textcolor{red}{\small $a_4$}};
\node[below] at (1.5,1.5) {\textcolor{red}{\small $a_5$}};
\node[below] at (2.5,1.5) {\textcolor{red}{\small $a_3$}};
\node[below] at (3.5,0.5) {\textcolor{red}{\small $a_2$}};
\node[below] at (4.5,0.5) {\textcolor{red}{\small $a_6$}};
\node[below] at (5.5,0.5) {\textcolor{red}{\small $a_1$}};
\node [below] at (3,-1) {(b)};
		\end{tikzpicture}
  \caption{  (a) A down-right path $\hP$ is depicted (the thick path) in the strip $\mathbb{S}_6$.
 (b) Associated to $\hP$, the two-layer graph $\PiP$ is depicted.  %The down-right path $\hP$ has been rotated counter-clockwise by $\pi/4$ to become the upper layer, and then translated downwards by $\sqrt{2}$ to become the lower layer. 
  The solid edges are labelled by numbers in red, which are the same as those labelling edges of $\hP$ in the strip. 
  The left and right boundary arcs are labelled respectively by $c_1$ and $c_2$ in red.
A two-layer configuration is an assignment of numbers $\la_i^{(j)}\in\ZZ$ to  the vertices $\p_i^{(j)}$ of $\PiP$, for $i=1,2$ and $j=0,\dots,N$.}
  \label{fig:two-path diagram}
\end{figure}

Now we define the Gibbs measure on the set of two-layer configurations on the two-layer graph $\gp$.
\begin{definition}[Two-layer geometric Gibbs measure]\label{defn: Gibbs measures LPP}
For $x,y\in \ZZ$, the weights of solid, dashed and arced edges are
\begin{subequations}
\label{eq:weightsLPP}
\be \label{eq:weight geometric LPP boundary}
\wt_{\lpp}\lb\raisebox{-5.5pt}{\arcleftthick}\rb=c_1^{x-y},\quad \wt_{\lpp}\lb\raisebox{-5.5pt}{\arcrightthick}\rb=c_2^{x-y},\ee
%\wt_{\lpp}\lb\arcleftthick\rb=c_1^{x-y},\quad \wt_{\lpp}\lb\arcrightthick\rb=c_2^{x-y},\ee
\be  \label{eq:weight geometric LPP edges}
\wt_{\lpp}\lb\bulkright\rb=%a^{x-y}\one_{x\geq y},\quad 
\wt_{\lpp}\lb\bulkleft\rb=a^{x-y}\one_{x \geq y},\ee
\be  \label{eq:weight geometric LPP dotted edges}
\wt_{\lpp}\lb\bulkrightdotted\rb=%\one_{x\geq y},\quad 
\wt_{\lpp}\lb\bulkleftdotted\rb=\one_{x \geq y}.\ee
\end{subequations}
We define the weight $\wt_{\lpp}^{\PiP}(\bl)$ of a two-layer configuration $\bl$ associated to a two-layer graph $\PiP$ to be the product of the above weights over all labelled solid edges, dotted edges and arcs in $\PiP$. In what follows it will also be convenient to be able to write $\wt_{\lpp}$ of a configuration drawn on a sub-graph of $\PiP$ to denote the weight of that configuration (thus extending the notation in \eqref{eq:weight geometric LPP boundary}-\eqref{eq:weight geometric LPP dotted edges}. For example, in \eqref{eq:diagram Cauchy LPP} below,
\be\label{eq:weight subgraph example}
\wt_{\lpp}\lb\raisebox{-35pt}{\begin{tikzpicture}
        \draw[thick] (-0.6,0.6)--(0,0)--(0.6,0.6);
        \draw[thick] (-0.6,-0.6)--(0,-1.2)--(0.6,-0.6);
        \draw[dashed] (-0.6,-0.6)--(0,0)--(0.6,-0.6);
        \node[above right] at (0.5,0.55) {\small $\mu_1$};
        \node[above right] at (0.5,-0.65) {\small  $\mu_2$};
        \node[above left] at (-0.5,0.55) {\small   $\la_1$};
        \node[above left] at (-0.5,-0.65) {\small  $\la_2$};
        \node[above] at (0.3,0.3) {\textcolor{red}{\small   $b$}};
        \node[above] at (0.3,-0.9) {\textcolor{red}{\small  $b$}};
        \node[above] at (-0.3,0.3) {\textcolor{red}{\small  $a$}};
        \node[above] at (-0.3,-0.9) {\textcolor{red}{\small  $a$}};
        \node[below] at (0,0) {\small  $\k_1$};
        \node[below] at (0,-1.2) {\small  $\k_2$};
\end{tikzpicture}}\rb = a^{\la_1-\k_1}\one_{\la_1\geq \k_1}a^{\la_2-\k_2}\one_{\la_2\geq \k_2}
b^{\mu_1-\k_1}\one_{\mu_1\geq \k_1}b^{\mu_2-\k_2}\one_{\mu_2\geq \k_2}\one_{\k_1\geq \la_2}\one_{\k_1\geq \mu_2}
\ee 
Notice that $\wt_{\lpp}^{\PiP}(\bl)$ is always positive and translation invariant, in the sense that, for all $x\in\ZZ$,

\be\label{eq:translation invariance two layer LPP}\wt_{\lpp}^{\PiP}(\bl)=\wt_{\lpp}^{\PiP}(\bl+x),\ee
where we write $\bl+x=(\la_1^{(0)}+x,\dots,\la_1^{(N)}+x,\la_2^{(0)}+x,\dots,\la_2^{(N)}+x)$.
\end{definition}

We will view the weight $\wt_{\lpp}^{\PiP}(\bl)$ as a measure on 
$\{\bl\in \ZZ^{2N+2}\}$. In fact, using the notation $\Sign_2$ that will be introduced momentarily in Section \ref{subsec:Skew Schur functions indexed by signatures LPP}, the measure is actually supported on $\Sign_2^{N+1}$ since each $\la^{(j)}\in \Sign_2$ has $\la^{(j)}_1\geq \la^{(j)}_2$. Due to the translation invariance \eqref{eq:translation invariance two layer LPP}, $\wt_{\lpp}^{\PiP}$ must have infinite mass. 
However, as we will see in Proposition \ref{prop:finiteness of weights LPP}, the measure of $\bl$ with fixed $\la_1^{(0)}$ (or any other fixed coordinate $\la_i^{(j)}$) is finite under the assumption $c_1c_2<1$. This will be key in turning these measures into probability measures. 
 
\subsection{Skew Schur functions indexed by signatures}\label{subsec:Skew Schur functions indexed by signatures LPP} 
Our construction of Markov dynamics on two-layer Gibbs measures (in Section \ref{subsec:local Markov operators LPP}) will be based on certain summation identities that we discuss now that are similar to Cauchy and Littlewood identities for Schur functions in symmetric function theory.

Fix some positive integer $n$. A signature of length $n$ is a nonincreasing  sequence of integers $\lambda=(\lambda_1\geq \dots\geq \lambda_n)$ where $\lambda_i\in \mathbb Z$ (note that this $\lambda$ should not be confounded with $\bl$). Let us denote by $\Sign_n$ the set of all signatures of length $n$. For two signatures $\mu, \lambda\in \Sign_n$, we will say that $\mu$ interlaces with 
$\lambda$, denoted $\mu\preceq \lambda$, if $\lambda_i\geq \mu_i\geq \lambda_{i+1}\geq \mu_{i+1}$ for all $2\leq i\leq n$. We will also use the notation $\vert \lambda/\mu\vert =\sum_{i=1}^n(\lambda_i-\mu_i)$. For $\mu, \lambda\in \Sign_n$ we define a polynomial in the variable  $a\in \mathbb C$ by 
$$s_{\lambda/\mu}(a) = \mathds{1}_{\mu\preceq\lambda}a^{\vert \lambda/\mu\vert}.$$ 
More generally, we define  multivariate polynomials $s_{\lambda/\mu}(a_1,\dots, a_k) $ by the branching rule 
\begin{equation}
   s_{\lambda/\mu}(a_1,\dots, a_k) = \sum_{\nu\in \Sign_n} s_{\lambda/\nu}(a_1,\dots a_{k-1})s_{\nu/\mu}(a_k).
   \label{eq:branchingrule}
\end{equation} 
When the signatures $\lambda$ and $\mu$ are nonnegative, $s_{\lambda/\mu}$ are skew Schur polynomials, which are symmetric polynomials. We refer to \cite{macdonald1995symmetric} for background on symmetric functions and Schur polynomials.  

We have the following skew Cauchy type identity: 
\begin{proposition}[Cauchy identity]\label{prop:Cauchy identity LPP}
Fix $n,k,\ell\in \ZZ_{\geq 1}$. For any complex  $\aa=(a_1, \dots, a_k)$ and $\bb=(b_1, \dots, b_{\ell})$, such that $\vert a_ib_j \vert<1$ for all $1\leq i\leq k$, $1\leq j\leq \ell$, we have for all $\lambda,\mu\in \Sign_n$
\begin{equation}
\sum_{\kappa\in \Sign_n} s_{\lambda/\kappa}(\aa)s_{\mu/\kappa}(\bb)
= 
\sum_{\pi\in \Sign_n} s_{\pi/\lambda}(\bb)s_{\pi/\mu}(\aa),
\label{eq:skewCauchySchur}
\end{equation} with both sums finite. When $k=\ell=1$, $n=2$, this says that for all $a,b\in \CC$ with $|ab|<1$ and $\la,\mu\in\Sign_2$ 
\be \label{eq:diagram Cauchy LPP}\sum_{\k_1,\k_2\in \ZZ}\wt_{\lpp}\lb
\raisebox{-35pt}{\begin{tikzpicture}
        \draw[thick] (-0.6,0.6)--(0,0)--(0.6,0.6);
        \draw[thick] (-0.6,-0.6)--(0,-1.2)--(0.6,-0.6);
        \draw[dashed] (-0.6,-0.6)--(0,0)--(0.6,-0.6);
        \node[above right] at (0.5,0.55) {\small $\mu_1$};
        \node[above right] at (0.5,-0.65) {\small  $\mu_2$};
        \node[above left] at (-0.5,0.55) {\small   $\la_1$};
        \node[above left] at (-0.5,-0.65) {\small  $\la_2$};
        \node[above] at (0.3,0.3) {\textcolor{red}{\small   $b$}};
        \node[above] at (0.3,-0.9) {\textcolor{red}{\small  $b$}};
        \node[above] at (-0.3,0.3) {\textcolor{red}{\small  $a$}};
        \node[above] at (-0.3,-0.9) {\textcolor{red}{\small  $a$}};
        \node[below] at (0,0) {\small  $\k_1$};
        \node[below] at (0,-1.2) {\small  $\k_2$};
\end{tikzpicture}}\rb
=\sum_{\pi_1,\pi_2\in \ZZ}\wt_{\lpp}\lb\raisebox{-35pt}{\begin{tikzpicture}
    \draw[thick] (-0.6,0.6)--(0,1.2)--(0.6,0.6);
    \draw[dashed] (-0.6,0.6)--(0,0)--(0.6,0.6);
    \draw[thick] (-0.6,-0.6)--(0,0)--(0.6,-0.6);
    \node[below right] at (0.5,0.65) {\small $\mu_1$};
    \node[below right] at (0.5,-0.55) {\small   $\mu_2$};
    \node[below left] at (-0.45,0.66) {\small  $\la_1$};
    \node[below left] at (-0.45,-0.55) {\small  $\la_2$};
    \node[above] at (0,0) {\small  $\pi_2$};
    \node[above] at (0,1.2) {\small  $\pi_1$};
    \node[above] at (0.35,-0.35){\textcolor{red}{\small  $a$}};
    \node[above] at (0.35,0.85){\textcolor{red}{\small  $a$}};
    \node[above] at (-0.35,-0.35){\textcolor{red}{\small  $b$}};
    \node[above] at (-0.35,0.85){\textcolor{red}{\small  $b$}};
\end{tikzpicture}}\rb. 
\ee
\end{proposition}
\begin{proof}
In view of \eqref{eq:branchingrule}, it suffices to prove the  $k=\ell=1$, $n$ general case (the only one we actually use):
$$
\sum_{\kappa\in \Sign_n: \kappa\preceq \lambda, \kappa\preceq \mu} \prod_{j=1}^n a^{\lambda_j-\kappa_j}b^{\mu_j-\kappa_j}
= 
\sum_{\pi\in \Sign_n: \lambda\preceq \pi, \mu\preceq \pi} \prod_{j=1}^n b^{\pi_j-\lambda_j}a^{\pi_j-\mu_j}.
$$
Both sums are geometric sums that converge when $\vert ab\vert <1$. The equality follows from the change of variables 
$$ \pi_j = -\kappa_{j-1} + \max\lbrace \lambda_j, \mu_j\rbrace +\min\lbrace \lambda_{j-1}, \mu_{j-1} \rbrace,$$
for all $1\leq j\leq n$, where indices are modulo $n$, i.e. $\kappa_0=\kappa_n, \lambda_0=\lambda_n, \mu_0=\mu_n$. This concludes the proof. 
\end{proof}

For $\la\in\Sign_n$ and $c\in\mathbb{C}$ we define the monomial 
$$\tau_\la(c):=c^{\sum_{j=1}^{n}(-1)^{j-1}\la_j}=c^{\la_1-\la_2+\la_3-\la_4+\dots}.$$
 
\begin{proposition}[Littlewood identity]\label{prop:Littlewood identity LPP}
Fix $n\in 2\ZZ_{\geq 1}$. For any $c\in \mathbb C$ and any complex $\aa=(a_1, \dots, a_k)$ such that $\vert a_ia_j \vert<1$ for all $1\leq i<j\leq k$
and $\vert a_i c \vert<1$ for all $1\leq i\leq k$, we have for all  $\k\in \Sign_n$
\begin{equation}
\sum_{\la\in \Sign_n}\tau_{\la}(c) s_{\k/\la}(\aa)
=
\sum_{\pi\in \Sign_n}\tau_{\pi}(c) s_{\pi/\k}(\aa),
    \label{eq:skeLittlewoodidentity}
\end{equation}
with both sums are finite. When $k=1$,  $n=2$, this says that  for all $c,a\in \CC$ with $|ca|<1$ and $\k\in\Sign_2$
\be 
    \label{eq:diagram Littlewood LPP} \sum_{\la_1,\la_2\in \ZZ}\wt_{\lpp}\lb\raisebox{-35pt}{\begin{tikzpicture}
    \draw[dashed](0,0.6)--(0.6,0);
    \draw[thick] (0,-0.6)--(0.6,0);
    \draw[thick] (0,0.6)--(0.6,1.2);
    \node[above right] at (0.5,-.1) {\small $\k_2$};
    \node[above right] at (0.5,1.1) {\small $\k_1$};
    \node[below] at (0,-0.6) {\small $\la_2$};
    \node[above] at (0,0.6) {\small $\la_1$};
    \node[above] at (0.3,-0.3) {\textcolor{red}{\small $a$}};
    \node[above] at (0.3,0.9) {\textcolor{red}{\small $a$}};
    \node[left] at (-0.6,0) {\textcolor{red}{\small $c$}};
    \draw[thick] (0,0.6) arc (90:270:0.6);
\end{tikzpicture}}\rb=\sum_{\pi_1,\pi_2\in \ZZ}\wt_{\lpp}\lb\raisebox{-35pt}{\begin{tikzpicture}
    \draw[dashed](0.6,0)--(0,-0.6);
    \draw[thick] (0,0.6)--(0.6,0);
    \draw[thick] (0,-0.6)--(0.6,-1.2);
    \node[below right] at (0.5,-1.1) {\small $\k_2$};
    \node[below right] at (0.5,.1) {\small $\k_1$};
    \node[below] at (0,-0.6) {\small $\pi_2$};
    \node[above] at (0,0.6) {\small $\pi_1$};
    \node[above] at (0.3,0.3) {\textcolor{red}{\small $a$}};
    \node[above] at (0.3,-0.9) {\textcolor{red}{\small $a$}};
    \node[left] at (-0.6,0) {\textcolor{red}{\small $c$}};
    \draw[thick] (0,0.6) arc (90:270:0.6);
\end{tikzpicture}}\rb. \ee
\end{proposition}
\begin{proof}
 Using the branching rule \eqref{eq:branchingrule} and the skew Cauchy identity \eqref{eq:skewCauchySchur}, it is enough to prove the result for $k=1$ (the only case we actually use below):
$$
\sum_{\la\in\Sign_n,\la\preceq\k} c^{\sum_{j=1}^{n}(-1)^{j-1}\la_j} \prod_{j=1}^na^{\k_j-\la_j} =
\sum_{\pi\in\Sign_n,\k\preceq\pi}c^{\sum_{j=1}^{n}(-1)^{j-1}\pi_j} \prod_{j=1}^na^{\pi_j-\k_j} .
$$
Both sums are geometric sums that converge when $\vert ac\vert <1$. The equality follows from change of variables
$$ \pi_j=-\lambda_{j-1}+\kappa_j+\kappa_{j-1}$$
for all $1\leq j\leq n$, where the indices are modulo $n$.
\end{proof}
 
\begin{remark} \label{rem:remarkskewCauchy}
The summation identity \eqref{eq:skewCauchySchur} is similar to the skew-Cauchy identity for Schur functions (see e.g. \cite[Chap I.5 Ex.26]{macdonald1995symmetric}), while \eqref{eq:skeLittlewoodidentity} is similar to skew Littlewood identities \cite[Chap I.5 Ex.27]{macdonald1995symmetric}. In contrast with these more usual skew Cauchy and Littlewood identities that involve sums over integer partitions, our identities in Propositions \ref{prop:Cauchy identity LPP} and \ref{prop:Littlewood identity LPP} involve sums over signatures (not necessarily nonnegative) and do not involve any normalization by the Cauchy kernel. Some generalizations of Proposition \ref{prop:Cauchy identity LPP}  have been considered  for other families of symmetric functions, in particular the Hall-Littlewood functions \cite{bufetov2018hall} and the spin Hall-Littlewood functions \cite{bufetov2019yang}. Although \cite[Eq (3.7)]{bufetov2018hall} can be shown to imply our Proposition \ref{prop:Cauchy identity LPP}, we give another proof that simply matches terms on both sides of the identity.
\end{remark}
 \begin{remark}
 	\label{rem:comparisonwithSchurmeasures}
 Using the notations above, $\wt_{\lpp}^{\PiP}(\bl)$ can be written as a product of skew Schur functions multiplied by boundary weights of the form $c_1^{|\la^{(0)}|_{-}}$ and $c_2^{|\la^{(N)}|_{-}}$ where $|\la|_{-}=\la_1-\la_2$. 
 This structure is similar to the Schur process, originally introduced in \cite{okounkov2003correlation} (as well as the variants considered in \cite{borodin2005eynard,  betea2018free}),  but there are two important differences: (1) $\wt_{\lpp}^{\PiP}(\bl)$ does not define a probability measure, it has infinite mass, and (2) it is a measure on sequences of signatures of fixed length (i.e., length $2$), instead of integer partitions (which have nonnegative coordinates and arbirary lengths). 
 
 The form of boundary weights that we use is essentially the same as in  \cite{borodin2005eynard}  which introduced a variant of the Schur process with one boundary (there, for a partition $\la$ the boundary weight $c_1^{|\la^{(0)}|_{-}}$ is defined by taking $|\la|_{-}=\sum_{i}(-1)^{i+1}\la_i$).
 In contrast, the work \cite{betea2018free} uses a different type of potential at the boundaries, rather of the type $c_3^{|\la^{(0)}|}$ and $c_4^{|\la^{(N)}|}$ with $|\la|=\sum_{i}\lambda_i$, and this allows to construct a Schur process with two boundaries. Further,  \cite[Remark 2.4]{betea2018free} explains that one can mix the two types of boundary condition (i.e., involving $c_1,c_2$ and $c_3,c_4$). Provided $c_3,c_4<1$, the resulting measure on sequences of partitions is finite, and hence can be normalized to be a probability measure. There is an LPP model on a strip connected to such two-boundary Schur processes that involves more complicated and inhomogeneous geometric weight parameters than those in Definition \ref{defn:inhomogeneous geometric LPP and specializations} (see \cite{betea2019new} for some details). In the $c_3,c_4\to 1$ limit, the restriction to the first two layers of this two boundary Schur process formally converge to our Gibbs measure and the inhomogeneous LPP model converges to the homogeneous version from  Definition \ref{defn:inhomogeneous geometric LPP and specializations}. It would be interesting to see if further information can be obtained by use of this relationship and if this provides another route to verify the stationarity of our Gibbs measure under the LPP recurrence relation.
 
 Finally, let us also mention that a two-sided Schur process, which is a probability measure on sequences of signatures (not necessarily nonnegative) of varying length, was considered in \cite{borodin2011schur}, but it does not seem to be related to our Gibbs measures. 
 \end{remark} 
  
\subsection{Local Markov kernels}\label{subsec:local Markov operators LPP}
 In this subsection we will define a two-layer local (in a sense that will be describe below) Markov dynamics on the strip that preserves the two-layer Gibbs measure in Section \ref{subsec:Two-layer Gibbs-like measure}. These are inspired by the push-block dynamics in the full-space \cite{borodin2014anisotropic} and half-space \cite{baik2018pfaffian} and built from the Cauchy (Proposition \ref{prop:Cauchy identity LPP}) and Littlewood (Proposition \ref{prop:Littlewood identity LPP}) identities.

\begin{definition}[Local moves]\label{defn: local moves}
We first define three types of local moves to evolve down-right paths:
\be \label{eq:local move one path bulk}
\begin{tikzpicture}[scale=0.6]
		\draw[dotted] (0,0) -- (1,0)--(1,1)--(0,1)--(0,0);
		\draw[very thick] (0,1) -- (0,0) -- (1,0);
		\end{tikzpicture}
  \quad 
\raisebox{5pt}{\scalebox{1.5}{$\longmapsto$}}
\quad 
\begin{tikzpicture}[scale=0.6]
		\draw[dotted] (0,0) -- (1,0)--(1,1)--(0,1)--(0,0);
		\draw[very thick] (0,1) -- (1,1) -- (1,0);
		\end{tikzpicture}, 
\quad\quad%\ee \be \label{eq:local move one path left boundary}
\begin{tikzpicture}[scale=0.6]
		\draw[dotted] (0,0) -- (0,1)--(-1,0)--(0,0);
            \draw[very thick] (0,0) -- (-1,0);
		\end{tikzpicture}
  \quad 
\raisebox{5pt}{\scalebox{1.5}{$\longmapsto$}}
\quad 
\begin{tikzpicture}[scale=0.6]
		\draw[dotted] (0,0) -- (0,1)--(-1,0)--(0,0);
		\draw[very thick] (0,0) -- (0,1);
		\end{tikzpicture}, 
\quad\quad%\ee \be \label{eq:local move one path right boundary}
\begin{tikzpicture}[scale=0.6]
		\draw[dotted] (0,0) -- (1,0)--(0,-1)--(0,0);
            \draw[very thick] (0,0) -- (0,-1);
		\end{tikzpicture}
  \quad 
\raisebox{5pt}{\scalebox{1.5}{$\longmapsto$}}
\quad 
\begin{tikzpicture}[scale=0.6]
		\draw[dotted] (0,0) -- (1,0)--(0,-1)--(0,0);
		\draw[very thick] (0,0) -- (1,0);
		\end{tikzpicture}. 
\ee 
The solid lines show the path prior to the move and the dotted lines after the move. These moves can be applied anywhere along a down-right path where admissible (i.e., where the down-right path looks like the starting state of the move). The first is a `bulk' move and the second and third are `boundary' moves. 
We will denote a down-right path before and after a local move by $\hP$ and $\htP$. Notice that the vertex set of $\hP$ and $\htP$ differ precisely by one vertex. Call the vertex in this symmetric difference in $\hP$ the `update vertex' and in $\htP$ the `updated vertex'. Similarly, call all vertex touches the edges that are updated between $\hP$ and $\htP$ the `proximate vertices'.
These local moves induce local moves on the associated  two-layer graphs:   
$$\raisebox{-25pt}
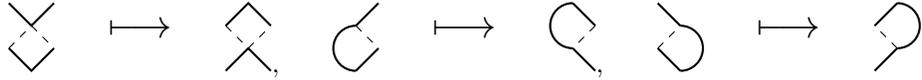
{\begin{tikzpicture}
        \draw[thick] (-0.3,0.3)--(0,0)--(0.3,0.3);
        \draw[thick] (-0.3,-0.3)--(0,-0.6)--(0.3,-0.3);
        \draw[dashed] (-0.3,-0.3)--(0,0)--(0.3,-0.3);
\end{tikzpicture}}
 \quad\quad
\raisebox{-12pt}{\scalebox{1.5}{$\longmapsto$}}
\quad\quad
\raisebox{-25pt}{\begin{tikzpicture}
    \draw[thick] (-0.3,0.3)--(0,0.6)--(0.3,0.3);
    \draw[dashed] (-0.3,0.3)--(0,0)--(0.3,0.3);
    \draw[thick] (-0.3,-0.3)--(0,0)--(0.3,-0.3);
\end{tikzpicture}}\raisebox{-25pt}{,}
\quad\quad%\ee\be\label{eq: local move two path left boundary}
\raisebox{-25pt}{\begin{tikzpicture}
    \draw[dashed] (0,0.3)--(0.3,0);
    \draw[thick] (0,-0.3)--(0.3,0);
    \draw[thick] (0,0.3)--(0.3,0.6);
    \draw[thick] (0,0.3) arc (90:270:0.3);
\end{tikzpicture}}
 \quad\quad
\raisebox{-12pt}{\scalebox{1.5}{$\longmapsto$}}
\quad\quad
\raisebox{-25pt}{\begin{tikzpicture}
    \draw[dashed] (0.3,0)--(0,-0.3);
    \draw[thick] (0,0.3)--(0.3,0);
    \draw[thick] (0,-0.3)--(0.3,-0.6);
    \draw[thick] (0,0.3) arc (90:270:0.3);
\end{tikzpicture}}\raisebox{-25pt}{,}
\quad\quad%\ee\be\label{eq: local move two path right boundary}
\raisebox{-25pt}{\begin{tikzpicture}
    \draw[dashed] (0,0)--(0.3,0.3);
    \draw[thick] (0.3,0.3)--(0,0.6);
    \draw[thick] (0,0)--(0.3,-0.3);
    \draw[thick] (0.3,-0.3) arc (-90:90:0.3);
\end{tikzpicture}}
 \quad\quad
\raisebox{-12pt}{\scalebox{1.5}{$\longmapsto$}}
\quad\quad
\raisebox{-25pt}{\begin{tikzpicture}
    \draw[dashed] (0.3,0)--(0,0.3);
    \draw[thick] (0,0.3)--(0.3,0.6);
    \draw[thick] (0,-0.3)--(0.3,0);
    \draw[thick] (0.3,0) arc (-90:90:0.3);
\end{tikzpicture}}\raisebox{-25pt}{.}$$  
We use the same language of `update', `updated' and `proximate' vertices when talking about configurations of the associated two-layer graph, e.g. if $\p_j$ is the update vertex for a move then $(\la^{(j)}_1,\la^{(j)}_2)$ are the update coordinates in the configuration $\bl$ for the two-layer graph $\PiP$. 
The shift of $\hP$ by $\tau_1$ can be realized as the composition of a sequence of $N$ bulk moves and a left and right boundary move (in any admissible sequence).
Note that after applying a sequence of local moves that transforms a down-right path $\hP$ to its shift by $\tau_1$, the associated two-layer graph cycles back to its starting state. This is key for the stationarity we will prove.
\end{definition}

\begin{definition}[Two-layer geometric Markov dynamics]\label{def:twolayerdynamics}
A two-layer geometric Markov dynamic on the strip is defined as a collection of transition probabilities (i.e., Markov kernels) $\{\U^{\PiP,\mathcal{G}\widetilde{\hP}}\}$ on the state space $\Sign_2^{N+1}$ associated to each pair $(\hP,\widetilde{\hP})$ of down-right paths connected by a single local move, i.e. $\hP\mapsto\widetilde{\hP}$. Here $\U_{\lpp}^{\PiP,\mathcal{G}\widetilde{\hP}}\lb\widetilde{\bl}|\bl\rb$ is the probability of transitioning from configuration $\bl\in\Sign_2^{N+1}$ for the two-layer graph $\PiP$ to the configuration  $\widetilde{\bl}\in\Sign_2^{N+1}$ for the two-layer graph $\mathcal{G}\widetilde{\hP}$. We require that these transition probabilities are local in that they only depend on the coordinates in $\bl$ associated to vertices proximate to the local move; and they act on the identity on all coordinates except for the update/updated vertices. Specifically, if the updated vertex from $\hP$ to $\widetilde{\hP}$ is a bulk vertex $\p_j$ then
$$\U_{\lpp}^{\PiP,\mathcal{G}\widetilde{\hP}}\lb\widetilde{\bl}|\bl\rb = 
\prod_{i\neq j} \one_{\widetilde{\lambda}^{(i)}=\lambda^{(i)}} \,\cdot\,\Ub(\widetilde{\lambda}^{(j)}|\lambda^{(j-1)},\lambda^{(j)},\lambda^{(j+1)};a,b)
$$
where $\Ub(\pi|\la,\k, \mu)$ is a probability distribution in $\pi\in \Sign_2$ given  $\lambda,\k,\mu\in\Sign_2$, and where $a,b$ are the bulk parameters labeling the edges $(\p_{j-1},\p_j)$ and $(\p_j,\p_{j+1})$ respectively.
If the update vertex from $\hP$ to $\widetilde{\hP}$ is a boundary vertex (say $\p_0$, the left-boundary) then
$$
\U_{\lpp}^{\PiP,\mathcal{G}\widetilde{\hP}}\lb\widetilde{\bl}|\bl\rb = 
\prod_{i\neq 0} \one_{\widetilde{\lambda}^{(i)}=\lambda^{(i)}} \,\cdot\,
\Ul(\widetilde{\lambda}^{(0)}|\lambda^{(0)},\lambda^{(1)};c,a)
$$
where $\Ul(\pi|\la,\k)$ is a probability distribution in $\pi\in \Sign_2$ given  $\lambda,\k\in \Sign_2$, and where $c=c_1$ is the boundary parameter labeling the left-boundary arc, and $a$ is the bulk parameters labeling the edge ($\p_0,\p_1)$.
Similarly, if the update vertex was on the right boundary $\U_{\lpp}^{\PiP,\mathcal{G}\widetilde{\hP}}\lb\widetilde{\bl}|\bl\rb$ is defined via $\Ur(\pi|\k,\mu;a,c)$ for the associated bulk and (right) boundary parameters $a$ and $c=c_2$. Our notation here for the local transition probabilities emphasizes their bulk or boundary nature via the superscript which is harvested from \eqref{eq:local move one path bulk}.

In order that $\U_{\lpp}^{\PiP,\mathcal{G}\widetilde{\hP}}$ preserves the two-layer Gibbs measures it will be sufficient that the local transition probabilities $\Ub,\Ul$ and $\Ur$ preserve the local weights (this is why we have included the dependence on the edge parameters that figure into the weights). Specifically we will assume that for all $\pi,\la,\k,\mu\in \Sign_2$ and all $a,b,c_1,c_2>0$ such that $ab,ac_1,ac_2<1$, 

\begin{subequations}
\begin{alignat}{3}
        \sum_{\k\in\Sign_2}\Ub(\pi|\la,\k, \mu;a,b) \wt_{\lpp}\lb\raisebox{-35pt}{\begin{tikzpicture}
        \draw[thick] (-0.6,0.6)--(0,0)--(0.6,0.6);
        \draw[thick] (-0.6,-0.6)--(0,-1.2)--(0.6,-0.6);
        \draw[dashed] (-0.6,-0.6)--(0,0)--(0.6,-0.6);
        \node[above right] at (0.5,0.55) {\small $\mu_1$};
        \node[above right] at (0.5,-0.65) {\small  $\mu_2$};
        \node[above left] at (-0.5,0.55) {\small   $\la_1$};
        \node[above left] at (-0.5,-0.65) {\small  $\la_2$};
        \node[above] at (0.3,0.3) {\textcolor{red}{\small   $b$}};
        \node[above] at (0.3,-0.9) {\textcolor{red}{\small  $b$}};
        \node[above] at (-0.3,0.3) {\textcolor{red}{\small  $a$}};
        \node[above] at (-0.3,-0.9) {\textcolor{red}{\small  $a$}};
        \node[below] at (0,0) {\small  $\k_1$};
        \node[below] at (0,-1.2) {\small  $\k_2$};
\end{tikzpicture}}\rb
&=
\wt_{\lpp}\lb\raisebox{-35pt}{\begin{tikzpicture}
    \draw[thick] (-0.6,0.6)--(0,1.2)--(0.6,0.6);
    \draw[dashed] (-0.6,0.6)--(0,0)--(0.6,0.6);
    \draw[thick] (-0.6,-0.6)--(0,0)--(0.6,-0.6);
    \node[below right] at (0.5,0.65) {\small $\mu_1$};
    \node[below right] at (0.5,-0.55) {\small   $\mu_2$};
    \node[below left] at (-0.45,0.66) {\small  $\la_1$};
    \node[below left] at (-0.45,-0.55) {\small  $\la_2$};
    \node[above] at (0,0) {\small  $\pi_2$};
    \node[above] at (0,1.2) {\small  $\pi_1$};
    \node[above] at (0.35,-0.35){\textcolor{red}{\small  $a$}};
    \node[above] at (0.35,0.85){\textcolor{red}{\small  $a$}};
    \node[above] at (-0.35,-0.35){\textcolor{red}{\small  $b$}};
    \node[above] at (-0.35,0.85){\textcolor{red}{\small  $b$}};
\end{tikzpicture}}\rb,  \label{eq:diagram definition bulk local operator}  \\
   \sum_{\la\in\Sign_2}\Ul(\pi|\la,\k;c_1,a)\wt_{\lpp}\lb\raisebox{-35pt}{\begin{tikzpicture}
    \draw[dashed](0,0.6)--(0.6,0);
    \draw[thick] (0,-0.6)--(0.6,0);
    \draw[thick] (0,0.6)--(0.6,1.2);
    \node[above right] at (0.5,-.1) {\small $\k_2$};
    \node[above right] at (0.5,1.1) {\small $\k_1$};
    \node[below] at (0,-0.6) {\small $\la_2$};
    \node[above] at (0,0.6) {\small $\la_1$};
    \node[above] at (0.3,-0.3) {\textcolor{red}{\small $a$}};
    \node[above] at (0.3,0.9) {\textcolor{red}{\small $a$}};
    \node[left] at (-0.6,0) {\textcolor{red}{\small $c_1$}};
    \draw[thick] (0,0.6) arc (90:270:0.6);
\end{tikzpicture}}\rb
&=
\wt_{\lpp}\lb\raisebox{-35pt}{\begin{tikzpicture}
    \draw[dashed](0.6,0)--(0,-0.6);
    \draw[thick] (0,0.6)--(0.6,0);
    \draw[thick] (0,-0.6)--(0.6,-1.2);
    \node[below right] at (0.5,-1.1) {\small $\k_2$};
    \node[below right] at (0.5,.1) {\small $\k_1$};
    \node[below] at (0,-0.6) {\small $\pi_2$};
    \node[above] at (0,0.6) {\small $\pi_1$};
    \node[above] at (0.3,0.3) {\textcolor{red}{\small $a$}};
    \node[above] at (0.3,-0.9) {\textcolor{red}{\small $a$}};
    \node[left] at (-0.6,0) {\textcolor{red}{\small $c_1$}};
    \draw[thick] (0,0.6) arc (90:270:0.6);
\end{tikzpicture}}\rb,\label{eq:diagram definition left boundary local operator} \\
\sum_{\mu\in\Sign_2}\Ur(\pi|\k,\mu;a,c_2)\wt_{\lpp}\lb\raisebox{-35pt}{\begin{tikzpicture}
    \draw[dashed] (0,0)--(0.6,0.6);
    \draw[thick] (0.6,0.6)--(0,1.2);
    \draw[thick] (0,0)--(0.6,-0.6);
    \node[above left] at (0.1,-0.1) {\small  $\k_2$};
    \node[above left] at (0.1,1) {\small  $\k_1$};
    \node[below ] at (0.7,-0.6) {\small $\mu_2$};
    \node[above ] at (0.7,0.6) {\small  $\mu_1$};
    \node  at (0.35,-0.1) {\textcolor{red}{\small $a$}};
    \node  at (0.35,1.1) {\textcolor{red}{\small $a$}};
    \node[right] at (1.2,0) {\textcolor{red}{\small  $c_2$}};
    \draw[thick] (0.6,-0.6) arc (-90:90:0.6);
\end{tikzpicture}}\rb
&=
\wt_{\lpp}\lb\raisebox{-35pt}{\begin{tikzpicture}
    \draw[dashed](0.6,0)--(0,0.6);
    \draw[thick] (0,0.6)--(0.6,1.2);
    \draw[thick] (0,-0.6)--(0.6,0);
    \node[ left] at (0.1,-0.6) {\small  $\k_2$};
    \node[ left] at (0.1,0.6) {\small $\k_1$};
    \node[below  ] at (0.7,0) {\small $\pi_2$};
    \node[above  ] at (0.7,1.2) {\small  $\pi_1$};
    \node  at (0.25,-0.1) {\textcolor{red}{\small $a$}};
    \node  at (0.25,1.1) {\textcolor{red}{\small $a$}};
    \node[right] at (1.2,0.6) {\textcolor{red}{\small $c_2$}};
    \draw[thick] (0.6,0) arc (-90:90:0.6);
\end{tikzpicture}}\rb. \label{eq:diagram definition right boundary local operator}
\end{alignat}
\end{subequations} 
These equations do not uniquely specify the transition probabilities. Definition \ref{def:push-block} exhibits one solution.

For any down-right paths $\hQ$ sitting above $\hP$ (i.e., achievable via a sequence of local moves), we define $\mathbb{U}(\hP,\hQ)$ the set of vertices between $\hP$ and $\hQ$, including those on $\hQ$ but excluding those on $\hP$. 
We compose the Markov kernels defined above according to local moves at the sequence of vertices in $\mathbb{U}(\hP,\hQ)$ in the lexicographical order of their coordinates $(n,m)\in\mathbb{U}(\hP,\hQ)$ (in fact, any admissible sequence will result in the same transition probability). This defines a transition probability $\U_{\lpp}^{\PiP,\mathcal{G}\hQ}\lb\bl'|\bl\rb$,
where $\bl\in\Sign_2^{N+1}$ represents a configuration on $\PiP$ and $\bl'\in \Sign_2^{N+1}$ a configuration on $\mathcal{G}\hQ$.
\end{definition}

\begin{corollary}[Measure preservation and shift-invariance]\label{Cor:properties local operators LPP}
We have the following properties:
\begin{enumerate}[wide, labelwidth=0pt, labelindent=0pt]
    \item [(1)] For all $\bl'\in\Sign_2^{N+1}$, we have
    \be\label{eq:compatibility of local operator with wt}
    \sum_{\bl\in\Sign_2^{N+1}}\U_{\lpp}^{\PiP,\mathcal{G}\hQ}\lb\bl'|\bl\rb\wt_{\lpp}^\PiP\lb\bl\rb=\wt_{\lpp}^\gq\lb\bl'\rb.
    \ee
    \item[(2)] For all $x\in\ZZ$ and $\bl,\bl'\in \Sign_2^{N+1}$ we have
    \be\label{eq:translation invariance Markov operators LPP}
    \U_{\lpp}^{\PiP,\mathcal{G}\hQ}\lb\bl'|\bl\rb=\U_{\lpp}^{\PiP,\mathcal{G}\hQ}\lb\bl'+x|\bl+x\rb.
    \ee 
\end{enumerate} 
\end{corollary}
\begin{proof}
Since \eqref{eq:compatibility of local operator with wt} is preserved by composition, we only need to show this property for $\hQ=\widetilde{\hP}$ being a local move of $\hP$, which follows from the assumptions \eqref{eq:diagram definition bulk local operator}, \eqref{eq:diagram definition left boundary local operator} and \eqref{eq:diagram definition right boundary local operator}.  
The translation invariance \eqref{eq:translation invariance Markov operators LPP} follows from \eqref{eq:compatibility of local operator with wt} and the weights being translation invariant.
\end{proof}

\begin{definition}[Geometric push-block dynamics]\label{def:push-block}
For the bulk, there is a unique solution $\Ub(\pi|\la,\k,\mu;a,b)$ to \eqref{eq:diagram definition bulk local operator} which  does not depend on $\k$. Denoting this by $\Ub(\pi\vert \lambda, \mu;a,b)$ observe that it is given by the weight on the right-hand side of \eqref{eq:diagram definition bulk local operator} divided by the sum of weights on the left-hand side (without the $\Ub$ factor). Explicitly plugging in the weights, this can be written as
\be  \label{eq:bulk operator LPP}
\Ub(\pi|\la,\mu;a,b)=\frac{b^{\pi_1+\pi_2-\la_1-\la_2}\one_{\pi\succeq\la}a^{\pi_1+\pi_2-\mu_1-\mu_2}\one_{\pi\succeq\mu}}{\sum_{\k}a^{\la_1+\la_2-\k_1-\k_2}\one_{\la\succeq\k}b^{\mu_1+\mu_2-\k_1-\k_2}\one_{\mu\succeq\k}}.
\ee
Due to the Cauchy identity (Proposition \ref{prop:Cauchy identity LPP}) this is a Markov transition matrix, i.e. it is stochastic.
Similarly, for the left boundary, the unique solution to \eqref{eq:diagram definition left boundary local operator} which does not depend on $\la$ is given explicitly by
\be   \label{eq:left boundary operator LPP}
\Ul(\pi|\k;c_1,a):=\frac{\one_{\pi\succeq\k}a^{\pi_1+\pi_2-\k_1-\k_2}c_1^{\pi_1-\pi_2}}{\sum_{\la}\one_{\k\succeq\la}a^{\k_1+\k_2-\la_1-\la_2}c_1^{\la_1-\la_2}}.
\ee 
Owing to the Littlewood identity (Proposition \ref{prop:Littlewood identity LPP}), this is a Markov transition matrix. Likewise, at the right boundary the unique solution to  \eqref{eq:diagram definition right boundary local operator}
which does not depend on $\mu$ is given explicitly by
\be   \label{eq:right boundary operator LPP}
\Ur(\pi|\k;a,c_2):=\frac{1_{\pi\succeq\k}a^{\pi_1+\pi_2-\k_1-\k_2}c_2^{\pi_1-\pi_2}}{\sum_{\mu}1_{\k\succeq\mu}a^{\k_1+\k_2-\mu_1-\mu_2}c_2^{\mu_1-\mu_2}}. 
\ee
\end{definition}

\subsection{First layer marginal}\label{subsec:Marginal distributions LPP}
We first show that the marginal distribution on the first layer of the two-layer geometric push-block Markov dynamics (Definition \ref{def:push-block}) correspond to the recurrence relation defining geometric LPP.
For a configuration $\bl$ on a two-layer graph we will use the shorthand $\bl_i:=(\bl_i^{(0)},\dots,\bl_i^{(N)})$ for $i=1,2$ so that $\bl=(\bl_1,\bl_2)$.
Under the additional restriction that $c_1c_2<1$, we will then take the marginal measure of the two-layer Gibbs measure $\wt_{\lpp}^{\PiP}$ on the first layer $\bl_1$ and multiply it by a finite normalization constant to define a probability measure $\mathrm{P}_{\lpp}^{\hP}$. We then prove the stationarity of $\mathrm{P}_{\lpp}^{\hP}$ under the geometric LPP recurrence relation. 

\begin{lemma}[First layer dynamics match the geometric LPP recurrence relation]\label{lem:marginal operator LPP} The geometric push-block Markov dynamics (Definition \ref{def:push-block} and \eqref{eq:recurrence geometric LPP}) restricted to the configuration on the upper layer of the two-layer graphs are marginally Markov and agree with the dynamics imposed by the geometric LPP recurrence relation (Definition \ref{defn:inhomogeneous geometric LPP and specializations}). In particular, this means that the law of $\pi_1$ under $\Ub(\pi|\la,\mu;a,b)$ only depends on $\la,\mu$ through $\la_1,\mu_1$ and that law, which we write by $\rmUb(\pi_1|\la_1,\mu_1;a,b)$, is given explicitly by
$$\rmUb(\pi_1|\la_1,\mu_1;a,b)=(1-ab)(ab)^{\pi_1-\max(\la_1,\mu_1)}\one_{\pi_1\geq \max(\la_1,\mu_1)},\qquad \textrm{hence}\quad \pi_1=\max(\la_1,\mu_1)+\Geom(ab)$$ 
for some independent geometric random variable $\Geom(ab)$. Similarly, on the left and right boundaries, the law of $\pi$ under $\Ul(\pi|\k;c_1,a)$ and $\Ur(\pi|\k;a,c_2)$ depend on $\k$ only through $\k_1$. Those laws, which we write as $\rmUl(\pi_1|\k_1;c_1,a)$ and $\rmUr(\pi_1|\k_1;a,c_2)$ respectively, are given explicitly by 
$$\rmUl(\pi_1|\k_1;c_1,a)=(1-ac_1)\one_{\pi_1\geq\k_1}(ac_1)^{\pi_1-\k_1},\qquad 
\rmUr(\pi_1|\k_1;a,c_2)=(1-ac_2)\one_{\pi_1\geq\k_1}(ac_2)^{\pi_1-\k_1}.$$
These (respectively) imply that for independent geometric random variables $\Geom(ac_1)$ and $\Geom(ac_2)$,
$$\pi_1=\k_1+\Geom(ac_1)\qquad \textrm{and}\quad \pi_1=\k_1+\Geom(ac_2).$$

Thus, the law of $\bl'_1$ under $\U_{\lpp}^{\PiP,\mathcal{G}\hQ}\lb\bl'|\bl\rb$
only depends on $\bl_1$ and hence is written as $\mathrm{U}_{\lpp}^{\hP,\hQ}\lb\bl'_1|\bl_1\rb$. These transition probabilities define Markov dynamics on the first layer of the two-layer graph which coincide with the recurrence relation \eqref{eq:recurrence geometric LPP} for geometric LPP, and if we initialize that with $G(\p_j)=\la_1^{(j)}$, $0\leq j\leq N$ on $\hP$ then
the probability that $G(\q_j)=\la_1'^{(j)}$, $0\leq j\leq N$ on $\hQ$ is precisely $\mathrm{U}_{\lpp}^{\hP,\hQ}\lb\bl'_1|\bl_1\rb$.
\end{lemma}
\begin{proof}
This type of result for push-block dynamics has appeared in many cases previously, e.g. \cite{borodin2014anisotropic,baik2018pfaffian}. We will prove the claim about the bulk kernel $\Ub$ and left boundary kernel $\Ul$, since the case of right boundary kernel $\Ur$ follows by renaming the variables in $\Ul$. 

By the Cauchy identity (Proposition \ref{prop:Cauchy identity LPP}), the bulk local kernel \eqref{eq:bulk operator LPP} can be written as:
$$\Ub(\pi|\la,\mu;a,b)=\frac{\one_{\max(\la_1,\mu_1)\leq\pi_1}(ab)^{\pi_1}}{\sum_{\max(\la_1,\mu_1)\leq\pi_1}(ab)^{\pi_1}}
\frac{\one_{\max(\la_2,\mu_2)\leq\pi_2\leq\min(\la_1,\mu_1)}(ab)^{\pi_2}}{\sum_{\max(\la_2,\mu_2)\leq\pi_2\leq\min(\la_1,\mu_1)}(ab)^{\pi_2}}.$$
Summing over $\pi_2$, we arrive at the claimed property and formula
$$\sum_{\pi_2}\Ub(\pi|\la,\mu;a,b)=(1-ab)(ab)^{\pi_1-\max(\la_1,\mu_1)}\one_{\pi_1\geq \max(\la_1,\mu_1)}=\rmUb(\pi_1|\la_1,\mu_1;a,b).$$
By the Littlewood identity (Proposition \ref{prop:Littlewood identity LPP}), the left boundary local kernel \eqref{eq:left boundary operator LPP} similarly can be written as:
$$\Ul(\pi|\k;c_1,a)=\frac{\one_{\k_1\leq\pi_1}(ac_1)^{\pi_1}}{\sum_{\k_1\leq\pi_1}(ac_1)^{\pi_1}}
\frac{\one_{\k_2\leq\pi_2\leq\k_1}\lb a/c_1\rb^{\pi_2}}{\sum_{\k_2\leq\pi_2\leq\k_1}\lb a/c_1\rb^{\pi_2}}.$$
Summing over $\pi_2$, we arrive at the claimed property and formula
$$\sum_{\pi_2}\Ul(\pi|\k;c_1,a)=(1-ac_1)(ac_1)^{\pi_1-\k_1}\one_{\pi_1\geq \k_1}=\rmUl(\pi_1|\k_1;c_1,a).$$
The claimed law of $\bl'_1$ under $\U_{\lpp}^{\PiP,\mathcal{G}\hQ}\lb\bl'|\bl\rb$ and relation to the LPP recurrence now follows immediately.
\end{proof}
 
We now show that fixing the value of $\bl$ at some vertex in the two-layer graph results in a finite partition function and hence the marginal law given that conditioning can be normalized to be a probability measure. 
  
\begin{proposition}\label{prop:finiteness of weights LPP}  
 Assume \eqref{eq:condition for inhomogeneous geometric LPP} and the additional condition $c_1c_2<1$. For any down-right path $\hP$ let 
$$ Z_{\lpp} = \sum_{\la_i^{(j)}\in\ZZ, (i,j)\neq (1,0)} \wt_{\lpp}^{\PiP}(\bl), $$
where $\la_1^{(0)}$ is fixed and the summation runs over all $\la_i^{(j)}\in\ZZ$ for $i=1,2$ and $j=0,\dots,N$ with $(i,j)\neq (1,0)$. Then, $Z_{\lpp}$ is finite and does not does not depend on the choice of  $\hP$ or $\la_1^{(0)}$. 
\end{proposition}
\begin{proof}
For any local move $\hP\mapsto\widetilde{\hP}$, we can use the Cauchy identity \eqref{eq:diagram Cauchy LPP} and the Littlewood identity \eqref{eq:diagram Littlewood LPP}, possibly multiple times,  to check that the sum $Z_{\lpp}$ is the same  for $\hP$ and for  $\widetilde{\hP}$. Therefore this quantity is the same for every down-right path $\hP$, and we only need to prove that it is finite for the horizontal path $\hP_h$. Similarly, the independence with respect to the fixed choice of $\la_1^{(0)}$ follows immediately from the fact that all weights in $\wt_{\lpp}^{\PiP}(\bl)$ are invariant by translation. In fact, on this account we could have instead fixed any other variable than $\lambda_1^{(0)}$ without changing the value of $Z_{\lpp}$. In fact, let us instead  fix $\lambda_2^{(0)}=0$ for this proof. Hence, letting $\lambda=(\lambda_1^{(0)}, \lambda_2^{(0)})$ and $\mu=(\lambda_1^{(N)}, \lambda_2^{(N)})$, and using the branching rule \eqref{eq:branchingrule}, we can write $Z_{\lpp}$ explicitly in terms of skew Schur polynomials (extended to signatures as in Section \ref{subsec:Skew Schur functions indexed by signatures LPP}) as
$$Z_{\lpp} = \sum_{\lambda, \mu\in \Sign_2}\mathds{1}_{\la_2=0} c_1^{\lambda_1-\lambda_2} s_{\mu/\la}(a_1, \dots, a_N)c_2^{\mu_1-\mu_2}.$$
Since $\la_2=0$, we may write $c_1^{\la_1-\la_2}=s_{\la/(0,0)}(c_1)$, so that via the branching rule \eqref{eq:branchingrule},
%$$  Z_{\lpp} = \sum_{\lambda, \mu\in \Sign_2}\mathds{1}_{\mu_2=0} c_1^{\lambda_1-\lambda_2} s_{\lambda/\mu}(a_1, \dots, a_N)c_2^{\mu_1-\mu_2}.$$
%When $\mu_2=0$, we may write $c_2^{\mu_1-\mu_2}=s_{\mu/(0,0)}(c_2)$, so that via the branching rule \eqref{eq:branchingrule}, 
$$  Z_{\lpp} = \sum_{\lambda, \mu \in \Sign_2}s_{\la/(0,0)}(c_1) s_{\mu/\la}(a_1, \dots, a_N)c_2^{\mu_1-\mu_2} = \sum_{\mu \in \Sign_2}c_2^{\mu_1-\mu_2} s_{\mu/(0,0)}(a_1, \dots, a_N, c_1).$$
%$$  Z_{\lpp} = \sum_{\lambda, \mu \in \Sign_2}c_1^{\lambda_1-\lambda_2} s_{\lambda/\mu}(a_1, \dots, a_N)s_{\mu/(0,0)}(c_2) = \sum_{\lambda \in \Sign_2}c_1^{\lambda_1-\lambda_2} s_{\lambda/(0,0)}(a_1, \dots, a_N, c_2).$$
When the variables $a_1, \dots, a_N$ satisfy \eqref{eq:condition for inhomogeneous geometric LPP}, all terms in the sum are nonnegative. Thus the sum may be bounded by the corresponding sum over all partitions (i.e. all nonnegative signatures of any length). Denoting by $\mathbb Y$ the set of all partitions, and noticing that $s_{\mu/(0,0)}$ is exactly the usual Schur function $s_{\mu} $, 
$$  Z_{\lpp} \leq \sum_{\mu\in \mathbb Y } c_2^{\mu_1-\mu_2+\mu_3-\dots} s_{\mu}(a_1, \dots, a_N, c_1)= \frac{1}{1-c_1c_2} \prod_{i=1}^N\left( \frac{1}{1-a_ic_1}\frac{1}{1-a_ic_2}\right) \prod_{1\leq i<j\leq n}\frac{1}{1-a_ia_j},$$
%$$  Z_{\lpp} \leq \sum_{\lambda\in \mathbb Y } c_1^{\lambda_1-\lambda_2+\lambda_3-\dots} s_{\lambda}(a_1, \dots, a_N, c_2)= \frac{1}{1-c_1c_2} \prod_{i=1}^N\left( \frac{1}{1-a_ic_1}\frac{1}{1-a_ic_2}\right) \prod_{1\leq i<j\leq n}\frac{1}{1-a_ia_j},$$
where in the last equality we have used a known summation identity, see \cite[Chap I.5, Ex. 7]{macdonald1995symmetric} which is valid as long as the product between any two variables has a modulus in $[0,1)$. Hence, as long as the conditions \eqref{eq:condition for inhomogeneous geometric LPP} and $c_1c_2<1$  are satisfied, the sum $ Z_{\lpp}$ is finite, which concludes the proof. 

Note that instead of explicitly relating to Schur polynomials, we could have proved this result by induction on $N$ and explicit computation (as we will do in proving Proposition \ref{prop:finiteness of weights LG} for the log-gamma polymer).
\end{proof}

We will now prove that the first layer marginal distribution of our two-layer Gibbs measures are stationary measures for the geometric LPP recurrence relation. To state this precisely, we need to introduce a few pieces of notation. 
Recalling the decomposition of $\bl=(\bl_1,\bl_2)$ define first layer marginal weights by 
\be\label{eqn:wtP}
\wt_{\lpp}^{\hP}(\bl_1):=\sum_{\bl_2\in\ZZ^{N+1}}\wt_{\lpp}^{\PiP}(\bl).
\ee
   The translation invariance \eqref{eq:translation invariance two layer LPP} of $\wt_{\lpp}^{\PiP}(\bl)$ implies that  for any $x\in\ZZ$,
$\wt_{\lpp}^{\hP}(\bl_1+x)=\wt_{\lpp}^{\hP}(\bl_1)$.
    When $c_1c_2<1$, by Proposition \ref{prop:finiteness of weights LPP}, for any fixed $\la_1^{(0)}$ we have 
    \be\label{eq:sum one layer being normalizing constant Z LPP}
    Z_{\lpp}=\sum_{\la_1^{(j)}\in\ZZ,1\leq j\leq N}\wt_{\lpp}^{\hP}(\bl_1)<\infty.\ee
    Let us introduce variables that record the first layer configuration centered by $\la^{(0)}_{1}$
    $$L_1{(j)}:=\la_1^{(j)}-\la_1^{(0)}$$ 
    for $1\leq j\leq N$ and the shorthand notation $\bL_1:=(L_1{(1)},\dots,L_1{(N)})$.  
    For $\bL_1\in \ZZ^{N}$ define
\be\label{eqn:Plpph}
\mathrm{P}_{\lpp}^{\hP} \lb \bL_1\rb:=\frac{1}{Z_{\lpp}}\wt_{\lpp}^\hP\lb\bl_1\rb.
\ee
    Due to translation invariance and finiteness of the normalizing constant \eqref{eq:sum one layer being normalizing constant Z LPP}, this 
    is a probability measure. We will show that this is the stationary measure for the geometric LPP recurrence relation. 
    
    Now we need some notation for the corresponding Markov dynamics. Recall the transition probability $\mathrm{U}_{\lpp}^{\hP,\hQ}(\bl_1'|\bl_1)$ for $\bl_1,\bl'_1\in\ZZ^{N+1}$ defined in Lemma \ref{lem:marginal operator LPP} which encodes the dynamics of geometric LPP passage times from the path $\hP$ to the path $\hQ$. We define another transition probability encoding the dynamics of the centered passage times $\bL_1$. For any $\bL_1, \bL_1'\in\ZZ^N$, define 
    \be\label{eq:definition of first marginal operator sum over x}
    \pU_{\lpp}^{\hP,\hQ}(\bL'_1|\bL_1):=\sum_{x\in\ZZ}\mathrm{U}_{\lpp}^{\hP,\hQ}(x,\bL'_1+x|0,\bL_1).
    \ee 
    Owing to the translation invariance of the dynamics defined by $\mathrm{U}_{\lpp}^{\hP,\hQ}(\bl_1'|\bl_1)$, this gives the transition probability from $\bL_1$ to  $\bL_1'$. The  following shows that the weights $\wt_{\lpp}^{\hP}$ from \eqref{eqn:wtP} and the probability measures $\mathrm{P}_{\lpp}^{\hP}$ from \eqref{eqn:Plpph} are stationary with respect to $\mathrm{U}_{\lpp}$ and $\pU_{\lpp}$. 
\begin{theorem}\label{thm:stationary measure LPP before sum over zero mode} 
For any $\bl_1'\in\ZZ^{N+1}$,
\be\label{eq:compatibility of markov operator with one row weitht LPP}
\sum_{\bl_1\in\ZZ^{N+1}}\mathrm{U}_{\lpp}^{\hP,\hQ}(\bl'_1|\bl_1)\wt_{\lpp}^{\hP}(\bl_1)=\wt_{\lpp}^\hQ\lb\bl'_1\rb.
\ee
Assume that $c_1c_2<1$, then for any $\bL_1'\in\ZZ^{N}$,
\be\label{eq:compatibility of probability measure with LPP}
\sum_{\bL_1\in\ZZ^{N}}\pU_{\lpp}^{\hP,\hQ}(\bL'_1|\bL_1)\mathrm{P}_{\lpp}^{\hP} \lb \bL_1\rb=\mathrm{P}_{\lpp}^{\hQ} \lb \bL'_1\rb.
\ee
\end{theorem}

\begin{proof} 
    The first statement \eqref{eq:compatibility of markov operator with one row weitht LPP} follows from summing \eqref{eq:compatibility of local operator with wt} over the second layer $\bl_2$ in conjunction with the definition of $\mathrm{U}_{\lpp}^{\hP,\hQ}(\bl'_1|\bl_1)$ as the marginal of $\U_{\lpp}^{\PiP,\mathcal{G}\hQ}\lb\bl'|\bl\rb$ (Lemma \ref{lem:marginal operator LPP}). 
 The second statement \eqref{eq:compatibility of probability measure with LPP} follows from translation invariance and the definition \eqref{eqn:Plpph}
of $\mathrm{P}_{\lpp}^{\hP} \lb \bL_1\rb$:
\begin{equation*}
    \begin{split}
        \text{LHS}\eqref{eq:compatibility of probability measure with LPP}&=\frac{1}{Z_{\lpp}}\sum_{\bL_1}\sum_x\mathrm{U}_{\lpp}^{\hP,\hQ}\lb x,\bL'_1+x|0,\bL_1\rb\wt^{\hP}_{\lpp}\lb 0,\bL_1\rb\\
        &=\frac{1}{Z_{\lpp}}\sum_{\bL_1}\sum_x\mathrm{U}_{\lpp}^{\hP,\hQ}\lb 0,\bL'_1|-x,\bL_1-x\rb\wt^{\hP}_{\lpp}\lb -x,\bL_1-x\rb\\
        &=\frac{1}{Z_{\lpp}}\wt_{\lpp}^{\hQ}\lb 0,\bL'_1\rb=\text{RHS}\eqref{eq:compatibility of probability measure with LPP}.
    \end{split}
\end{equation*} 
\end{proof}

\subsection{Proof of Theorem \ref{thm:main theorem LPP}}\label{subsec:proof of main theorem LPP} 
So far we have shown that provided $c_1c_2<1$, the stationary measure for the geometric LPP recurrence relation can be realized as a marginal of the two-layer Gibbs measures. In order to go beyond this restriction on $c_1c_2<1$ we will sum out the `zero-mode'. Specifically, we will prove in Proposition \ref{prop:probability measure coincide LPP} that, provided $c_1c_2<1$, for a horizontal path $\hP$ with edge labels $\bb=(b_1,\dots,b_N)$, the probability measure $\mathrm{P}^{\hP}_{\lpp}$ \eqref{eqn:Plpph} defined as a marginal of the two-layer Gibbs measure coincides with $\mathrm{P}^{\bb,c_1,c_2}_{\stat\lpp}$ defined as a marginal of pair of reweighted inhomogeneous random walks (Definition \ref{defn:rescaled random walks LPP inhomogeneous}). We then prove that the probability measure $\mathrm{P}^{\bb,c_1,c_2}_{\stat\lpp}$ is well-defined without the constraint $c_1c_2<1$ and real analytic in these boundary parameters. Combining this with Theorem \ref{thm:stationary measure LPP before sum over zero mode} and the uniqueness of analytic continuation of real analytic functions, we prove Theorem \ref{thm:main theorem LPP}. Let us note that everything done below could be adapted to general down-right paths as opposed to sticking to horizontal paths, as we do now.

\begin{proposition}[Summing out the zero-mode]\label{prop:probability measure coincide LPP}
    Suppose $\hP$ is a horizontal path with labels $\bb=(b_1,\dots,b_N)$. Then the probability measure $\mathrm{P}^{\hP}_{\lpp}$ \eqref{eqn:Plpph}  coincides with $\mathrm{P}^{\bb,c_1,c_2}_{\stat\lpp}$ (Definition \ref{defn:rescaled random walks LPP inhomogeneous}) provided $c_1c_2<1$.  
\end{proposition}
\begin{proof}
    We use the following set of variables: For all $1\leq j\leq N$ let
$$\Delta:=\la_1^{(0)}-\la_2^{(0)},\quad L_1{(j)}:=\la_1^{(j)}-\la_1^{(0)},\quad L_2{(j)}:=\la_2^{(j)}-\la_2^{(0)}$$
and write $L_1{(0)}=L_2{(0)}=0$ and $\bL_i:=(L_i{(1)},\dots,L_i{(N)})$ for $i=1,2$. 
Recall that $\bL=(\bL_1,\bL_2)$.
We define
$$\wth_{\lpp}^{\bb}\lbe \Delta;\bL\rbe:=\wt_{\lpp}^{\gp}\lbe\bl\rbe,$$
which is well-defined due to the translation invariance \eqref{eq:translation invariance two layer LPP} of $\wt_{\lpp}^{\gp}\lbe\bl\rbe$.
When $c_1c_2<1$, 
\be\label{eq:rewriting weights LPP}
\mathrm{P}_{\lpp}^{\hP}(\bL_1)=\frac{\wt_{\lpp}^{\hP}(\bl_1)}{Z_{\lpp}}=\frac{\sum_{\bL_2}\sum_{\Delta}\wth_{\lpp}^{\bb}\lbe\Delta;\bL\rbe}{\sum_{\bL}\sum_{\Delta}\wth_{\lpp}^{\bb}\lbe\Delta;\bL\rbe}.
\ee
By Definition \ref{defn: Gibbs measures LPP}, we explicit evaluate weight $\wth_{\lpp}^{\bb}\lbe \Delta;\bL\rbe$ as
\be\label{eq:hat weight LPP}
\begin{split}
    \wth_{\lpp}^{\bb}\lbe \Delta;\bL\rbe=&(c_1c_2)^{\Delta}c_2^{L_1{(N)}-L_2{(N)}}\prod_{j=1}^{N}\one_{L_1{(j-1)}-L_2{(j)}+\Delta\geq 0}\\
    &\times\lb \one_{L_1{(N)}\geq\dots\geq L_1{(0)}=0}\prod_{j=1}^{N}  b_j^{L_1{(j)}-L_1{(j-1)}}\rb\lb\one_{L_2{(N)}\geq\dots\geq L_2{(0)}=0}\prod_{j=1}^{N}b_j^{L_2{(j)}-L_2{(j-1)}}\rb \\
    =&\one_{\Delta\geq\max_{1\leq j\leq N}(L_2{(j)}-L_1{(j-1)})}(c_1c_2)^{\Delta}c_2^{L_1{(N)}-L_2{(N)}}\prod_{j=1}^N(1-b_j)^{-2}\PP^{\bb,\bb}_{\grw}\lbe\bL\rbe,
\end{split}
\ee
where $\PP^{\bb,\bb}_{\grw}$ from Definition \ref{defn:rescaled random walks LPP inhomogeneous} is the law of two independent inhomogeneous geometric random walks.

When $c_1c_2<1$, we sum \eqref{eq:hat weight LPP} over $\Delta\in\ZZ$ (as goes into computing the right-hand side of \eqref{eq:rewriting weights LPP}) and obtain:
\be\label{eq:relation between hat weight to geometric random walk}
(1-c_1c_2)\prod_{j=1}^N(1-b_j)^2\lb\sum_{\Delta}\wth_{\lpp}^{\bb}\lbe \Delta;\bL\rbe\rb 
=V^{c_1,c_2}_{\lpp}(\bL)\PP^{\bb,\bb}_{\grw}\lbe\bL\rbe,\ee
where we recall from \eqref{eq:PstatLPP inhomogeneous} 
$ V^{c_1,c_2}_{\lpp}(\bL)=(c_1c_2)^{\max_{1\leq j\leq N}(L_2{(j)}-L_1{(j-1)})}c_2^{L_1{(N)}-L_2{(N)}}$.
Combining this deduction with \eqref{eq:rewriting weights LPP} we arrive at the matching of $\mathrm{P}^{\hP}_{\lpp}$ with $\mathrm{P}^{\bb,c_1,c_2}_{\stat\lpp}$ from \eqref{eq:PstatLPP inhomogeneous}.
\end{proof}

We now prove that the normalization constant $\mathcal{Z}_{\lpp}$ in the definition \eqref{eq:PstatLPP inhomogeneous} of the reweighted inhomogeneous geometric random walk measure $\mathrm{P}_{\stat\lpp}^{\bb,c_1,c_2}$ is finite only assuming \eqref{eq:condition for inhomogeneous geometric LPP} (without the condition $c_1c_2<1$), as claimed in the statement of the main Theorem \ref{thm:main theorem LPP}.

\begin{proposition}\label{prop:finiteness of weights after summing zero mode LPP}
%Recall
%$$ V^{c_1,c_2}_{\lpp}(\bL):=(c_1c_2)^{\max_{1\leq j\leq N}(L_2{(j)}-L_1{(j-1)})}c_2^{L_1{(N)}-L_2{(N)}}. $$
Recalling $V^{c_1,c_2}_{\lpp}(\bL)$ from \eqref{eq:PstatLPP inhomogeneous}  and assuming \eqref{eq:condition for inhomogeneous geometric LPP},
    $\mathbb{E}^{\bb,\bb}_{\grw}\lbE V^{c_1,c_2}_{\lpp}(\bL)\rbE$ is finite. 
\end{proposition}
\begin{proof}
Continuing with the notation from Proposition \ref{prop:probability measure coincide LPP} and using \eqref{eq:relation between hat weight to geometric random walk},
%    Recall that $\Delta=\l_1^{(0)}-\l_2^{(0)}$, $L_i{(j)}=\l_i^{(j)}-\l_i^{(0)}$ for $i=1,2$ and $j=1,\dots,N$ and $\hP$ is a horizontal path whose edges are labelled $b_1,\dots,b_N$.
when $c_1c_2<1$ we have 
\be\label{eq: in proof of finiteness LPP}\mathbb{E}^{\bb,\bb}_{\grw}\lbE V^{c_1,c_2}_{\lpp}(\bL)\rbE
=(1-c_1c_2)\prod_{j=1}^N(1-b_j)^2\sum_{\Delta,\bL}\wth_{\lpp}^{\bb}\lb \Delta;\bL\rb.\ee
Observe that $\sum_{\Delta,\bL}\wth_{\lpp}^{\bb}\lb \Delta;\bL\rb=Z_{\lpp}$, which is finite by Proposition \ref{prop:finiteness of weights LPP}. Hence \eqref{eq: in proof of finiteness LPP} is finite.

Now consider the case when $c_1c_2\geq 1$. 
Since we are working along a horizontal path, we have that $L_i{(N)}\geq\dots\geq L_i{(0)}=0$ for $i=1,2$. Thus $L_2{(j)}-L_1{(j-1)}\leq L_2{(N)}$ for all $j$ which implies 
$$V^{c_1,c_2}_{\lpp}(\bL)=(c_1c_2)^{\max_{1\leq j\leq N}(L_2{(j)}-L_1{(j-1)})}c_2^{L_1{(N)}-L_2{(N)}}\leq (c_1c_2)^{L_2{(N)}}c_2^{L_1{(N)}-L_2{(N)}}=c_2^{L_1{(N)}}c_1^{L_2{(N)}}.$$
(Note that this inequality has used here the fact that $c_1c_2\geq 1$ implies $(c_1c_2)^x\leq(c_1c_2)^y$ for $x\leq y$. This is why we separately addressed the $c_1c_2<1$ case.)
Therefore we have
$$\mathbb{E}^{\bb,\bb}_{\grw}\lbE V^{c_1,c_2}_{\lpp}(\bL)\rbE 
    \leq \mathbb E^{\bb,\bb}_{\grw}\left[c_2^{L_1{(N)}}c_1^{L_2{(N)}}\right] 
    =\prod_{j=1}^N \mathbb E^{\bb,\bb}_{\grw}\left[ c_2^{L_1{(j)}-L_1{(j-1)}}\right]\prod_{j=1}^N\mathbb E^{\bb,\bb}_{\grw}\left[ c_1^{L_2{(j)}-L_2{(j-1)}}\right].$$
%\begin{equation*}\begin{split}\mathbb{E}^{\bb,\bb}_{\grw}\lbE V^{c_1,c_2}_{\lpp}(\bL_1,\bL_2)\rbE &\leq \mathbb E^{\bb,\bb}_{\grw}\left[c_2^{L_1^{(N)}}c_1^{L_2^{(N)}}\right]\\&=\prod_{j=1}^N \mathbb E^{\bb,\bb}_{\grw}\left[ c_2^{L_1^{(j)}-L_1^{(j-1)}}\right]\prod_{j=1}^N\mathbb E^{\bb,\bb}_{\grw}\left[ c_1^{L_2^{(j)}-L_2^{(j-1)}}\right].\end{split} \end{equation*} 
The expectations on the right-hand side are, however, finite since $c_1a_j<1$ and $c_2a_j<1$ for $1\leq j\leq N$.
\end{proof}

We are now positioned to complete the proof of Theorem \ref{thm:main theorem LPP} using real analytic continuation.

 \begin{proof}[Proof of Theorem \ref{thm:main theorem LPP}]
We recall that  $\pU_{\lpp}^{\hP,\hQ}\lb\bL_1'|\bL_1\rb$ \eqref{eq:definition of first marginal operator sum over x} gives the transition probability from $\bL_1$ to $\bL_1'$ for the geometric LPP recurrence (i.e., if started with initial condition $\bL_1$ along  $\hP$, this is the probability that the recurrence produces $\bL_1'$ for the last passage times along $\hQ$ centered by the value on the left-boundary). Since we are presently only going to consider horizontal paths, we introduce a slight overload of our notation and write $\pU_{\lpp}^{\bb,c_1,c_2}:=\pU_{\lpp}^{\hP,\tau_1\hP}$, where $\hP$ is a horizontal path with labels $\bb=(b_1,\dots,b_N)$. 
%and $\tau_1\bb=(b_2,\dots,b_N,b_1)$ are labels on the translated path $\tau_1\hP$. 
Of course our notation $\pU_{\lpp}^{\hP,\tau_1\hP}$ hid the implicit dependence on the edge and boundary parameters which we have now made more explicit in this special horizontal case.
 
The stationarity we aim to prove in this theorem can be rewritten as:
For any $\bL'_1\in\ZZ^N$, we have
\be\label{eq:compatibility with probability measure after analytic continuatio LPP}
    \sum_{ \bL_1\in\ZZ^N }\pU_{\lpp}^{\bb,c_1,c_2}\lb \bL'_1|\bL_1\rb\mathrm{P}^{\bb,c_1,c_2}_{\stat\lpp}(\bL_1)=\mathrm{P}^{\tau_1\bb,c_1,c_2}_{\stat\lpp}(\bL_1').\ee
When $c_1c_2<1$, taking $\hQ=\tau_1\hP$ in Theorem \ref{thm:stationary measure LPP before sum over zero mode} and using the matching from Proposition \ref{prop:probability measure coincide LPP}, implies \eqref{eq:compatibility with probability measure after analytic continuatio LPP}. 
To prove \eqref{eq:compatibility with probability measure after analytic continuatio LPP} holds without the constraint $c_1c_2<1$, we prove that both sides (after multiplying by a common denominator) are real analytic functions, and then argue by the uniqueness of analytic continuation.
    
Observe that \eqref{eq:compatibility with probability measure after analytic continuatio LPP} is equivalent with:
    \be\label{eq:in the proof of stationarity after analytic continuation LPP}
    \sum_{ \bL_1 }\pU_{\lpp}^{\bb,c_1,c_2}\lb \bL'_1|\bL_1\rb\lb\sum_{\bL_2 }V^{c_1,c_2}_{\lpp}(\bL)\frac{\PP^{\bb,\bb}_{\grw}\lbe\bL\rbe}{\prod_{j=1}^N(1-b_j)^2}\rb=\sum_{\bL'_2 }V^{c_1,c_2}_{\lpp}(\bL')\frac{\PP^{\bb,\bb}_{\grw}\lbe\bL'\rbe}{\prod_{j=1}^N(1-b_j)^2}.
    \ee  
We fix all other parameters $c_2,a_1,\dots,a_N$ and variables $\bL'_1$, and regard $c_1$ as the only variable. Then we have currently shown that  \eqref{eq:in the proof of stationarity after analytic continuation LPP} holds on the interval $0<c_1<1/\max(c_2,a_1,\dots,a_N)$. We want to prove that it holds on the larger interval $0<c_1<1/\max(a_1,\dots,a_N)$. 

 We observe that since $\max_{1\leq j\leq N}(L_2{(j)}-L_1{(j-1)})\geq L_2{(1)}\geq 0$,  $V^{c_1,c_2}_{\lpp}(\bL)$ \eqref{eq:PstatLPP inhomogeneous} is a positive power of $c_1$. Hence the parenthesis in the LHS of \eqref{eq:in the proof of stationarity after analytic continuation LPP}, as well as the RHS of the same equation are both power series of $c_1$ with nonnegative coefficients that depend on $\bL_1$ (and respectively $\bL'_1$). The kernel $\pU_{\lpp}^{\bb,c_1,c_2}$ on the left-hand side is the result of composing  $N+1$ local Markov kernels $\U^{\rerotateangle}_{\lpp}\circ\U^{\rellcorner}_{\lpp}\circ\dots\circ\U^{\rellcorner}_{\lpp}\circ\U^{\reacuteangle}_{\lpp}$ (recall from Lemma \ref{lem:marginal operator LPP}) followed by summing over translations as in \eqref{eq:definition of first marginal operator sum over x}. Only the left boundary kernel $\U^{\reacuteangle}_{\lpp}$ depends on $c_1$ and the other kernels are constant with respect to $c_1$. The entries of $\U^{\reacuteangle}_{\lpp}$ are equal to $(1-a_1c_1)$ times a power of $a_1c_1$. Therefore \eqref{eq:in the proof of stationarity after analytic continuation LPP} is of the form $(1-a_1c_1)f(c_1)=g(c_1)$, where $f$ and $g$ are power series with nonnegative coefficients. We know that \eqref{eq:in the proof of stationarity after analytic continuation LPP} holds on the smaller interval $0<c_1<1/\max(c_2,a_1,\dots,a_N)$ and (by Proposition \ref{prop:finiteness of weights after summing zero mode LPP}) that both sides are finite on the larger interval $0<c_1<1/\max(a_1,\dots,a_N)$. By the uniqueness of analytic continuation of real analytic functions (see, e.g. \cite[Corollary 1.2.6]{krantz2002primer}), \eqref{eq:in the proof of stationarity after analytic continuation LPP} holds on the larger interval. 
Since \eqref{eq:in the proof of stationarity after analytic continuation LPP} is equivalent to \eqref{eq:compatibility with probability measure after analytic continuatio LPP} we conclude the  stationarity in Theorem \ref{thm:main theorem LPP}.

Finally we will prove the uniqueness of the stationary measure in the homogeneous $\bb=\aa=(a,\ldots,a)$ case. Observe that by virtue of the variational formula \eqref{eqn:gvar} for LPP, any stationary measure along a horizontal path must be supported on weakly increasing natural numbers. Consider initial conditions whereby $G(0,0)=0$ and $G(j,0)-G(j-1,0)=x_j$ for $1\leq j\leq N$ with $x_1,\ldots,x_N\in \ZZ_{\geq 0}$.  By \cite[Theorem 5.5]{benaim2022markov} it suffices to show that for any $x'_1,\ldots,x'_N\in \ZZ_{\geq 0}$, there is a strictly positive probability that  $G(j+1,1)-G(j,1)=x'_j$ for $1\leq j\leq N$. (This implies that the Markov kernel $\pU_{\lpp}^{\aa,c_1,c_2}$ has a unique stationary measure.) It is easy to demonstrate the strictly positive transition probability. Let $r=1+\sum_{\ell=1}^N x_\ell$. There is a strictly positive probability that the recursion results in updating $G(1,1)=r$ and then sequentially $G(j+1,1)=G(j,1)+x'_j$ for $1\leq j\leq N$. Indeed, this is true because $G(1,1)\geq G(1,0)$, 
$G(j+1,1)\geq\max(G(j,1),G(j+1,0))$ for $1\leq j\leq N-1$
and $G(N+1,1)\geq G(N,1)$,
hence this update has strictly positive probability.
We conclude the proof of uniqueness.
Ergodicity follows from the fact that unique (more generally, extremal) invariant probability measures are ergodic, see for example \cite[Proposition 4.29]{benaim2022markov}.
\end{proof}

\subsection{Matrix product ansatz}
\label{sec:MPA LPP}
When $c_1c_2<1$, the probability measure $\mathrm{P}_{\lpp}^{\hP} \lb \bL_1\rb$ can be written as a matrix product. Recall the definition of $\mathrm{P}_{\lpp}^{\hP} \lb \bL_1\rb$ in terms of the two-layer Gibbs measure in  \eqref{eqn:Plpph}, \eqref{eqn:wtP}. In order to express $\wt_{\lpp}^{\PiP}(\bl)$ as a product of matrices, we define a change of variables  $x_i=\vert \lambda_1^{(i)}-\la_1^{(i-1)}\vert$ (first layer increments) for $1\leq i\leq N$ and $n_i=\la_1^{(i)}-\la_2^{(i)}$ (gaps between layers) for $0\leq i\leq N$.

\begin{definition}
Let us define, for each $x\in \mathbb Z_{\geq 0}$ and $a\in (0,1)$,  infinite matrices  $\mathbf M_x^{\rightarrow}[a], \mathbf M_x^{\downarrow}[a] \in \mathbb R^{\mathbb Z_{\geq 0} \times \mathbb Z_{\geq 0}}$ with matrix elements given by, for $n,n'\in \mathbb Z_{\geq 0}$, 
\begin{equation}
	\mathbf M_x^{\rightarrow}[a](n,n') =\wt_{\lpp}\left.\lb\raisebox{-30pt}{\begin{tikzpicture}
			\draw[thick] (0,0)--(0.6,0.6);
			\draw[thick] (0,-1.2)--(0.6,-0.6);
			\draw[dashed] (0,0)--(0.6,-0.6);
			\node[right] at (0.5,0.55) {\small $\la'_1$};
			\node[right] at (0.6,-0.6) {\small  $\la'_2$};
			\node[left] at (0,0) {\small   $\la_1$};
			\node[left] at (0,-1.2) {\small  $\la_2$};
			\node[above] at (0.3,0.3) {\textcolor{red}{\small   $a$}};
			\node[above] at (0.3,-0.9) {\textcolor{red}{\small  $a$}};
	\end{tikzpicture}}\rb\right\vert_{{\footnotesize \begin{matrix}
				\la_1-\la_2=n\\ \la'_1-\la'_2=n'\\ \la'_1-\la_1=x
	\end{matrix}}}  =  a^{2x+n-n'}\mathds{1}_{n'\geq x}\mathds{1}_{x+n-n'\geq 0},
	\label{eq:defMatrix M right}
\end{equation} 
\begin{equation}\mathbf M_x^{\downarrow}[a](n,n') =\wt_{\lpp}\left. \lb\raisebox{-30pt}{\begin{tikzpicture}
			\draw[thick] (-0.6,0.6)--(0,0);
			\draw[thick] (-0.6,-0.6)--(0,-1.2);
			\draw[dashed] (-0.6,-0.6)--(0,0);
			\node[left] at (-0.5,0.55) {\small   $\la_1$};
			\node[left] at (-0.5,-0.65) {\small  $\la_2$};
			\node[above] at (-0.3,0.3) {\textcolor{red}{\small  $a$}};
			\node[above] at (-0.3,-0.9) {\textcolor{red}{\small  $a$}};
			\node[right] at (0,0) {\small  $\la'_1$};
			\node[right] at (0,-1.2) {\small  $\la'_2$};
	\end{tikzpicture}}\rb\right\vert_{{\footnotesize \begin{matrix}
				\la_1-\la_2=n\\ \la'_1-\la'_2=n'\\ \la_1-\la'_1=x
	\end{matrix}}} = a^{2x +n'-n} \mathds{1}_{n\geq x}\mathds{1}_{x+n'-n\geq 0}.
	\label{eq:defMatrix M down}
\end{equation} 
In particular, $\mathbf M_x^{\downarrow}[a]$ is the transpose of $ \mathbf M_x^{\rightarrow}[a]$. Let us also define vectors $\mathbf w, \mathbf v\in \mathbb R^{\mathbb Z_{\geq 0}}$ with coefficients $ \mathbf w(n)=(c_1)^n$ and $ \mathbf v(n)=(c_2)^n$ for $n\in \mathbb Z_{\geq 0}$.
\label{def:matrices and vectors}
\end{definition}
Thus, via the change of variables, for any path $\hP$ with vertices $\left(\p_i\right)_{0\leq i\leq N}$ and edge labels $\mathbf b$, we have
$$ \wt_{\lpp}^{\PiP}(\bl) = \mathbf w(n_0) \left( \prod_{i=1}^N \mathbf M^{\p_i-\p_{i-1}}_{x_i}[b_i](n_{i-1}, n_i)  \right) \mathbf v(n_N),  $$
where $\p_i-\p_{i-1}\in \lbrace \rightarrow, \downarrow\rbrace$, since $\hP$ is a down-right path. 
Thus, we have that 
\begin{equation}
	 \mathrm{P}_{\lpp}^{\hP} \lb \bL_1\rb= \frac{1}{Z_{\rm Geo}} \mathbf w^t \left( \prod_{i=1}^N \mathbf M^{\p_i-\p_{i-1}}_{x_i}[b_i] \right)\mathbf v, 
	 \label{eq:MPAform}
\end{equation}
where matrices in the product are written from left to right and $x_i=\vert  \bL_i-\bL_{i-1}\vert$.

Such a matrix product expression is reminiscent of the matrix product ansatz, introduced in \cite{derrida1993exact}. This technique can indeed be adapted to geometric LPP. 
\begin{proposition}\label{prop:MPA LPP}
Assume that a family of probability distributions $\mathrm{P}^{\hP}$ on $\mathbb Z_{\geq 0}^N$, indexed by down-right paths $\hP=(\mathbf p_0,  \dots, \mathbf p_N)$ with edge labels $\mathbf b$, takes the form 
\begin{equation}
	\mathrm{P}^{\hP} \lb x_1, \dots, x_N\rb= \frac{1}{Z} \overline{\mathbf w}^t \left( \prod_{i=1}^N \overline{\mathbf M}^{\p_i-\p_{i-1}}_{x_i}[b_i] \right)\overline{\mathbf v}, 
	\label{eq:MPAform2}
\end{equation}
where $\overline{\mathbf M}^{\rightarrow}_{x_i}[b_i], \overline{\mathbf M}^{\downarrow}_{x_i}[b_i]$ are matrices in $\mathbb R^{\mathbb Z_{\geq 0} \times \mathbb Z_{\geq 0}}$,  $\overline{\mathbf w}, \overline{\mathbf v}$ are vectors in $\mathbb R^{\mathbb Z_{\geq 0}}$, and $Z$ is a finite constant chosen so that \eqref{eq:MPAform2} is a well-defined probability measure on $\mathbb Z_{\geq 0}^N$. Then, the measure $\mathrm{P}^{\hP}$ is stationary for the geometric LPP dynamics, in the sense of \eqref{eq:compatibility of probability measure with LPP}, if the following commutation relations hold for all $x,y\geq 0$: 
\begin{subequations}
	\label{eq:MPA}
	\begin{align}
		\overline{\mathbf M}^{\rightarrow}_x[a]\overline{\mathbf M}^{\downarrow}_y[b] &= (ab)^{\min\lbrace x,y\rbrace}(1-ab) \sum_{z\geq \max\lbrace 0, y-x\rbrace}	\overline{\mathbf M}^{\downarrow}_z[b]\overline{\mathbf M}^{\rightarrow}_{x-y+z}[a],\label{eq:MPAbulk}\\ 
		\overline{\mathbf w}^t \overline{\mathbf M}^{\downarrow}_x[a] &= (ac_1)^x(1-ac_1) \sum_{z\geq 0} \overline{\mathbf w}^t \overline{\mathbf M}^{\rightarrow}_z[a],\label{eq:MPAleft}\\ 
		\overline{\mathbf M}^{\rightarrow}_x[a]\overline{\mathbf v} &=(ac_2)^x(1-ac_2) \sum_{z\geq 0}\overline{\mathbf M}^{\downarrow}_z[a]\overline{\mathbf v}.\label{eq:MPAright}
	\end{align}
\end{subequations}
\end{proposition}
\begin{proof}
As in \cite{yang2022stationary}, which proves a similar result for the stochastic six-vertex model, it suffices to ensure that \eqref{eq:compatibility of probability measure with LPP} holds when $\mathrm{P}_{\lpp}^{\hP}$ is replaced by $\mathrm{P}^{\hP}$ (thought of as a measure on $\bL_1=(0,x_1,x_1+x_2,\ldots,x_1+\cdots x_N)$) from above and when $\hP \mapsto \hQ$ is one of the elementary moves from Definition \ref{defn: local moves}. In other words, it suffices to check that $\mathrm{P}^{\hP}$ pushes forward  to  $\mathrm{P}^{\hQ}$  under the action of local operators  $\rmUb, \rmUl, \rmUr$ from Lemma \ref{lem:marginal operator LPP}, where $\hQ$ is the corresponding locally updated path. Under the appropriate changes of variables, \eqref{eq:MPAbulk} implies the invariance with respect to $\rmUb$. Indeed, if the local move transforms the point $\mathbf p_i$ on the path $\hP$ to a new point $\mathbf q_i$, the increments $x_i=x$ and $x_{i+1}=y$ may arise on $\hQ$  only if the weight $\o_{\mathbf q_i}=\min\lbrace x, y\rbrace$ (which arises with probability $(b_ib_{i+1})^{\min\lbrace x,y\rbrace}(1-b_ib_{i+1})$), while the increments on $\hP$ before the local move can be anything such that $x_{i+1}-x_{i}=x-y$. Likewise, \eqref{eq:MPAleft} implies invariance with respect to $\rmUl$, and \eqref{eq:MPAright} implies invariance with respect to $\rmUr$.
\end{proof}

It is not a priori clear how to find matrices and vectors satisfying \eqref{eq:MPA}. We will see that our two layer Gibbs measures, which originate in a completely different context (i.e. Gibbsian line ensemble and probability measures based on families of symmetric functions),  provide a representation for the quadratic algebra \eqref{eq:MPA}.  
\begin{proposition}
The matrices $\mathbf M_x^{\rightarrow}[a], \mathbf M_x^{\downarrow}[a]$ and vectors $\mathbf w, \mathbf v$ (Definition \ref{def:matrices and vectors}) satisfy the relations \eqref{eq:MPA}. 
\label{prop:MPArepresentation}
\end{proposition}
\begin{proof}
The proof relies on the Littlewood and Cauchy identities for skew Schur functions, as well as the truncated Toeplitz structure of the matrices $\mathbf M_x^{\rightarrow}[a], \mathbf M_x^{\downarrow}[a]$. Let us first observe that the Cauchy identity \eqref{eq:diagram Cauchy LPP} and the Littlewood identity  \eqref{eq:diagram Littlewood LPP} may be  rewritten as 
\begin{subequations}
	\label{eq:CauchyLittlewoodMatrix}
	\begin{align}
		 \sum_{z\geq \max\lbrace 0,d\rbrace} \mathbf M^{\downarrow}_z[a]\mathbf M_{z-d}^{\rightarrow}[b]  &=  \sum_{z\geq \max\lbrace 0,d\rbrace} \mathbf M^{\rightarrow}_{z-d}[b]\mathbf M_{z}^{\downarrow}[a], \text{ for all }d\in \mathbb Z, 
		\label{eq:skewCauchymatrix} \\ 
			\sum_{z\geq 0} \mathbf w^t \mathbf M^{\rightarrow}_z[a]  &= 	\sum_{z\geq 0} \mathbf w^t \mathbf M^{\downarrow}_z[a], 
		\label{eq:Littlewoodmatrix1}\\
			\sum_{z\geq 0}  \mathbf M^{\downarrow}_z[a]\mathbf v  &= 	\sum_{z\geq 0} \mathbf M^{\rightarrow}_z[a]\mathbf v.  
		\label{eq:Littlewoodmatrix2}
	\end{align}
\end{subequations}
Equations \eqref{eq:CauchyLittlewoodMatrix} are not directly equivalent to \eqref{eq:MPA}, but equations \eqref{eq:MPA} are implied by equations \eqref{eq:CauchyLittlewoodMatrix} if the following extra relations are satisfied:
\begin{subequations}
	\label{eq:extra relations}
	\begin{align}
		\mathbf M^{\rightarrow}_x[a]\mathbf M^{\downarrow}_y[b]  &= (ab)^{\min\lbrace x,y\rbrace}(1-ab)\sum_{z\geq \max\lbrace 0,y-x\rbrace} \mathbf M^{\rightarrow}_{z+x-y}[a]\mathbf M_{z}^{\downarrow}[b],\label{eq:extra relation bulk}\\ 
		\mathbf w^t \mathbf M^{\downarrow}_x[a]&=(ac_1)^x(1-ac_1)\sum_{y\geq 0} \mathbf w^t \mathbf M^{\downarrow}_y[a],
		\label{eq:extrarelation left}\\ 
		 \mathbf M^{\rightarrow}_x[a] \mathbf v&=(ac_2)^x(1-ac_2)\sum_{y\geq 0} \mathbf M^{\rightarrow}_y[a] \mathbf v.
		\label{eq:extrarelation right}
	\end{align}
\end{subequations} 
Hence, it suffices to prove that the matrices $\mathbf M_x^{\rightarrow}[a], \mathbf M_x^{\downarrow}[a]$ and vectors $\mathbf w, \mathbf v$ satisfy \eqref{eq:extra relations}. Those relations could be verified explicitly by plugging the matrix elements and computing. To avoid such computations, we may observe that truncated Toeplitz matrices are  products of Toeplitz matrices, so that  one can write 
\begin{equation}
	\mathbf M_x^{\rightarrow}[a] = \left( \sum_{k=0}^{+\infty} a^k \mathbf S^k \right) a^x \mathbf T^x, \;\; \;\;\; \mathbf M_x^{\downarrow}[a] = a^x \mathbf S^x  \left( \sum_{k=0}^{+\infty} a^k \mathbf T^k \right), 
	\label{eq:Toeplitz}
\end{equation}
where $\mathbf S$ is the lower shift matrix (that is $\mathbf S(n,n')=\mathds 1_{n=n'+1}$)   and $\mathbf T$ is the upper shift matrix (that is $\mathbf T(n,n')= \mathds 1_{n'=n+1}$). Using \eqref{eq:Toeplitz}, the relations \eqref{eq:extrarelation left} and \eqref{eq:extrarelation left} follow immediately from the eigenrelations $\mathbf w\mathbf S=c_1 \mathbf w$ and $\mathbf T\mathbf v=c_2 \mathbf v$. To prove \eqref{eq:extra relation bulk}, we write 
\begin{align*}
	\sum_{z\geq \max\lbrace 0,y-x\rbrace} \mathbf M^{\rightarrow}_{z+x-y}[a]\mathbf M_{z}^{\downarrow}[b] &= \sum_{k\geq 0, \ell\geq 0, z\geq \max\lbrace x, y\rbrace} a^{k+z-y}b^{z-x+\ell} \mathbf S^k\mathbf T^{z-y}\mathbf S^{z-x} \mathbf T^{\ell},\\
	&= \frac{1}{1-ab}\sum_{k\geq 0, \ell\geq 0 } a^{k+(x-y)_+}b^{(y-x)_++\ell} \mathbf S^k\mathbf T^{(x-y)_+}\mathbf S^{(y-x)_+} \mathbf T^{\ell},\\
	&= \frac{(ab)^{-\min\{x,y\}}}{1-ab} \sum_{k\geq 0, \ell\geq 0 } a^{k+x}b^{y+\ell} \mathbf S^k\mathbf T^{x}\mathbf S^{y} \mathbf T^{\ell}\\
	&= \frac{(ab)^{-\min\{x,y\}}}{1-ab} 	\mathbf M^{\rightarrow}_x[a]\mathbf M^{\downarrow}_y[b], 
\end{align*}
where in the first line we have just used  \eqref{eq:Toeplitz}, in the second line we have used $\mathbf T^k\mathbf S^k=\mathbf I$,  summed over $z$, and used the notation $(x-y)_+=\max\lbrace x-y,0\rbrace$,  and in the third line we have used the more general relation $\mathbf T^x\mathbf S^y = \mathbf T^{(x-y)_+} \mathbf S^{(y-x)_+}$. This proves that \eqref{eq:extra relation bulk} holds, which concludes the proof of Proposition \ref{prop:MPArepresentation}. 
\end{proof}
%Equations \eqref{eq:skewCauchymatrix}, \eqref{eq:Littlewoodmatrix1} and \eqref{eq:Littlewoodmatrix2} ensure that the normalization $Z_{\rm Geo}$ in \eqref{eq:MPAform} does not depend on the path $\hP$ (which we already knew from Proposition \ref{prop:finiteness of weights LPP}). 

\section{Stationary measure for log-gamma polymer on a strip}
\label{sec:LG}

The treatment here follows the same steps as that of geometric LPP from Section \ref{sec:LPP}, though at some points we encounter increased complexity (hence the reader is encouraged to first read Section \ref{sec:LPP}). Namely, as state-spaces are now uncountable (products of $\RR$) we work with probability densities instead of probability mass functions; and whereas the proof of Proposition \ref{prop:finiteness of weights LPP} used Schur polynomials, we prove Proposition \ref{prop:finiteness of weights LG} here via induction (a more general result using Whittaker functions may be possible though is not pursued here). The biggest difference, however, comes in the proof of the analytic continuation needed to release the condition $u+v>0$ (or $c_1c_2<1$ in the geometric case). In the geometric case, real analyticity of both sides of \eqref{eq:compatibility with probability measure after analytic continuatio LPP} follows since both sides are power series in $c_1$. The sums in \eqref{eq:compatibility with probability measure after analytic continuatio LPP} are replaced by integrals in \eqref{eq:compatibility with probability measure after analytic continuatio LG} and thus the dependence on $u$ is not as simple. Now, in Section \ref{sec:proof of real analyticity}, we utilize tools from complex analysis (i.e., Morera's theorem) to show that both sides of \eqref{eq:compatibility with probability measure after analytic continuatio LG} are in fact holomorphic on a suitable open set.

\subsection{Log-gamma polymer with inhomogeneous weights} 
\label{subsec:Log-gamma polymer with inhomogeneous weights} 
We define the log-gamma polymer with inhomogeneous weights and then state the main theorem of this section constructing its stationary measure on a horizontal path, which generalizes Theorem \ref{thm:stationary measure log-gamma intro} in the introduction.
\begin{definition}[Inhomogeneous log-gamma polymer]\label{defn:inhomogeneous LG and specializations} 
Let $\a_1,\dots, \a_N\in\RR$ and $u,v\in\RR$ be  such that:
\be\label{eq:condition for inhomogeneous log gamma}
\a_i+\a_j>0,\quad \a_i+u>0,\quad \a_i+v>0,\quad\forall 1\leq i,j\leq N.
\ee 
We will always assume these conditions in this paper as they are necessary for the inverse-gamma random variables below to be defined. We call the $\alpha_i$ `bulk parameters' and the $u,v$ `boundary parameters'.
Define $\a_{j+kN}=\a_j$ for $k\in\mathbb{Z}$ and $1\leq j\leq N$.
We define the inhomogeneous version of log-gamma polymer on a strip by the same recurrence \eqref{eq:recurrence log gamma} and initial condition as the homogeneous model, but now with $\vo_{n,m}\sim\Gav(\a_n+\a_m)$ in the bulk $0\leq m<n<m+N$, $\vo_{m,m}\sim\Gav(\a_m+u)$ on the left boundary and $\vo_{m+N,m}\sim\Gav(\a_m+v)$ on the right boundary.  
We label each edge of the strip by a number, which will be needed later to define Gibbs  measures.
In particular, we label the horizontal edge $(n-1,m)\rightarrow(n,m)$ by $\a_n$ and the vertical edge $(n,m-1)\rightarrow(n,m)$ by $\a_m$. See Figure \ref{fig:strip for N=5} for an illustration in the geometric LPP setting. To transform to the log-gamma setting, each $a_i$ should be replaced by $\a_i$, $c_1$ by $u$ and $c_2$ by $v$.
\end{definition}

Next we introduce the inhomogeneous version of reweighted log-gamma random walk generalizing Definition \ref{def:rescaled random walks}. We then introduce the main theorem of this section that this reweighted random walk is stationary under inhomogeneous log-gamma polymer, which generalizes Theorem \ref{thm:stationary measure log-gamma intro}. 
\begin{definition}[Reweighted inhomogeneous log-gamma random walks]\label{defn:rescaled random walks LG inhomogeneous} 
Assume that $(\alpha_i)_{1\leq i\leq N}\in \mathbb R_{>0}^N$ and $u,v\in \mathbb R$ satisfy \eqref{eq:condition for inhomogeneous log gamma}. Suppose $\hP$ is a horizontal path with labels on edges $\bbb=(\b_1,\dots,\b_N)$ from left to right ($\bbb$ must be a cyclic shift of $\boldsymbol{\a}=(\a_1,\dots,\a_N)$). Consider two independent  random walks $\bL_1=\big(L_1(j)\big)_{1\leq j\leq N}\in \RR^{N}$ and $\bL_2=\big(L_2(j)\big)_{1\leq j\leq N}\in \RR^{N}$ starting from  $L_1{(0)}=L_2{(0)}=0$ with independent increments distributed for $1\leq j\leq N$ as $$L_1{(j)}-L_1{(j-1)}\sim \log(\Ga^{-1}(\b_j)),\quad\textrm{and}\quad L_2{(j)}-L_2{(j-1)}\sim \log(\Ga^{-1}(\b_j)).$$
We use shorthand $\bL=(\bL_1,\bL_2)$ and denote by $\PP^{\bbb,\bbb}_{\lgrw}\lbe\bL\rbe$  the associated probability density (here and below, always against Lebesgue measure), and by $\mathbb{E}^{\bbb,\bbb}_{\lgrw}$ the expectation with respect to this probability measure. 
We define a new probability density $\PP_{\stat\lgg}^{\bbb,u,v}$ by reweighting the density $\PP^{\bbb,\bbb}_{\lgrw}$ as 
 \be \label{eq:PstatLG inhomogeneous}
\PP^{\bbb,u,v}_{\stat\lgg}(\bL):= \frac{V^{u,v}_{\lgg}(\bL)\PP^{\bbb,\bbb}_{\lgrw}\lb\bL\rb}{\hZ^{\bbb,u,v}_{\lgg}} \quad\textrm{where}\quad V_{\lgg}^{u,v}(\bL):=\lb \sum_{j=1}^N e^{L_2{(j)}-L_1{(j-1)}}\rb^{-(u+v)}\!\!\!\!\!\!\!\!\!\!\!\!e^{-v(L_1(N)-L_2(N))},
\ee 
and the normalizing constant  $\hZ^{\bbb,u,v}_{\lgg}:=\mathbb{E}^{\bbb,\bbb}_{\lgrw}\lbE V^{u,v}_{\lgg}(\bL)\rbE$ will be proved to be finite in Proposition \ref{prop:finiteness of weights after summing zero mode LG}.
%We take the marginal on $\bL_1$ of the probability density $\PP^{\bbb,u,v}_{\stat\lgg}$ on $\bL=(\bL_1,\bL_2)$, which defines a probability density $\mathrm{P}^{\bbb,u,v}_{\stat\lgg}$ on $\bL_1$. 
Let $\mathrm{P}^{\bbb,u,v}_{\stat\lgg}(\bL_1)$ denote the marginal density of $\bL_1$ for the probability density $\PP^{\bbb,u,v}_{\stat\lgg}(\bL)$ where $\bL=(\bL_1,\bL_2)$.
\end{definition}

The following is the main theorem in this section.
 \begin{theorem}\label{thm:main theorem LG}  
 Assume that $(\alpha_i)_{1\leq i\leq N}\in \mathbb R_{>0}^N$ and $u,v\in \mathbb R$ satisfy \eqref{eq:condition for inhomogeneous log gamma}. 
Suppose $\bbb=(\b_1,\dots,\b_N)$ are labels on some horizontal path $\hP$, and $\tau_1\bbb=(\b_2,\dots,\b_N,\b_1)$ are labels on the translated path $\tau_1\hP$. 
Consider the inhomogeneous log-gamma polymer model from Definition \ref{defn:inhomogeneous LG and specializations} starting from an initial condition along $\hP$ given by $(h(\p_j))_{0\leq j\leq N}$ where  $(h(\p_j)-h(\p_0))_{1\leq j\leq N}$ is distributed according to the probability measure with density $\mathrm{P}^{\bbb,u,v}_{\stat\lgg}$ with respect to Lebesgue measure on $\RR^{N}$. Then the  density of $(h(\tau_1\p_j)-h(\tau_1\p_0))_{1\leq j\leq N}$ coincides with $\mathrm{P}^{\tau_1\bbb,u,v}_{\stat\lgg}$. 
For the homogeneous (i.e. $\boldsymbol{\a}=(\a,\ldots, \a)$) case, the probability measure with density $\mathrm{P}^{\a,u,v}_{\stat\lgg}$  is the (unique) ergodic stationary measure on the horizontal path $\hP$ for the log-gamma polymer recurrence relation (see Definition \ref{defn:homogeneous log gamma polymer}).

\end{theorem} 

The proof of this theorem is given in Section \ref{subsec:proof of main theorem LG}. In Section \ref{subsec:Two-layer Gibbs-like measure LG} we construct the two-layer Gibbs measures indexed by down-right paths. 
We will construct in Section \ref{subsec:local Markov operators LG} a Markov dynamics under which the two-layer Gibbs measures are stationary. 
Under $u+v>0$, in Section \ref{subsec:Marginal distributions LG} we will turn the two-layer Gibbs measures into probability measures which are stationary under the log-gamma polymer recurrence relation. To obtain stationarity outside $u+v>0$, in Section \ref{subsec:proof of main theorem LG} we argue by the uniqueness of analytic continuation of real analytic functions.  
Let us also note here that though this theorem is formulated along only horizontal paths, it is possible to formulate and prove such a result along any down-right path in a similar manner.

\subsection{Two-layer Gibbs measure}\label{subsec:Two-layer Gibbs-like measure LG}
We will use the same two-layer graph $\PiP$ associated to a down-right path $\hP$ as in Definition \ref{defn:two path diagram and configuration LPP}. 
A two-layer configuration $\bl=(\la_1^{(0)},\dots,\la_1^{(N)},\la_2^{(0)},\dots,\la_2^{(N)})$ now is  an assignment of $\la_i^{(j)}\in\RR$ ($\RR$ in place of $\ZZ$) to each vertex $\p_i^{(j)}$ of $\PiP$, for $i=1,2$ and $j=0,\dots,N$.

\begin{definition}[Two-layer log-gamma Gibbs measure]\label{defn: Gibbs measures LG}
For $x,y\in \RR$, the weights of solid, dashed and arced edges are
\begin{subequations}
\label{eq:weightsLG}
    \be\label{eq:weight LG boundary}
\wt_{\lgg}\lb\raisebox{-4.5pt}{\arcleftthicklg}\rb=e^{-u(x-y)},\quad \wt_{\lgg}\lb\raisebox{-4.5pt}{\arcrightthicklg}\rb=e^{-v(x-y)},\ee
\be  \label{eq:weight LG edges}
\wt_{\lgg}\lb\bulkrightlg\rb=\wt_{\lgg}\lb\bulkleftlg\rb=e^{-\alpha(x-y)-e^{-(x-y)}},\ee
\be  \label{eq:weight LG dotted edges}
\wt_{\lgg}\lb\bulkrightdotted\rb=\wt_{\lgg}\lb\bulkleftdotted\rb=e^{-e^{-(x-y)}}.\ee
\end{subequations}
Notice that the weights in \eqref{eq:weight LG edges} are precisely of the form of the density of the log-gamma distribution \eqref{eq:lgdensity}. This explains why log-gamma random walks will naturally arise in this context. 
We define the weight $\wt_{\lgg}^{\PiP}(\bl)$ of a two-layer configuration $\bl$ associated to a two-layer graph $\PiP$ to be the product of the above weights over all labelled solid edges, dotted edges and arcs in $\PiP$.
We will also write $\wt_{\lgg}$ to denote the weight of a configuration drawn on a sub-graph of $\PiP$, in a similar manner as \eqref{eq:weight subgraph example}.
Like \eqref{eq:translation invariance two layer LPP},  $\wt_{\lgg}^{\PiP}\lb\bl\rb$ is always positive and translation invariant, in the sense that, for all $x\in\RR$,
\be\label{eq:translation invariance two layer LG}\wt_{\lgg}^{\PiP}(\bl)=\wt_{\lgg}^{\PiP}(\bl+x).\ee
\end{definition}

We will view the weight $\wt_{\lgg}^{\PiP}(\bl)$ as a measure on 
$\{\bl\in \RR^{2N+2}\}$. Due to the translation invariance \eqref{eq:translation invariance two layer LG}, $\wt_{\lgg}^{\PiP}$ must have infinite mass. 
However, as we will see in Proposition \ref{prop:finiteness of weights LG}, the measure of $\bl$ with fixed $\la_1^{(0)}$ (or any other fixed coordinate $\la_i^{(j)}$) is finite under the assumption $u+v>0$. This will be key in turning these Gibbs measures into probability measures.  

We now prove a Cauchy and a Littlewood type identity for the log-gamma weights.
\begin{lemma}[Cauchy type identity]\label{lem:Cauchy LG}
For any $\alpha,\beta\in \CC$ with $\Re(\alpha+\beta)>0$ and any $\la_1,\la_2,\mu_1,\mu_2\in\RR$
    \be\label{eq:diagram Cauchy LG}
         \int_{\RR^2}\wt_{\lgg}\lb\raisebox{-35pt}{\begin{tikzpicture}
        \draw[thick] (-0.6,0.6)--(0,0)--(0.6,0.6);
        \draw[thick] (-0.6,-0.6)--(0,-1.2)--(0.6,-0.6);
        \draw[dashed] (-0.6,-0.6)--(0,0)--(0.6,-0.6);
        \node[above right] at (0.5,0.55) {\small $\mu_1$};
        \node[above right] at (0.5,-0.65) {\small  $\mu_2$};
        \node[above left] at (-0.5,0.55) {\small   $\la_1$};
        \node[above left] at (-0.5,-0.65) {\small  $\la_2$};
        \node[above] at (0.3,0.25) {\textcolor{red}{\small   $\beta$}};
        \node[above] at (0.3,-0.95) {\textcolor{red}{\small  $\beta$}};
        \node[above] at (-0.3,0.3) {\textcolor{red}{\small  $\alpha$}};
        \node[above] at (-0.3,-0.9) {\textcolor{red}{\small  $\alpha$}};
        \node[below] at (0,0) {\small  $\k_1$};
        \node[below] at (0,-1.2) {\small  $\k_2$};
\end{tikzpicture}}\rb\d\k_1\d\k_2
=\int_{\RR^2}\wt_{\lgg}\lb\raisebox{-35pt}{\begin{tikzpicture}
    \draw[thick] (-0.6,0.6)--(0,1.2)--(0.6,0.6);
    \draw[dashed] (-0.6,0.6)--(0,0)--(0.6,0.6);
    \draw[thick] (-0.6,-0.6)--(0,0)--(0.6,-0.6);
    \node[below right] at (0.5,0.65) {\small $\mu_1$};
    \node[below right] at (0.5,-0.55) {\small   $\mu_2$};
    \node[below left] at (-0.45,0.66) {\small  $\la_1$};
    \node[below left] at (-0.45,-0.55) {\small  $\la_2$};
    \node[above] at (0,0) {\small  $\pi_2$};
    \node[above] at (0,1.2) {\small  $\pi_1$};
    \node[above] at (0.35,-0.35){\textcolor{red}{\small  $\alpha$}};
    \node[above] at (0.35,0.85){\textcolor{red}{\small  $\alpha$}};
    \node[above] at (-0.35,-0.4){\textcolor{red}{\small  $\beta$}};
    \node[above] at (-0.35,0.8){\textcolor{red}{\small  $\beta$}};
\end{tikzpicture}}\rb\d\pi_1\d\pi_2,\ee
with both integrals finite.
\end{lemma}
\begin{proof}
This identity can be explicitly written as (finiteness is easily seen from this too)
\begin{multline*}
    \int_{\RR^2} \d\pi_1\d\pi_2e^{-\a(\pi_1+\pi_2-\mu_1-\mu_2)-\b(\pi_1+\pi_2-\la_1-\la_2)-e^{-(\pi_1-\la_1)}-e^{-(\la_1-\pi_2)}-e^{-(\pi_2-\la_2)}-e^{-(\pi_1-\mu_1)}-e^{-(\mu_1-\pi_2)}-e^{-(\pi_2-\mu_2)}} \\ =  \int_{\RR^2}\d\k_1\d\k_2 e^{-\a(\la_1+\la_2-\k_1-\k_2)-\b(\mu_1+\mu_2-\k_1-\k_2)-e^{-(\la_1-\k_1)}-e^{-(\k_1-\la_2)}-e^{-(\la_2-\k_2)}-e^{-(\mu_1-\k_1)}-e^{-(\k_1-\mu_2)}-e^{-(\mu_2-\k_2)}},
\end{multline*} 
which follows from change of variables 
$\pi_1=-\k_2+\log\lb e^{\la_1}+e^{\mu_1}\rb-\log\lb e^{-\la_2}+e^{-\mu_2}\rb$ and 
$\pi_2=-\k_1+\log\lb e^{\la_2}+e^{\mu_2}\rb-\log\lb e^{-\la_1}+e^{-\mu_1}\rb$.
\end{proof}
\begin{lemma}[Littlewood type identity]\label{lem:Littlewood LG}
For any $u,\alpha\in \CC$ with $\Re(u+\alpha)>0$ and any $\k_1,\k_2\in\RR$
\be \label{eq:diagram Littlewood LG}
 \int_{\RR^2}\wt_{\lgg}\lb\raisebox{-35pt}{\begin{tikzpicture}
    \draw[dashed](0,0.6)--(0.6,0);
    \draw[thick] (0,-0.6)--(0.6,0);
    \draw[thick] (0,0.6)--(0.6,1.2);
    \node[above right] at (0.5,-.1) {\small $\k_2$};
    \node[above right] at (0.5,1.1) {\small $\k_1$};
    \node[below] at (0,-0.6) {\small $\la_2$};
    \node[above] at (0,0.6) {\small $\la_1$};
    \node[above] at (0.3,-0.3) {\textcolor{red}{\small $\alpha$}};
    \node[above] at (0.3,0.9) {\textcolor{red}{\small $\alpha$}};
    \node[left] at (-0.6,0) {\textcolor{red}{\small $u$}};
    \draw[thick] (0,0.6) arc (90:270:0.6);
\end{tikzpicture}} \rb\d\la_1\d\la_2
=\int_{\RR^2}\wt_{\lgg}\lb\raisebox{-35pt}{\begin{tikzpicture}
    \draw[dashed](0.6,0)--(0,-0.6);
    \draw[thick] (0,0.6)--(0.6,0);
    \draw[thick] (0,-0.6)--(0.6,-1.2);
    \node[below right] at (0.5,-1.1) {\small $\k_2$};
    \node[below right] at (0.5,.1) {\small $\k_1$};
    \node[below] at (0,-0.6) {\small $\pi_2$};
    \node[above] at (0,0.6) {\small $\pi_1$};
    \node[above] at (0.35,0.3) {\textcolor{red}{\small $\alpha$}};
    \node[above] at (0.35,-0.9) {\textcolor{red}{\small $\alpha$}};
    \node[left] at (-0.6,0) {\textcolor{red}{\small $u$}};
    \draw[thick] (0,0.6) arc (90:270:0.6);
\end{tikzpicture}}\rb\d\pi_1\d\pi_2,\ee
with both integrals finite.
\end{lemma}

\begin{proof}
This identity can be explicitly written as  (finiteness is easily seen from this too)
\begin{multline} \label{eq:Littlewood LG}
        \int_{\RR^2}\d\pi_1\d\pi_2e^{-\a(\pi_1+\pi_2-\k_1-\k_2)-u(\pi_1-\pi_2)-e^{-(\pi_1-\k_1)}-e^{-(\k_1-\pi_2)}-e^{-(\pi_2-\k_2)}} \\
        =\int_{\RR^2}\d\la_1\d\la_2e^{-\a(\k_1+\k_2-\la_1-\la_2)-u(\la_1-\la_2)-e^{-(\k_1-\la_1)}-e^{-(\la_1-\k_2)}-e^{-(\k_2-\la_2)}}, 
    \end{multline}
which follows from change of variables
$\pi_1=-\lambda_2+\kappa_1+\kappa_2$ and  $\pi_2=-\lambda_1+\kappa_1+\kappa_2$.
\end{proof}

\begin{remark}
\label{rem:Whittaker}
Identities \eqref{eq:diagram Cauchy LG} and \eqref{eq:diagram Littlewood LG} can be generalized as in Propositions \ref{prop:Cauchy identity LPP} and \ref{prop:Littlewood identity LPP}  to integral identities involving multiple layers (general $n\in \ZZ_{\geq 1}$ for  \eqref{eq:diagram Cauchy LG} and $n\in 2\ZZ_{\geq 1}$ for  \eqref{eq:diagram Littlewood LG}) , and  $\alpha, \beta$ can be replaced by sets of variables. The weights in \eqref{eq:diagram Cauchy LG} and \eqref{eq:diagram Littlewood LG} can be written in terms of Baxter Q operators from \cite{gerasimov2008baxter} which may be thought as skew version of the $\mathfrak{gl}_n(\mathbb R)$-Whittaker functions. Actually, the skew Cauchy type identity is \cite[Theorem 2.3 (2.24)]{gerasimov2008baxter}, up to a change of variables. Given that such skew Cauchy and skew Littlewood identities hold for skew Schur functions as well as skew Whittaker functions, it is likely that similar identities hold as well for skew $q$-Whittaker functions. 
\end{remark} 

\subsection{Local Markov kernels}\label{subsec:local Markov operators LG}
As in the geometric LPP case in Section \ref{subsec:local Markov operators LPP}, in this section we define a two-layer local Markov dynamics on the strip that preserves the two-layer Gibbs measure in Section \ref{subsec:Two-layer Gibbs-like measure LG}. Recall two-layer graph local moves  (Definition \ref{defn: local moves}) and the two-layer geometric Markov dynamics (Definition \ref{def:twolayerdynamics}).

\begin{definition}[Two-layer log-gamma Markov dynamics]\label{def:twolayerdynamicsLG}
A two-layer log-gamma Markov dynamic on the strip is defined as a collection of transition probabilities densities $\{\U_{\lgg}^{\PiP,\mathcal{G}\widetilde{\hP}}\}$ on the state space $\RR^{2N+2}$ associated to each pair $(\hP,\widetilde{\hP})$ of down-right paths connected by a single local move, i.e. $\hP\mapsto\widetilde{\hP}$. Here $\U_{\lgg}^{\PiP,\mathcal{G}\widetilde{\hP}}\lb\widetilde{\bl}|\bl\rb$ is the probability density of transitioning from configuration $\bl\in\RR^{2N+2}$ for the two-layer graph $\PiP$ to the configuration  $\widetilde{\bl}\in\RR^{2N+2}$ for the two-layer graph $\mathcal{G}\widetilde{\hP}$. We require that these transition probability densities are local in that they only depend on the coordinates in $\bl$ associated to vertices proximate to the local move; and they act on the identity on all coordinates except for the update/updated vertices. Specifically, if the updated vertex from $\hP$ to $\widetilde{\hP}$ is a bulk vertex $\p_j$ then
$$\U_{\lgg}^{\PiP,\mathcal{G}\widetilde{\hP}}\lb\widetilde{\bl}|\bl\rb = 
\prod_{i\neq j} \delta_{\widetilde{\lambda}^{(i)}=\lambda^{(i)}} \,\cdot\,\Ubg(\widetilde{\lambda}^{(j)}|\lambda^{(j-1)},\lambda^{(j)},\lambda^{(j+1)};\alpha,\beta)
$$
where $\Ubg(\pi|\la,\k, \mu;\alpha,\beta)$ is a probability density in $\pi\in \RR^2$ given  $\lambda,\k,\mu\in\RR^2$, and where $\alpha,\beta$ are the bulk parameters labeling the edges $(\p_{j-1},\p_j)$ and $(\p_j,\p_{j+1})$ respectively. Notice the presence of Dirac delta functions above in the description of $\U_{\lgg}^{\PiP,\mathcal{G}\widetilde{\hP}}\lb\widetilde{\bl}|\bl\rb$. This implies that the only variable that changes is $\lambda^{(j)}$ and thus it was a bit of an abuse of notation to call this a probability density. However, $\Ubg(\pi|\la,\k, \mu;\alpha,\beta)$ is a bona-fida probability density in $\pi$.
If the update vertex from $\hP$ to $\widetilde{\hP}$ is a boundary vertex (say $\p_0$, the left-boundary) then
$$
\U_{\lgg}^{\PiP,\mathcal{G}\widetilde{\hP}}\lb\widetilde{\bl}|\bl\rb = 
\prod_{i\neq 0} \delta_{\widetilde{\lambda}^{(i)}=\lambda^{(i)}} \,\cdot\,
\Ulg(\widetilde{\lambda}^{(0)}|\lambda^{(0)},\lambda^{(1)};u,\alpha)
$$
where $\Ulg(\pi|\la,\k;u,\alpha)$ is a probability distribution in $\pi\in \RR^2$ given  $\lambda,\k\in \RR^2$, and where $u$ is the boundary parameter labeling the left-boundary arc, and $\alpha$ is the bulk parameters labeling the edge ($\p_0,\p_1)$.
Similarly, if the update vertex was on the right boundary $\U_{\lgg}^{\PiP,\mathcal{G}\widetilde{\hP}}\lb\widetilde{\bl}|\bl\rb$ is defined via $\Ur(\pi|\k,\mu;\alpha,v)$ for the associated bulk and (right) boundary parameters $\alpha$ and $v$. Our notation here for the local transition probabilities emphasizes their bulk or boundary nature via the superscript which is harvested from \eqref{eq:local move one path bulk}.

In order that $\U_{\lgg}^{\PiP,\mathcal{G}\widetilde{\hP}}$ preserves the two-layer log-gamma Gibbs measures it will be sufficient that the local transition probabilities $\Ubg,\Ulg$ and $\Urg$ preserve the local weights. Specifically we will assume that for all $\pi,\la,\k,\mu\in \RR^2$ and all $\alpha,\beta,u,v>0$ such that $\alpha+\beta,\alpha+u,\alpha+v>0$, 

\begin{subequations}
\begin{alignat}{3}
\int_{\RR^2}\Ubg(\pi|\la,\mu,\k;\alpha,\beta)\wt_{\lgg}\lb\raisebox{-35pt}{\begin{tikzpicture}
        \draw[thick] (-0.6,0.6)--(0,0)--(0.6,0.6);
        \draw[thick] (-0.6,-0.6)--(0,-1.2)--(0.6,-0.6);
        \draw[dashed] (-0.6,-0.6)--(0,0)--(0.6,-0.6);
        \node[above right] at (0.5,0.55) {\small $\mu_1$};
        \node[above right] at (0.5,-0.65) {\small  $\mu_2$};
        \node[above left] at (-0.5,0.55) {\small   $\la_1$};
        \node[above left] at (-0.5,-0.65) {\small  $\la_2$};
        \node[above] at (0.3,0.25) {\textcolor{red}{\small   $\beta$}};
        \node[above] at (0.3,-0.95) {\textcolor{red}{\small  $\beta$}};
        \node[above] at (-0.3,0.3) {\textcolor{red}{\small  $\alpha$}};
        \node[above] at (-0.3,-0.9) {\textcolor{red}{\small  $\alpha$}};
        \node[below] at (0,0) {\small  $\k_1$};
        \node[below] at (0,-1.2) {\small  $\k_2$};
\end{tikzpicture}}\rb\d\k_1\d\k_2
=\wt_{\lgg}\lb\raisebox{-35pt}{\begin{tikzpicture}
    \draw[thick] (-0.6,0.6)--(0,1.2)--(0.6,0.6);
    \draw[dashed] (-0.6,0.6)--(0,0)--(0.6,0.6);
    \draw[thick] (-0.6,-0.6)--(0,0)--(0.6,-0.6);
    \node[below right] at (0.5,0.65) {\small $\mu_1$};
    \node[below right] at (0.5,-0.55) {\small   $\mu_2$};
    \node[below left] at (-0.45,0.66) {\small  $\la_1$};
    \node[below left] at (-0.45,-0.55) {\small  $\la_2$};
    \node[above] at (0,0) {\small  $\pi_2$};
    \node[above] at (0,1.2) {\small  $\pi_1$};
    \node[above] at (0.35,-0.35){\textcolor{red}{\small  $\alpha$}};
    \node[above] at (0.35,0.85){\textcolor{red}{\small  $\alpha$}};
    \node[above] at (-0.35,-0.4){\textcolor{red}{\small  $\beta$}};
    \node[above] at (-0.35,0.8){\textcolor{red}{\small  $\beta$}};
\end{tikzpicture}}\rb,  \label{eq:diagram definition bulk local operator LG}  \\
\int_{\RR^2}\Ulg(\pi|\k,\la;u,\alpha)\wt_{\lgg}\lb\raisebox{-35pt}{\begin{tikzpicture}
    \draw[dashed](0,0.6)--(0.6,0);
    \draw[thick] (0,-0.6)--(0.6,0);
    \draw[thick] (0,0.6)--(0.6,1.2);
    \node[above right] at (0.5,-.1) {\small $\k_2$};
    \node[above right] at (0.5,1.1) {\small $\k_1$};
    \node[below] at (0,-0.6) {\small $\la_2$};
    \node[above] at (0,0.6) {\small $\la_1$};
    \node[above] at (0.3,-0.3) {\textcolor{red}{\small $\alpha$}};
    \node[above] at (0.3,0.9) {\textcolor{red}{\small $\alpha$}};
    \node[left] at (-0.6,0) {\textcolor{red}{\small $u$}};
    \draw[thick] (0,0.6) arc (90:270:0.6);
\end{tikzpicture}}\rb\d\la_1\d\la_2
=\wt_{\lgg}\lb\raisebox{-35pt}{\begin{tikzpicture}
    \draw[dashed](0.6,0)--(0,-0.6);
    \draw[thick] (0,0.6)--(0.6,0);
    \draw[thick] (0,-0.6)--(0.6,-1.2);
    \node[below right] at (0.5,-1.1) {\small $\k_2$};
    \node[below right] at (0.5,.1) {\small $\k_1$};
    \node[below] at (0,-0.6) {\small $\pi_2$};
    \node[above] at (0,0.6) {\small $\pi_1$};
    \node[above] at (0.35,0.3) {\textcolor{red}{\small $\alpha$}};
    \node[above] at (0.35,-0.9) {\textcolor{red}{\small $\alpha$}};
    \node[left] at (-0.6,0) {\textcolor{red}{\small $u$}};
    \draw[thick] (0,0.6) arc (90:270:0.6);
\end{tikzpicture}}\rb,\label{eq:diagram definition left boundary local operator LG} \\
\int_{\RR^2}\Urg(\pi|\k,\mu;\alpha,v)\wt_{\lgg}\lb\raisebox{-35pt}{\begin{tikzpicture}
    \draw[dashed] (0,0)--(0.6,0.6);
    \draw[thick] (0.6,0.6)--(0,1.2);
    \draw[thick] (0,0)--(0.6,-0.6);
    \node[above left] at (0.1,-0.1) {\small  $\k_2$};
    \node[above left] at (0.1,1) {\small  $\k_1$};
    \node[below ] at (0.7,-0.6) {\small $\mu_2$};
    \node[above ] at (0.7,0.6) {\small  $\mu_1$};
    \node  at (0.35,-0.1) {\textcolor{red}{\small $\alpha$}};
    \node  at (0.35,1.1) {\textcolor{red}{\small $\alpha$}};
    \node[right] at (1.2,0) {\textcolor{red}{\small  $v$}};
    \draw[thick] (0.6,-0.6) arc (-90:90:0.6);
\end{tikzpicture}}\rb\d\mu_1\d\mu_2
=\wt_{\lgg}\lb\raisebox{-35pt}{\begin{tikzpicture}
    \draw[dashed](0.6,0)--(0,0.6);
    \draw[thick] (0,0.6)--(0.6,1.2);
    \draw[thick] (0,-0.6)--(0.6,0);
    \node[ left] at (0.1,-0.6) {\small  $\k_2$};
    \node[ left] at (0.1,0.6) {\small $\k_1$};
    \node[below  ] at (0.7,0) {\small $\pi_2$};
    \node[above  ] at (0.7,1.2) {\small  $\pi_1$};
    \node  at (0.25,-0.1) {\textcolor{red}{\small $\alpha$}};
    \node  at (0.25,1.1) {\textcolor{red}{\small $\alpha$}};
    \node[right] at (1.2,0.6) {\textcolor{red}{\small $v$}};
    \draw[thick] (0.6,0) arc (-90:90:0.6);
\end{tikzpicture}}\rb. \label{eq:diagram definition right boundary local operator LG}
\end{alignat}
\end{subequations} 
These equations do not uniquely specify the transition probabilities. Definition \ref{def:push-blockLG} exhibits one solution.

As in the geometric case, this defines a transition probability density $\U_{\lgg}^{\PiP,\mathcal{G}\hQ}\lb\bl'|\bl\rb$, for any pair of down-right paths with $\hQ$ above $\PiP$ and where $\bl\in\RR^{2N+2}$ is a configuration on $\PiP$ and $\bl'\in \RR^{2N+2}$ on $\mathcal{G}\hQ$.
\end{definition}

\begin{corollary}
\label{cor:properties local operators LG}
We have the following properties:
\begin{enumerate}[wide, labelwidth=0pt, labelindent=0pt]
    \item [(1)] For all $\bl'\in\RR^{2N+2}$, we have
    \be\label{eq:compatibility of local operator with wt LG}
    \int_{\RR^{2N+2}}\U_{\lgg}^{\PiP,\mathcal{G}\hQ}\lb\bl'|\bl\rb\wt_{\lgg}^{\gp}\lb\bl\rb\d\bl
    =\wt_{\lgg}^\gq\lb\bl'\rb.
    \ee
    \item[(2)] For all $x\in\RR$, we have
    \be\label{eq:translation invariance Markov operators LG}
    \U_{\lgg}^{\PiP,\mathcal{G}\hQ}\lb\bl'|\bl\rb=\U_{\lgg}^{\PiP,\mathcal{G}\hQ}\lb\bl'+x|\bl+x\rb.
    \ee 
\end{enumerate} 
\end{corollary}
\begin{proof}
This is proved exactly as with Corollary \ref{Cor:properties local operators LPP} in geometric case.  
\end{proof}

As in Definition \ref{def:push-blockLG} we define push-block solutions to the local transition probability density identities. 

\begin{definition}[Log-gamma push-block dynamics]\label{def:push-blockLG}
For the bulk, there is a unique solution $\Ubg(\pi|\la,\k,\mu;\alpha,\beta)$ to \eqref{eq:diagram definition bulk local operator LG} which  does not depend on $\k$. Denoting this by $\Ubg(\pi\vert \lambda, \mu;\alpha,\beta)$ observe that it is given by the weight on the right-hand side of \eqref{eq:diagram definition bulk local operator LG} divided by the sum of weights on the left-hand side (without the $\Ubg$ factor). Explicitly plugging in the weights, this can be written as
\begin{multline}\label{eq:bulk operator LG}
    \Ubg\lb\pi|\la,\mu;\alpha,\beta\rb = \\
    \frac{e^{-\a(\pi_1+\pi_2-\mu_1-\mu_2)-\b(\pi_1+\pi_2-\la_1-\la_2)-e^{-(\pi_1-\la_1)}-e^{-(\la_1-\pi_2)}-e^{-(\pi_2-\la_2)}-e^{-(\pi_1-\mu_1)}-e^{-(\mu_1-\pi_2)}-e^{-(\pi_2-\mu_2)}}}{\int_{\RR^2}e^{-\a(\la_1+\la_2-\k_1-\k_2)-\b(\mu_1+\mu_2-\k_1-\k_2)-e^{-(\la_1-\k_1)}-e^{-(\k_1-\la_2)}-e^{-(\la_2-\k_2)}-e^{-(\mu_1-\k_1)}-e^{-(\k_1-\mu_2)}-e^{-(\mu_2-\k_2)}}\d\k_1\d\k_2}. 
\end{multline}
Due to the Cauchy type identity (Lemma \ref{lem:Cauchy LG}) this is a transition probability density.
Similarly, for the left boundary, the unique solution to \eqref{eq:diagram definition left boundary local operator LG} which does not depend on $\la$ is given explicitly by
 \be   \label{eq:left boundary operator LG}
\Ulg(\pi|\k;u,\alpha)=\frac{e^{-\a(\pi_1+\pi_2-\k_1-\k_2)-u(\pi_1-\pi_2)-e^{-(\pi_1-\k_1)}-e^{-(\k_1-\pi_2)}-e^{-(\pi_2-\k_2)}}}{ \int_{\RR^2}e^{-\a(\k_1+\k_2-\la_1-\la_2)-u(\la_1-\la_2)-e^{-(\k_1-\la_1)}-e^{-(\la_1-\k_2)}-e^{-(\k_2-\la_2)}}\d\la_1\d\la_2}.
\ee
Owing to the Littlewood type  identity  (Lemma \ref{lem:Littlewood LG}), this is a transition probability density. Likewise, at the right boundary the unique solution to  \eqref{eq:diagram definition right boundary local operator LG}
which does not depend on $\mu$ is given explicitly by
\be   \label{eq:right boundary operator LG}
\Urg(\pi|\k;\alpha,v)=\frac{e^{-\a(\pi_1+\pi_2-\k_1-\k_2)-v(\pi_1-\pi_2)-e^{-(\pi_1-\k_1)}-e^{-(\k_1-\pi_2)}-e^{-(\pi_2-\k_2)}}}{\int_{\RR^2}e^{-\a(\k_1+\k_2-\mu_1-\mu_2)-v(\mu_1-\mu_2)-e^{-(\k_1-\mu_1)}-e^{-(\mu_1-\k_2)}-e^{-(\k_2-\mu_2)}}\d\mu_1,\d\mu_2}.
\ee 
\end{definition}

\subsection{First layer marginal}
\label{subsec:Marginal distributions LG}
We first show that the marginal distribution on the first layer of the two-layer log-gamma  push-block Markov dynamics correspond to the recurrence relation defining the log-gamma polymer partition function.
For a configuration $\bl$ on a two-layer graph we will use the shorthand $\bl_i:=(\bl_i^{(0)},\dots,\bl_i^{(N)})$ for $i=1,2$ so that $\bl=(\bl_1,\bl_2)$.
Under the additional restriction that $u+v>0$, we will then take the marginal measure of the two-layer Gibbs measure $\wt_{\lgg}^{\PiP}$ on the first layer $\bl_1$ and multiply it by a finite normalization constant to define a probability measure $\mathrm{P}_{\lgg}^{\hP}$. We then prove the stationarity of $\mathrm{P}_{\lgg}^{\hP}$ under the log-gamma polymer recurrence relation.

\begin{lemma}[First layer dynamics match the log-gamma polymer recurrence relation]\label{lem:marginal operator LG} The log-gamma push-block Markov dynamics (Definition \ref{def:push-blockLG}) restricted to the configuration on the upper layer of the two-layer graphs are marginally Markov and agree with the dynamics imposed by the log-gamma polymer recurrence relation (Definition \ref{defn:inhomogeneous LG and specializations} and \eqref{eq:recurrence log gamma}).
In particular, this means that the law of $\pi_1$ under 
$\Ubg(\pi|\la,\mu;\alpha,\beta)$ only depends on $\la,\mu$ 
through $\la_1,\mu_1$ and that law has density 
$\rmUblg(\pi_1|\la_1,\mu_1;\alpha,\beta)$ given explicitly by 
$$\rmUblg(\pi_1|\la_1,\mu_1;\alpha,\beta)=Z(\la_1,\mu_1;\alpha,\beta)^{-1}e^{-(\a+\b)\pi_1-\frac{e^{\la_1}+e^{\mu_1}}{e^{\pi_1}}}$$
for some normalization $Z(\la_1,\mu_1;\alpha,\beta)>0$. This implies that for an independent $\vo\sim\Ga^{-1}(\a+\b)$,
$$e^{\pi_1}=\vo(e^{\la_1}+e^{\mu_1}).$$
Similarly, on the left and right boundaries, the law of $\pi$ under $\Ulg(\pi|\k;u,\alpha)$ and $\Urg(\pi|\k;\alpha,v)$ depend on $\k$ only through $\k_1$. The density of those laws, which we write as $\rmUllg(\pi_1|\k_1;u,\alpha)$ and $\rmUrlg(\pi_1|\k_1;\alpha,v)$ respectively, are given explicitly by 
$$
\rmUllg(\pi_1|\k_1;u,\alpha)=Z(\k_1;u,\alpha)^{-1} e^{-(\a+u)\pi_1-\frac{e^{\k_1}}{e^{\pi_1}}}\qquad 
\rmUrlg(\pi_1|\k_1;\alpha,v)=Z(\k_1;v,\alpha)^{-1}e^{-(\b+v)\pi_1-\frac{e^{\k_1}}{e^{\pi_1}}}
$$
for some normalization $Z(\k_1;u,\alpha),Z(\k_1;v,\alpha)>0$. This (respectively) implies that for an independent $\vo\sim\Ga^{-1}(\a+u)$ and $\vo\sim\Ga^{-1}(\b+v)$, 
$$e^{\pi_1}=\vo e^{\k_1}\qquad \textrm{and}\quad e^{\pi_1}=\vo e^{\k_1}$$

Thus, the law of $\bl'_1$ under $\U_{\lgg}^{\PiP,\mathcal{G}\hQ}\lb\bl'|\bl\rb$
only depends on $\bl_1$ and hence is written as $\mathrm{U}_{\lgg}^{\hP,\hQ}\lb\bl'_1|\bl_1\rb$. These transition probability densities define Markov dynamics on the first layer of the two-layer graph that coincide with the recurrence relation for the log-gamma polymer free energy (the free energy $h$ is log of the partition function $z$ in \eqref{eq:recurrence log gamma}). In particular, if we initialize that recurrence with 
$h(\p_j)=\la_1^{(j)}$, $0\leq j\leq N$ on $\hP$ (i.e. $z(\p_j)=e^{\la_1^{(j)}}$) then the probability density that 
$h(\q_j)=\la_1'^{(j)}$, $0\leq j\leq N$ on $\hQ$ is precisely $\mathrm{U}_{\lpp}^{\hP,\hQ}\lb\bl'_1|\bl_1\rb$. 
\end{lemma}

\begin{proof}
    We only prove this statement for bulk kernel $\Ubg$ and left boundary kernel $\Ulg$, since the case of right boundary kernel $\Urg$ follows from renaming the variables in $\Ulg$.

    By the Cauchy type identity  (Lemma \ref{lem:Cauchy LG}), the bulk local kernel \eqref{eq:bulk operator LG} can be written as:
\begin{multline*}
    \Ubg\lb\pi|\la,\mu;\alpha,\beta\rb 
    =\\
    \frac{e^{-(\a+\b)\pi_1-e^{-(\pi_1-\la_1)}-e^{-(\pi_1-\mu_1)}}}{\int_{\RR}e^{-(\a+\b)\pi_1-e^{-(\pi_1-\la_1)}-e^{-(\pi_1-\mu_1)}}\d\pi_1}
\frac{e^{-(\a+\b)\pi_2-e^{-(\la_1-\pi_2)}-e^{-(\pi_2-\la_2)}-e^{-(\mu_1-\pi_2)}-e^{-(\pi_2-\mu_2)}}}{\int_{\RR}e^{-(\a+\b)\pi_2-e^{-(\la_1-\pi_2)}-e^{-(\pi_2-\la_2)}-e^{-(\mu_1-\pi_2)}-e^{-(\pi_2-\mu_2)}}\d\pi_2}.
\end{multline*}
Integrate over $\pi_2$, we arrive at the claimed property and formula 
$$\int_{\RR}\Ubg\lb\pi|\la,\mu;\alpha,\beta\rb\d\pi_2=Z(\la_1,\mu_1;\alpha,\beta)^{-1}e^{-(\a+\b)\pi_1-e^{-(\pi_1-\la_1)}-e^{-(\pi_1-\mu_1)}}=\rmUblg(\pi_1|\la_1,\mu_1;\alpha,\beta).$$
By the Littlewood type identity (Lemma \ref{lem:Littlewood LG}), the left boundary local kernel \eqref{eq:left boundary operator LG} can similarly be written as:
$$\Ulg\lb\pi|\k;u,\alpha\rb=\frac{e^{(-\a-u)\pi_1-e^{-(\pi_1-\k_1)}}}{\int_{\RR}e^{(-\a-u)\pi_1-e^{-(\pi_1-\k_1)}}\d\pi_1}
\frac{e^{(-\a+u)\pi_2-e^{-(\k_1-\pi_2)}-e^{-(\pi_2-\k_2)}}}{\int_{\RR}e^{(-\a+u)\pi_2-e^{-(\k_1-\pi_2)}-e^{-(\pi_2-\k_2)}}\d\pi_2}.$$
Integrate over $\pi_2$, we arrive at the claimed property and formula 
$$\int_{\RR}\Ulg\lb\pi|\k;u,\alpha\rb\d\pi_2=Z(\k_1;u,\alpha)^{-1}e^{(-\a-u)\pi_1-e^{-(\pi_1-\k_1)}}=\rmUllg(\pi_1|\k_1;u,\alpha).$$
The claimed law of $\bl'_1$ under $\U_{\lgg}^{\PiP,\mathcal{G}\hQ}\lb\bl'|\bl\rb$ and relation to the polymer recurrence follows immediately.
\end{proof}

We now show that fixing the value of $\bl$ at some vertex in the two-layer graph results in a finite partition function and hence the marginal law given that conditioning can be normalized to be a probability measure. 
 
\begin{proposition}\label{prop:finiteness of weights LG} 
Assume \eqref{eq:condition for inhomogeneous log gamma} and the additional condition $u+v>0$. For any down-right path $\hP$ let
$$Z_{\lgg}=\int_{\mathbb R^{2N+1}} \wt_{\lgg}^{\mathcal{GP}}\lb \bl\rb \prod_{\substack{i=1,2, \;\; j=0,\dots, N \\ (i,j)\neq (1,0)}}\d\lambda_i^{(j)},$$
where $\la_1^{(0)}$ is fixed and the integral is over the set of all $\la_i^{(j)}\in\RR$ for $i=1,2$ and $j=0,\dots,N$ with $(i,j)\neq (1,0)$. Then, $Z_{\lgg}$ is finite and  does not depend on the choice of  $\hP$ or $\la_1^{(0)}$. 
\end{proposition}
\begin{proof} 
The fact that $Z_{\lgg}$ does not depend on the choice of  $\hP$ or $\la_1^{(0)}$ follows from the same argument as in the geometric LPP case (Proposition \ref{prop:finiteness of weights LPP}), with the only difference that we now  use the log-gamma versions of the Cauchy and Littlewood identities (Lemmas \ref{lem:Cauchy LG} and \ref{lem:Littlewood LG}). Hence we only need to prove that $Z_{\lgg}$ is finite for the horizontal path $\hP_h$. The proof in Proposition \ref{prop:finiteness of weights LPP} bounded the partition function by summations involving Schur polynomials. While in this case it may be possible to do something similar with Whittaker functions in place of Schur polynomials, we instead present an inductive proof in $N$ of this finiteness.

\smallskip
\noindent $N=1$ base case claim: When  $u+v>0$,$v+\a>0$ and $u+\a>0$,
we have 
\be\label{in the proof of proposition finiteness LG before analytic continuation two}\int_{\RR^3}\wt_{\lgg}\lb\raisebox{-40pt}{\begin{tikzpicture}
    \draw[dashed] (0,0.6)--(0.6,0);
    \draw[thick] (0,-0.6)--(0.6,0);
    \draw[thick] (0,0.6)--(0.6,1.2);
    \node[below  ] at (0.8,0) {\small$\la_2^{(1)}$};
    \node[above  ] at (0.8,1.2) {\small $\la_1^{(1)}$};
    \node[below] at (0,-0.6) {\small $\la_2^{(0)}$};
    \node[above] at (0,0.6) {\small $\la_1^{(0)}$};
    \node[below] at (0.35,-0.3) {\textcolor{red}{\small $\a$}};
    \node[above] at (0.35,0.9) {\textcolor{red}{\small $\a$}};
    \node[left] at (-0.6,0) {\textcolor{red}{\small $u$}};
    \node[right] at (1.2,0.8) {\textcolor{red}{\small $v$}};
    \draw[thick] (0,0.6) arc (90:270:0.6);
    \draw[thick] (0.6,0) arc (-90:90:0.6);
\end{tikzpicture}}\rb\d\lb\la_1^{(1)}-\la_1^{(0)}\rb\d\lb\la_2^{(0)}-\la_1^{(0)}\rb\d\lb\la_2^{(1)}-\la_1^{(0)}\rb
<\infty.\ee
A change of variables $\la_1^{(1)}-\la_2^{(1)}=x$, $\la_1^{(1)}-\la_1^{(0)}=s$, $\la_2^{(1)}-\la_2^{(0)}=t$ rewrites the LHS above as
$$\int_{\RR^3}e^{-\a s-e^{-s}}e^{-\a t-e^{-t}}e^{-e^{-(x-s)}}e^{-u(x-s+t)}e^{-vx}\d s\d t\d x=\Gamma(u+v)\Gamma(\a+v)\Gamma(u+\a)<\infty,$$
where the evaluation of the integral uses Fubini. This proves the base case claim.

\smallskip
\noindent Inductive claim: For any $u+\a>0$ and $\varepsilon>0$ we can find $0\leq\delta<\varepsilon$ and $C>0$ such that for $\gamma=\min(\a,u)-\delta$
\be\label{eq:in the proof of proposition finiteness weights LG}
    \int_{\RR^2}\wt_{\lgg}\lb\raisebox{-40pt}{\begin{tikzpicture}
    \draw[dashed](0,0.6)--(0.6,0);
    \draw[thick] (0,-0.6)--(0.6,0);
    \draw[thick] (0,0.6)--(0.6,1.2);
    \node[above right] at (0.5,-.1) {\small $\la_2^{(1)}$};
    \node[above right] at (0.5,1.1) {\small $\la_1^{(1)}$};
    \node[below] at (0,-0.6) {\small $\la_2^{(0)}$};
    \node[above] at (0,0.6) {\small $\la_1^{(0)}$};
    \node[above] at (0.35,-0.3) {\textcolor{red}{\small $\alpha$}};
    \node[above] at (0.35,0.9) {\textcolor{red}{\small $\alpha$}};
    \node[left] at (-0.6,0) {\textcolor{red}{\small $u$}};
    \draw[thick] (0,0.6) arc (90:270:0.6);
\end{tikzpicture}}\rb\d\la_1^{(0)}\d\la_2^{(0)}
\leq Ce^{-\gamma(\la_1^{(1)}-\la_2^{(1)})} 
=C\wt_{\lgg}\lb\raisebox{-33pt}{\begin{tikzpicture}
    \node[below] at (0.2,-0.6) {\small $\la_2^{(1)}$};
    \node[above] at (0.2,0.6) {\small $\la_1^{(1)}$};
    \node[left] at (-0.6,0){\textcolor{red}{\small $\gamma$}};
    \draw[thick] (0,0.6) arc (90:270:0.6);
\end{tikzpicture}}\rb.
\ee
Suppose the above claim holds. Then the partition function $Z_{\lgg}$ is bounded above by a constant times the partition function with $N$ replaced by $N-1$ and the boundary parameter $u$ replaced by $\gamma$. If we choose $\ep>0$ to be the minimum of sums of any two parameters in $u,v,\a_1,\dots,\a_N$, then in the partition function after the reduction, the  sum of any two distinct parameters is still positive. Repeating this reduction until the $N=1$ case and then using the above calculated bound in the case proves the proposition.  

To prove \eqref{eq:in the proof of proposition finiteness weights LG} we let $\la_1^{(1)}-\la_2^{(1)}=x$ and make the change of variables 
$\la_1^{(1)}-\la_1^{(0)}=s$, $\la_2^{(1)}-\la_2^{(0)}=t$ in the integral in the left-hand side of \eqref{eq:in the proof of proposition finiteness weights LG}. That desired inequality is rewritten (after multiplying $e^{ux}$) as 
%becomes equivalent to $$\int_{\RR^2}e^{-\a s-e^{-s}}e^{-\a t-e^{-t}}e^{-e^{-(x-s)}}e^{-u(x-s+t)}\d s\d t\leq Ce^{-\gamma x},$$ i.e.
\be\label{eq:58rewrite}
\int_{\RR^2}e^{(u-\a)s-e^{-s}-e^{s-x}}e^{-(u+\a)t-e^{-t}}\d s\d t\leq Ce^{(u-\gamma)x}.
\ee
We prove \eqref{eq:58rewrite} in two cases. When $\a>u$, by dropping $e^{s-x}>0$ in the exponent we have
$$\text{LHS\eqref{eq:58rewrite}}<\int_{\RR^2}e^{(u-\a)s-e^{-s}}e^{-(u+\a)t-e^{-t}}\d s\d t=\Gamma(\a-u)\Gamma(\a+u).$$
Observe that the right-hand side is a constant with respect to $x$ and thus if we take $\delta=0$ (so that $u-\gamma=0$) and $C=\Gamma(\a-u)\Gamma(\a+u)$, we have $\Gamma(\a-u)\Gamma(\a+u)\leq Ce^{(u-\gamma)x}$ as desired to show \eqref{eq:58rewrite}.
The second case to show \eqref{eq:58rewrite} is when $\a\leq u$ in which case we chose any $0<\delta<\ep$. Rearranging \eqref{eq:58rewrite}, we need to prove that
    $$\int_{\RR^2}e^{(u-\a)s-e^{-s}-e^{s-x}-(u-\a+\delta)x}e^{-(u+\a)t-e^{-t}}\d s\d t\leq C.$$
  The integral in $t$ is constant in $x$ so we only need to prove that there is some $C>0$ such that for any $x\in\RR$,
    $$\int_{\RR}e^{(u-\a)s-e^{-s}-e^{s-x}-(u-\a+\delta)x}\d s\leq C.$$
    The function $e^{s-x}+(u-\a+\delta)x$ goes to infinity at both $x=-\infty$ and $x=\infty$. We taking derivatives shows that this function
    %$$\frac{d}{dx}\lb e^{s-x}+(u-\a+\delta)x\rb=-e^{s-x}+(u-\a+\delta),$$
    attains its global minimum at $x=s-\ln(u-\a+\delta)$. Thus, 
    $$
     \int_{\RR}e^{(u-\a)s-e^{-s}-e^{s-x}-(u-\a+\delta)x}\d s\leq \int_{\RR}e^{(u-\a)s-e^{-s}-(u-\a+\delta)-(u-\a+\delta)(s-\ln(u-\a+\delta))}\d s
    $$
    which is easily seen to be a finite constant in $x$ (in fact $e^{-(u-\a+\delta)+(u-\a+\delta)\ln(u-\a+\delta)}\Gamma(\delta)$). This completes the proof of  \eqref{eq:58rewrite}, and hence the inductive claim and thus the proposition.
\end{proof}

We will now prove that the first layer marginal distribution of our two-layer Gibbs measures are stationary measures for the log-gamma polymer recurrence relation. To state this precisely, we need to introduce a few pieces of notation. 
%\begin{definition}\label{defn:normalized measure and probability measure LPP before sum zero mode}
Recalling the decomposition of $\bl=(\bl_1,\bl_2)$ define first layer marginal weights by
\be\label{eqn:wtPlgg}
\wt_{\lgg}^{\hP}(\bl_1):=\int_{\RR^{N+1}}\wt_{\lgg}^{\gp}\lb\bl\rb\d\bl_2.
\ee
The translation invariance \eqref{eq:translation invariance two layer LG} of $\wt_{\lpp}^{\PiP}(\bl)$ implies that  for any $x\in\RR$, $\wt_{\lgg}^{\hP}(\bl_1+x)=\wt_{\lgg}^{\hP}(\bl_1)$.
When $u+v>0$, by Proposition \ref{prop:finiteness of weights LG}, for any fixed $\la_1^{(0)}$ we have
\be\label{eq:sum one layer being normalizing constant Z LG}
Z_{\lgg}=\int_{\RR^N}\wt_{\lgg}^{\hP}(\bl_1)\prod_{j=1}^N\d\la_1^{(j)}<\infty.
\ee
Let us introduce variables that record the first layer configuration centered by $\la^{(0)}_{1}$ for $1\leq j\leq N$ and the shorthand notation $\bL_1:=(L_1{(1)},\dots,L_1{(N)})$.  
For $\bL_1\in \RR^N$ define 
\be\label{eqn:Plpphlgg}
\mathrm{P}_{\lgg}^{\hP} \lb \bL_1\rb:=\frac{1}{Z_{\lgg}}\wt_{\lgg}^\hP\lb\bl_1\rb.
\ee
    Due to translation invariance and finiteness of the normalizing constant \eqref{eq:sum one layer being normalizing constant Z LG}, this 
    is a probability density. 
    
    Now we need some notation for the corresponding Markov dynamics. Recall the transition probability density $\mathrm{U}_{\lgg}^{\hP,\hQ}(\bl_1'|\bl_1)$ for $\bl_1,\bl'_1\in\RR^{N+1}$ defined in Lemma \ref{lem:marginal operator LG} which encodes the dynamics of the log-gamma polymer free energy from the path $\hP$ to the path $\hQ$.  We define another transition probability density encoding the dynamics of the centered free energies  $\bL_1$. For any $\bL_1, \bL_1'\in\RR^N$,
    \be\label{eq:definition of first marginal operator sum over x LG}\pU_{\lgg}^{\hP,\hQ}\lb\bL'_1|\bL_1\rb:=\int_{\RR}\mathrm{U}_{\lgg}^{\hP,\hQ}\lb x,\bL'_1+x|0,\bL_1\rb\d x.\ee
%\end{definition}  
    Owing to the translation invariance of the dynamics defined by $\mathrm{U}_{\lgg}^{\hP,\hQ}(\bl_1'|\bl_1)$, this gives the transition probability density from $\bL_1$ to  $\bL_1'$. The  following shows that the weights $\wt_{\lgg}^{\hP}$ from \eqref{eqn:wtPlgg} and the probability density $\mathrm{P}_{\lgg}^{\hP}$ from \eqref{eqn:Plpphlgg} are stationary with respect to $\mathrm{U}_{\lgg}$ and $\pU_{\lgg}$.

\begin{theorem}\label{thm:stationary measure LG before sum over zero mode}
For any $\bl_1'\in\RR^{N+1}$,
\be\label{eq:compatibility of markov operator with one row weitht LG}
\int_{\RR^{N+1}}\mathrm{U}_{\lgg}^{\hP,\hQ}\lb\bl'_1|\bl_1\rb\wt_{\lgg}^{\hP}(\bl_1)\d\bl_1=\wt_{\lgg}^\hQ\lb\bl'_1\rb.
\ee
Assume that $u+v>0$, then for any $\bL_1'\in\RR^{N}$,
\be\label{eq:compatibility of probability measure with LG}
\int_{\RR^N}\pU_{\lgg}^{\hP,\hQ}\lb\bL'_1|\bL_1\rb\mathrm{P}_{\lgg}^{\hP} \lb \bL_1\rb\d\bL_1=\mathrm{P}_{\lgg}^{\hQ} \lb \bL'_1\rb.
\ee 
\end{theorem}

\begin{proof} 
    The first statement \eqref{eq:compatibility of markov operator with one row weitht LG} follows from integrating \eqref{eq:compatibility of local operator with wt LG} over the second layer $\bl_2$ in conjunction with the definition of $\mathrm{U}_{\lgg}^{\hP,\hQ}(\bl'_1|\bl_1)$ as the marginal of $\U_{\lgg}^{\PiP,\mathcal{G}\hQ}\lb\bl'|\bl\rb$ (Lemma \ref{lem:marginal operator LG}).
 The second statement \eqref{eq:compatibility of probability measure with LG} follows from translation invariance and the definition \eqref{eqn:Plpphlgg}
of $\mathrm{P}_{\lgg}^{\hP} \lb \bL_1\rb$:
\begin{equation*}
    \begin{split}
        \text{LHS}\eqref{eq:compatibility of probability measure with LG}&=\frac{1}{Z_{\lgg}}
        \int_{\RR^N}\d\bL_1\int_{\R}\d x
        \mathrm{U}_{\lgg}^{\hP,\hQ}\lb x,\bL'_1+x|0,\bL_1\rb\wt^{\hP}_{\lgg}\lb 0,\bL_1\rb\\
        &=\frac{1}{Z_{\lgg}}\int_{\RR^N}\d\bL_1\int_{\R}\d x\mathrm{U}_{\lgg}^{\hP,\hQ}\lb 0,\bL'_1|-x,\bL_1-x\rb\wt^{\hP}_{\lgg}\lb -x,\bL_1-x\rb\\
        &=\frac{1}{Z_{\lgg}}\wt_{\lgg}^{\hQ}\lb 0,\bL'_1\rb=\text{RHS}\eqref{eq:compatibility of probability measure with LG}.
    \end{split}
\end{equation*} 
\end{proof}

\subsection{Proof of Theorem \ref{thm:main theorem LG}}\label{subsec:proof of main theorem LG} 

So far we have shown that provided $u+v>0$, the stationary measure for the log-gamma polymer free energy recurrence relation can be realized as a marginal of the two-layer Gibbs measures. In order to go beyond this restriction on $u+v>0$ we will integrate out the `zero-mode'. Specifically, we will prove in Proposition \ref{prop:probability measure coincide LG} that, provided $u+v>0$, for a horizontal path $\hP$ with edge labels $\bb=(b_1,\dots,b_N)$, the probability density $\mathrm{P}^{\hP}_{\lgg}$ \eqref{eqn:Plpphlgg} defined as a marginal of the two-layer Gibbs measure coincides with $\mathrm{P}^{\bb,c_1,c_2}_{\stat\lgg}$ defined as a marginal of pair of reweighted inhomogeneous random walks (Definition \ref{defn:rescaled random walks LG inhomogeneous}). We then prove that the probability measure $\mathrm{P}^{\bb,c_1,c_2}_{\stat\lgg}$ is well-defined without the constraint $u+v>0$ and real analytic in these boundary parameters. Combining this with Theorem \ref{thm:stationary measure LG before sum over zero mode} and the uniqueness of analytic continuation of real analytic functions, we prove Theorem \ref{thm:main theorem LG}. This proof follows the same approach as used to prove Theorem \ref{thm:main theorem LPP} in Section \ref{subsec:proof of main theorem LPP} with the exception that the proof of real analyticity in the boundary parameter will be trickier here since the transition probabilities no longer depend as power-series on that parameter. This is due to the fact that the state-space in the log-gamma setting is $\RR$, not $\ZZ$. We will use tools from complex analysis to demonstrate the desired real analyticity.

\begin{proposition}[Integrating out the zero-mode]\label{prop:probability measure coincide LG}
    Suppose $\hP$ is a horizontal path with labels $\bbb=(\b_1,\dots,\b_N)$. Then the probability density $\mathrm{P}^{\hP}_{\lgg}$ coincides with $\mathrm{P}^{\bbb,u,v}_{\stat\lgg}$  (Definition \ref{defn:rescaled random walks LG inhomogeneous}) provided $u+v>0$.  
\end{proposition}
\begin{proof}
    We use the following set of variables: For all $1\leq j\leq N$ let
$$\Delta:=\la_1^{(0)}-\la_2^{(0)},\quad L_1{(j)}:=\la_1^{(j)}-\la_1^{(0)},\quad L_2{(j)}:=\la_2^{(j)}-\la_2^{(0)}.$$
and write $L_1{(0)}=L_2{(0)}=0$ and $\bL_i:=(L_i{(1)},\dots,L_i{(N)})$ for $i=1,2$.
Recall $\bL=(\bL_1,\bL_2)$.
We define:
$$\wth_{\lgg}^{\bbb}\lbe \Delta;\bL\rbe:=\wt_{\lgg}^{\gp}\lbe\bl\rbe,$$
which is well-defined due to the translation invariance \eqref{eq:translation invariance two layer LG} of $\wt_{\lgg}^{\gp}\lbe\bl\rbe$. When $u+v>0$,
\be\label{eq:rewriting weights LG}
\PP_{\lgg}^{\hP}(\bL_1)=\frac{\wt_{\lgg}^{\hP}(\bl_1)}{Z_{\lgg}}=
\frac{\int_{\RR^N}\d\bL_2\int_{\RR}\d\Delta\wth_{\lgg}^{\bbb}\lbe\Delta;\bL\rbe}{\int_{\RR^{2N}}\d\bL\int_{\RR}\d\Delta\wth_{\lgg}^{\bbb}\lbe\Delta;\bL\rbe}.
\ee
By Definition \ref{defn: Gibbs measures LG}, we explicit evaluate this weight as
\be \label{eq:hat weight LG}
\begin{split}
    \wth_{\lgg}^{\bbb}\lbe \Delta;\bL\rbe
    =& e^{-(u+v)\Delta}\prod_{j=1}^Ne^{-e^{-(\Delta+L_1{(j-1)}-L_2{(j)})}}e^{-v(L_1{(N)}-L_2{(N)})}\\
    &\times\lb\prod_{j=1}^N e^{-\b_j(L_1{(j)}-L_1{(j-1)})-e^{-(L_1{(j)}-L_1{(j-1)})}} \rb
    \lb\prod_{j=1}^N e^{-\b_j(L_2{(j)}-L_2{(j-1)})-e^{-(L_2{(j)}-L_2{(j-1)})}}\rb\\
    =&e^{-(u+v)\Delta}\prod_{j=1}^Ne^{-e^{-(\Delta+L_1{(j-1)}-L_2{(j)})}}e^{-v(L_1(N)-L_2(N))}\Gamma(\b_j)^2\PP^{\bbb,\bbb}_{\lgrw}(\bL).
\end{split}
\ee
where $\PP^{\bbb,\bbb}_{\lgrw}$ from Definition \ref{defn:rescaled random walks LG inhomogeneous} is the law of two independent inhomogeneous log-gamma random walks.

When $u+v>0$, we integrate \eqref{eq:hat weight LG} over $\Delta\in\RR$ and obtain
\be\label{eq:integrate hat weight LG}
\Gamma(u+v)^{-1} \prod_{j=1}^N\Gamma(\b_j)^{-2}\int_{\RR}\wth_{\lgg}^{\bbb}\lbe \Delta;\bL\rbe\d \Delta=V_{\lgg}^{u,v}(\bL) \PP^{\bbb,\bbb}_{\lgrw}(\bL), 
\ee 
where we recall  from  \eqref{eq:PstatLG inhomogeneous} that
$$V_{\lgg}^{u,v}(\bL)=\lb \sum_{j=1}^N e^{L_2{(j)}-L_1{(j-1)}}\rb^{-(u+v)}e^{-v(L_1(N)-L_2(N))},$$
and where have used the identity that for $S>0$, 
$$\int_{\RR}e^{-(u+v)\Delta-e^{-\Delta}S}\d \Delta=\Gamma(u+v)S^{-u-v }.$$
%and we recall
%$$V_{\lgg}^{u,v}(\bL)=\lb \sum_{j=1}^N e^{L_2{(j)}-L_1{(j-1)}}\rb^{-(u+v)}e^{-v(L_1(N)-L_2(N))}.$$
By \eqref{eq:rewriting weights LG} and \eqref{eq:integrate hat weight LG} we conclude the proof.
\end{proof}

We now prove that the normalization constant $\mathcal{Z}_{\lgg}$ in the definition \eqref{eq:PstatLG inhomogeneous} of the reweighted inhomogeneous log-gamma random walk measure $\mathrm{P}_{\stat\lgg}^{\bb,c_1,c_2}$ is finite only assuming \eqref{eq:condition for inhomogeneous log gamma} (without the condition $u+v>0$), as claimed in the statement of the main Theorem \ref{thm:main theorem LG}.

\begin{proposition}\label{prop:finiteness of weights after summing zero mode LG}  
Recalling $V^{u,v}_{\lgg}(\bL)$ from \eqref{eq:PstatLG inhomogeneous}  and assuming \eqref{eq:condition for inhomogeneous log gamma},
    $\mathbb{E}^{\bbb,\bbb}_{\lgrw}\lbE V^{c_1,c_2}_{\lgg}(\bL)\rbE$ is finite.
%Recall $$V_{\lgg}^{u,v}(\bL):=\lb \sum_{j=1}^N e^{L_2{(j)}-L_1{(j-1)}}\rb^{-(u+v)}e^{-v(L_1(N)-L_2(N))} .$$ Under the basic condition \eqref{eq:condition for inhomogeneous log gamma} on parameters $u,v, \a_1,\dots,\a_N$, the expectation $\mathbb{E}^{\bbb,\bbb}_{\lgrw}\lbE V^{u,v}_{\lgg}(\bL)\rbE$ is finite.
\end{proposition}
\begin{proof}
Continuing with the notation from Proposition \ref{prop:probability measure coincide LG} and using 
\eqref{eq:integrate hat weight LG},  
when $u+v>0$, we have
\be\label{eq: in proof of finiteness LG}
\mathbb{E}^{\bbb,\bbb}_{\lgrw}\lbE V^{u,v}_{\lgg}(\bL)\rbE=\Gamma(u+v)^{-1} \prod_{j=1}^N\Gamma(\b_j)^{-2} \int_{\RR^{2N}}\d\bL\int_{\RR}\d\Delta\wth_{\lgg}^{\bbb}\lbe\Delta;\bL\rbe.\ee
We observe that the integrals equal $Z_{\lgg}$, which is finite by Proposition \ref{prop:finiteness of weights LG}. Hence \eqref{eq: in proof of finiteness LG} is finite.

Now we consider the case when $u+v\leq 0$. (Note that the argument here differs from that in the proof of Proposition \ref{prop:finiteness of weights after summing zero mode LPP}). First observe that 
$$\mathbb{E}^{\bbb,\bbb}_{\lgrw}\lbE V^{u,v}_{\lgg}(\bL)\rbE=\prod_{j=1}^N \Gamma(\beta_j)^{-2} \Gamma(\beta_j+u)\Gamma(\beta_j+v)\mathbb{E}^{\bbb+v,\bbb+u}_{\lgrw}\lbE \lb \sum_{j=1}^N e^{-(L_2{(N)}-L_2{(j)})-L_1{(j-1)}}\rb^{-(u+v)}\rbE.$$
We will prove that the expectation in the right is finite. Using multiple times the inequality $(x+y)^a\leq \max(1,2^{a-1})(x^a+y^a)$ for any $x,y>0$ and $a\geq 0$ we find that there exists a constant $C_{N,u+v}>0$ such that
\be\label{eq:inequality in proof of finiteness LG}
 \begin{split}       \lb \sum_{j=1}^N e^{-(L_2{(N)}-L_2{(j)})-L_1{(j-1)}}\rb^{-(u+v)}
        \leq C_{N,u+v}\sum_{j=1}^N\lb e^{-(L_2{(N)}-L_2{(j)})-L_1{(j-1)}}\rb^{-(u+v)}\\
        =C_{N,u+v}\sum_{j=1}^N\prod_{\ell=1}^{j-1}\lb e^{-(L_1{(\ell)}-L_1{(\ell-1)})}\rb^{-(u+v)}
        \prod_{\ell=j+1}^{N}\lb e^{-(L_2{(\ell)}-L_2{(\ell-1)})}\rb^{-(u+v)}.
    \end{split}
\ee
Recall that under the measure $\PP^{\bbb+v,\bbb+u}_{\lgrw}\lb\bL\rb$, for $1\leq j\leq N$,
$$e^{-(L_1{(j)}-L_1{(j-1)})}\sim\Ga(\b_j+v),\quad e^{-(L_2{(j)}-L_2{(j-1)})}\sim\Ga(\b_j+u).$$
Thus, by \eqref{eq:inequality in proof of finiteness LG}, $\mathbb{E}^{\bbb+v,\bbb+u}_{\lgrw}\lbE \lb \sum_{j=1}^N e^{-(L_2{(N)}-L_2{(j)})-L_1{(j-1)}}\rb^{-(u+v)}\rbE$ can be bounded above by a sum of product of positive moments of independent Gamma random variables with parameters $\a_j+u$ and $\a_j+v$. Such random variables have finite positive moments of all orders, thus  $\mathbb{E}_{\lgrw}^{\bbb,\bbb} \lbE V_{\lgg}^{u,v}(\bL)\rbE$ is finite. 
\end{proof}
 
We are now positioned to complete the proof of Theorem \ref{thm:main theorem LG} using real analytic continuation. As noted above, the proof of real analyticity in the boundary parameters is considerably trickier here since the measures in question are not power series anymore in those parameter. We state the desired real analyticity as Proposition \ref{Prop:real analytical LG} in the proof of Theorem \ref{thm:main theorem LG} and then devote Section \ref{sec:proof of real analyticity} to proving it via a combination of complex analytic tools and explicit estimates on transition densities that  figure into our Markov dynamics.

\begin{proof}[Proof of Theorem \ref{thm:main theorem LG}]
\label{sec:analyticcontinuationLG}
We recall that $\pU_{\lgg}^{\hP,\hQ}\lb\bL_1'|\bL_1\rb$ \eqref{eq:definition of first marginal operator sum over x LG} gives the transition probability density from $\bL_1$ to $\bL_1'$ for the log-gamma free energy recurrence (i.e., if started with initial condition $\bL_1$ along $\hP$, this is the probability that the recurrence produces $\bL_1'$ for the last passage times along $\hQ$ centered by the value on the left-boundary). Since we are presently only going to consider horizontal paths, we introduce a slight overload of our notation and write $\pU_{\lgg}^{\bbb,u,v}:=\pU_{\lgg}^{\hP,\tau_1\hP}$, where $\hP$ is a horizontal path with labels $\bbb=(\b_1,\dots,\b_N)$. 
%and $\tau_1\bbb=(\b_2,\dots,\b_N,\b_1)$ are labels on the translated path $\tau_1\hP$. 
Of course our notation $\pU_{\lgg}^{\hP,\tau_1\hP}$ hid the implicit dependence on the edge and boundary parameters which we have now made more explicit in this special horizontal case.
 
The stationarity we aim to prove in this theorem can be rewritten as:
For any $\bL'_1\in\RR^N$, we have
\be\label{eq:compatibility with probability measure after analytic continuatio LG}
    \int_{\RR^N}\pU_{\lgg}^{\bbb,u,v}\lb \bL'_1|\bL_1\rb\mathrm{P}^{\bbb,u,v}_{\stat\lgg} \lb \bL_1\rb\d\bL_1=\mathrm{P}^{\tau_1\bbb,u,v}_{\stat\lgg} \lb \bL'_1\rb.\ee
When $u+v>0$, taking $\hQ=\tau_1\hP$ in Theorem \ref{thm:stationary measure LG before sum over zero mode} and using the matching from Proposition \ref{prop:probability measure coincide LG}, implies \eqref{eq:compatibility with probability measure after analytic continuatio LG}.  
To prove \eqref{eq:compatibility with probability measure after analytic continuatio LG} holds without assuming $u+v>0$, we prove (let $\min(\bbb)=\min(\b_1,\ldots,\b_N)=\min(\a_1,\ldots,\a_N)$):
\begin{proposition}\label{Prop:real analytical LG}
Both sides of \eqref{eq:compatibility with probability measure after analytic continuatio LG} are real analytic functions of $u$ for $u\in(-\min(\bbb),\infty)$.    
\end{proposition}
This proposition will be proved in Section \ref{sec:proof of real analyticity}. 
Since we know that the equality \eqref{eq:compatibility with probability measure after analytic continuatio LG}  holds on the smaller interval $u\in(-\min(v,\a_1,\dots,\a_N),\infty)$, by the uniqueness of analytic continuation of real analytic functions (see, e.g. \cite[Corollary 1.2.6]{krantz2002primer}) it follows from Proposition \ref{Prop:real analytical LG} that the equality in  \eqref{eq:compatibility with probability measure after analytic continuatio LG}  extended to the larger interval  $u\in(-\min(\bbb),\infty)$. This proves the  stationarity in Theorem \ref{thm:main theorem LG} for all $u,v$ that satisfy \eqref{eq:condition for inhomogeneous log gamma}.

Finally to prove the uniqueness of the stationary measure in the homogeneous $\bbb=(\a,\ldots,\a)$ case, observe that as a function of $\bL_1$ and $\bL'_1$ in $\RR^N$,
$\pU_{\lgg}^{\bbb,u,v}\lb \bL'_1|\bL_1\rb$ is strictly positive. This can be seen directly from the strict positivity of the log-gamma density and the fact that $\pU_{\lgg}^{\bbb,u,v}\lb \bL'_1|\bL_1\rb$ encodes the transition probability density for the log-gamma recursion. 
This strict positivity implies that the Markov chain is irreducible. Thus by  \cite[Theorem 5.5]{benaim2022markov} the stationary measure is unique and the Markov chain is ergodic.
\end{proof}
 
\subsection{Proof of Proposition \ref{Prop:real analytical LG}}\label{sec:proof of real analyticity}   
In the proof we will use increment variables $X_i{(j)}:=L_i{(j)}-L_i{(j-1)}$ for $i=1,2$ and $1\leq j\leq N$, and the shorthand $\bX_i:=(X_i{(1)},\dots,X_i{(N)})$ for $i=1,2$.   
We also write $\bX:=(\bX_1,\bX_2)$.
The Jacobian determinant of the change of variables between $\bX_1$ and $\bL_1$ equals $1$, hence 
\be\label{eqn:XLcov}
\cP_{\lgg}^{\bbb,u,v}\lb \bX_1 \rb:=\mathrm{P}_{\stat\lgg}^{\bbb,u,v} \lb\bL_1\rb, \qquad \textrm{and}\qquad \cU_{\lgg}^{\bbb,u,v} (\bX'_1|\bX_1):=\pU_{\lgg}^{\bbb,u,v}\lb \bL'_1|\bL_1\rb
\ee
define probability densities and transition probability densities in these new variables and \eqref{eq:compatibility with probability measure after analytic continuatio LG} becomes
\be\label{eq:stationarity in X}
    \int_{\RR^N}\cU_{\lgg}^{\bbb,u,v}\lb \bX'_1|\bX_1\rb\cP^{\bbb,u,v}_{\lgg} \lb \bX_1\rb\d\bX_1=\cP^{\tau_1\bbb,u,v}_{\lgg} \lb \bX'_1\rb.
\ee
To prove Proposition \ref{Prop:real analytical LG} it therefore suffices to show real analyticity in $u\in(-\min(\bbb),\infty)$ of both sides of \eqref{eq:stationarity in X}. We actually prove a stronger result which requires a bit of notation. 
\begin{definition}\label{defn:class X of functions} 
For an open disk $U\subset\mathbb{C}$ and $m\in \ZZ_{\geq 1}$, a measurable function $U\times \RR^m\ni (u,\x)\mapsto g_u(\x)\in \CC$ (we use the notation $\x=(x_1,\ldots,x_m)$ here and below) is in $\X^m(U)$ if:
    \begin{enumerate}[wide, labelwidth=0pt, labelindent=0pt]
        \item \label{Property 1 class X} There exists $C,r>0$ such that, for all $(u,\x)\in U\times \RR^m$,
        $$|g_u(\x)|\leq C\exp\lb-r\sum_{i=1}^m|x_i|\rb.$$ 
        \item \label{Property 2 class X} For any fixed $\x\in\RR^m$, $u\mapsto g_u(\x)$ is holomorphic on $U$.
    \end{enumerate}
\end{definition}

\begin{proposition}\label{prop:conclusion of step 1}
For any $u_0\in(-\min(\bbb),\infty)$, there exists an open disk $ U\subset\CC$ containing $u_0$ such that $U\times \RR^N\ni (u,\bX_1)\mapsto \cP_{\lgg}^{\bbb,u,v}\lb \bX_1 \rb$
is in $\X^N(U)$.
\end{proposition}

For transition probability densities $\cU(\x'|\x)$ and $\widetilde{\cU}(\x'|\x)$ with $\x,\x'\in\RR^N$, we define their action on $\CC$-valued measurable functions $\phi(\x)$, $\x\in\RR^N$ and their composition as
\be\label{eq:cucomp}
(\cU\phi)(\x'):=\int_{\RR^N}\cU(\x'|\x)\phi(\x)\d\x,\qquad\textrm{and}\qquad
\widetilde{\cU}\circ \cU (\x''|\x):=\int_{\RR^N}\widetilde{\cU}(\x''|\x')\cU(\x'|\x)\d\x'.
\ee
One can observe that $\lb\widetilde{\cU}\circ \cU\rb\phi=\widetilde{\cU}\lb\cU\phi\rb$.
In particular, if $\phi$ is a real-valued probability density function, then  $\cU\phi$ is also a probability density function.

\begin{proposition}\label{prop:conclusion of step 3}
Assume $U\subset\CC$ is an open disk such that there is an $\ep>0$ so that $U\subset\left\{ z\in\CC: \Re z>\ep-\min(\bbb)\right\}$. 
Then the action of the transition probability density $\cU_{\lgg}^{\bbb,u,v}$ preserves $\X^N(U)$. In other words, if the function  $U\times \RR^N\ni (u,\x)\mapsto g_u(\x)$ is in $\X^N(U)$ then $U\times \RR^N\ni (u,\x)\mapsto (\cU_{\lgg}^{\bbb,u,v} g_u)(\x)$ is in $\X^N(U)$ as well.
%, i.e. if function $\phi_u(\x)$ belongs to $\X(U)$, then 
%$$(\cU_{\lgg}^{\bbb }\phi_u)(\x'):=\int_{\RR^N}\cU_{\lgg}^{\bbb }(\x'|\x)\phi_u(\x)\d\x$$
%also belongs to $\X(U)$.
\end{proposition}

These two propositions will be proved respectively in Section \ref{subsubsec:proof of proposition conclusion of step 1} and Section \ref{subsubsec:proof of proposition conclusion of step 3}.

\begin{proof}[Proof of Proposition \ref{Prop:real analytical LG}]
For each fixed $u_0\in(-\min(\bbb),\infty)$,
by Proposition \ref{prop:conclusion of step 1} we can choose an open disk $U\subset\CC$ containing $u_0$ such that
functions $U\times \RR^N\ni (u,\bX_1)\mapsto \cP_{\lgg}^{\bbb,u,v}\lb \bX_1 \rb$ and $U\times \RR^N\ni (u,\bX'_1)\mapsto\cP^{\tau_1\bbb,u,v}_{\lgg} \lb \bX'_1\rb$
both belong to $\X^N(U)$.
%the family of functions $\cP_{\lgg}^{\bbb,u,v}\lb \bX_1 \rb$ on $\bX_1\in\RR^N$ depending on parameter $u\in U$ is in class $\X(U)$. 
By Proposition \ref{prop:conclusion of step 3}, after possibly shrinking the disk $U$ to have minimal real part strictly exceeding 
$-\min(\bbb)$, we have that the function $U\times \RR^N\ni (u,\bX'_1)\mapsto\lb\cU_{\lgg}^{\bbb,u,v}\cP_{\lgg}^{\bbb,u,v}\rb\lb \bX_1' \rb$ is also in $\X^N(U)$. In particular this implies that both sides of \eqref{eq:stationarity in X} are holomorphic on $u\in U$ for any fixed $\bX_1'$. 
Since $u_0$ can be chosen through the whole interval $(-\min(\bbb),\infty)$, it follows from this conclusion that both sides of \eqref{eq:stationarity in X} are real analytic functions of $u$ on this interval when the variables $\bbb,v$ and $\bX'_1$ are fixed.
\end{proof}

In preparation for the proofs of Propositions \ref{prop:conclusion of step 1} and \ref{prop:conclusion of step 3} we record here a result that will be used to prove that integrals of certain types of holomorphic functions remain holomorphic. 
\begin{lemma}\label{lem:Morena argument}
    Suppose $U\subset \CC$ is an open disk, $m\in \ZZ_{\geq 1}$ and that $g_u(\x)$ is a function of $(u,\x)\in U\times\RR^m$ such that for any fixed $\x\in \RR^m$, $u\mapsto g_u(\x)$ is holomorphic in $U$. If 
    $\int_{\RR^N}\sup_{u\in U}|g_u(\x)|\d\x<\infty$  (denoting $\d\x=\d x_1\dots\d x_N$)
    then $h(u):=\int_{\RR^N}g_u(\x)\d\x$ is holomorphic in $U$ as well.  
\end{lemma} 
\begin{proof}
    By assumption  $s(\x):=\sup_{u\in U}|g_u(\x)|$ is an integrable function of $\x\in\RR^m$. Suppose $\lbb u_k\rbb_{k=1}^{\infty}$ is a sequence of points in $U$ converging to $u$. Then since $u\mapsto g_u(\x)$ is holomorphic in $U$ (and in particular, continuous), we have that $\lim_{k\rightarrow\infty}g_{u_k}(\x)=g_{u}(\x)$ for any fixed $\x\in\RR^m$, and moreover $|g_{u_k}(\x)|\leq s(\x)$ for all $k$. By the dominated convergence theorem, $\lim_{k\rightarrow\infty}\int_{\RR^m}g_{u_k}(\x)\d\x=\int_{\RR^m}g_{u}(\x)\d\x=:h(u)$. Since the sequence $\lbb u_k\rbb_{k=1}^{\infty}$ is arbitrary, we have proved that
    $h(u)$
    is continuous in $U$. For any triangular contour $\gamma$ contained in $U$, denoting its length by $|\gamma|$, we have
    $$\int_{\gamma}\int_{\RR^m}|g_{u}(\x)|\d\x\d u\leq |\gamma|\int_{\RR^m}\sup_{u\in U}|g_u(\x)|\d\x<\infty.$$
    By Fubini's theorem and the Cauchy integral theorem, we have
    $$\int_{\gamma}h(u)\d u= \int_{\gamma}\int_{\RR^m}g_{u}(\x)\d\x\d u=\int_{\RR^m}\int_{\gamma}g_{u}(\x)\d u\d\x=\int_{\RR^m}0=0.$$
    By Morera's theorem (Theorem 5.1 in Chapter 2 of \cite{stein2010complex}), $h(u)$ is a holomorphic function on $U$.
\end{proof}

\subsubsection{Proof of Proposition \ref{prop:conclusion of step 1}}
\label{subsubsec:proof of proposition conclusion of step 1}
Using \eqref{eqn:XLcov}, the formula in \eqref{eq:PstatLG inhomogeneous} for $\mathrm{P}_{\stat\lgg}^{\bbb,u,v} \lb\bL_1\rb$ is rewritten as
\be\label{eq:quotient expression of P in proof of main theorem LG}\cP_{\lgg}^{\bbb,u,v}\lb \bX_1 \rb
=\frac{\int_{\RR^{N}} H^{\bbb,u,v}(\bX)\d\bX_2}{\int_{\RR^{2N}}H^{\bbb,u,v}(\bX)\d\bX}\ee
where 
$$
H^{\bbb,u,v}(\bX):=\lb \sum_{j=1}^N e^{-\sum_{i=j+1}^NX_2(i)-\sum_{i=1}^{j-1}X_1(i)}\rb^{-(u+v)}\!\!\!\prod_{j=1}^N f_{\b_j+v}\lb X_1(j)\rb f_{\b_j+u}\lb X_2(j)\rb,\quad f_\theta(x):=e^{-\theta x-e^{-x}}.
$$
We will repeatedly use the following bound on $f_{\theta}$.
    \begin{lemma}\label{lem:bounds of log gamma density by exponential}
  For any $\theta>0$, let $C_{\theta}=\max\lb 1,e^{2\theta\ln(2\theta)-2\theta}\rb$. Then $f_{\theta}(x)\leq C_{\theta}e^{-\theta|x|}$ for all $x\in\mathbb{R}$.
\end{lemma}
\begin{proof}[Proof of Lemma \ref{lem:bounds of log gamma density by exponential}]
    Observe that $C_{\theta}\geq\exp\lb\max_{y\geq 0}(2\theta y-e^y)\rb$. For $x\geq 0$,  $f_{\theta}(x)=e^{-\theta x-e^{-x}}<e^{-\theta x}\leq C_{\theta}e^{-\theta|x|}$. For $x<0$, $C_{\theta}\geq e^{2\theta|x|-e^{|x|}}$ and hence $f_{\theta}(x)=e^{\theta|x|-e^{|x|}}\leq C_{\theta}e^{-\theta|x|}$.
\end{proof}

To control the ratio of integrals in \eqref{eq:quotient expression of P in proof of main theorem LG} we first prove a bound on the norm of
$H^{\bbb,u,v}\lb \bX\rb$. 
%We remark that this is the only place in the proof that we use the technical condition $\a_1,\dots,\a_N\in(0,\infty)$.

\begin{lemma}\label{lemma:bounds on initial weight by exponential}
    For any $u_0\in(-\min(\bbb),\infty)$, there exists an open disk $U\subset\CC$ containing $u_0$, and $C,r>0$ such that
\be\label{eq:exponential bound of weight}
    \left\vert H^{\bbb,u,v}\lb \bX\rb\right\vert\leq C\exp\bigg(-r\Big(\sum_{j=1}^N|X_1{(j)}|+\sum_{j=1}^N|X_2{(j)}|\Big)\bigg).\ee 
\end{lemma}
\begin{proof} Observe that, when $u$ can take complex values we have (where $\ru$ is the real part of $u$)
\be\label{eqn:normH}
|H^{\bbb,u,v}(\bX)|=\bigg(\sum_{j=1}^N e^{-\sum_{i=j+1}^NX_2(i)-\sum_{i=1}^{j-1}X_1(i)}\bigg)^{-\ru-v} \prod_{j=1}^N f_{\b_j+v}\lb X_1(j)\rb \prod_{j=1}^Nf_{\b_j+\ru}\lb X_2(j)\rb.
\ee
\begin{enumerate}[wide, labelwidth=0pt, labelindent=0pt]
    \item [(1)] When $\ru+v\leq 0$, using $(x+y)^a\leq \max(1,2^{a-1})(x^a+y^a)$ for any $x,y,a>0$ to expand the first term in \eqref{eqn:normH} (cf. \eqref{eq:inequality in proof of finiteness LG}), we obtain 
     $$
            \left\vert H^{\bbb,u,v}\lb \bX\rb\right\vert\leq  C\sum_{\ell=1}^N\prod_{j=1}^{\ell-1}f_{\b_j+v-\lb\Re u+v\rb}\lb X_1(j)\rb\prod_{j=\ell}^Nf_{\b_j+v}\lb X_1(j)\rb 
              \prod_{j=1}^\ell f_{\b_j+\Re u}\lb X_2(j)\rb\prod_{j=\ell+1}^Nf_{\b_j+\Re u-(\Re u+v)}\lb X_2(j)\rb.
     $$
    %\begin{multline*} 
     %       \left\vert H^{\bbb,u,v}\lb \bX\rb\right\vert\leq  C\sum_{\ell=1}^N\prod_{j=1}^{\ell-1}f_{\b_j+v-\lb\Re u+v\rb}\lb X_1(j)\rb\prod_{j=\ell}^Nf_{\b_j+v}\lb X_1(j)\rb \\
     %        \times\prod_{j=1}^\ell f_{\b_j+\Re u}\lb X_2(j)\rb\prod_{j=\ell+1}^Nf_{\b_j+\Re u-(\Re u+v)}\lb X_2(j)\rb.
    %\end{multline*}
    \item [(2)] When $\Re u+v> 0$ and $\Re u\leq 0$. 
    \begin{equation*}
        \begin{split}
           \left\vert H^{\bbb,u,v}\lb \bX\rb\right\vert
            \leq&\lb e^{-\sum_{i=1}^{N-1}X_1(i) }\rb^{-(\Re u+v)} \prod_{j=1}^N f_{\b_j+v}\lb X_1(j)\rb\prod_{j=1}^N f_{\b_j+\Re u}\lb X_2(j)\rb\\
            =& f_{\b_N+v}\lb X_1(N)\rb\prod_{j=1}^{N-1} f_{\b_j-\Re u}\lb X_1(j)\rb\prod_{j=1}^N f_{\b_j+\Re u}\lb X_2(j)\rb.
        \end{split}
    \end{equation*}
    \item [(3)] When $\Re u+v> 0$ and $v\leq 0$.
    \begin{equation*}
        \begin{split}
            \left\vert H^{\bbb,u,v}\lb \bX\rb\right\vert\leq&\lb e^{- \sum_{i=2}^NX_2(i)   }\rb^{-(\Re u+v)} \prod_{j=1}^N f_{\b_j+v}\lb X_1(j)\rb\prod_{j=1}^N f_{\b_j+\Re u}\lb X_2(j)\rb\\
            =& f_{\b_1+\Re u}\lb X_2{(1)}\rb\prod_{j=1}^N f_{\b_j+v}\lb X_1(j)\rb\prod_{j=2}^N f_{\b_j-v}\lb X_2(j)\rb.
        \end{split}
    \end{equation*}
    \item [(4)] When $\Re u>0$ and $v>0$.  We use
$$\lb \sum_{j=1}^N e^{-\sum_{i=j+1}^NX_2(i)-\sum_{i=1}^{j-1}X_1(i)}\rb^{-\ru-v}\leq 
\lb e^{\sum_{i=1}^{N-1}X_1(i) }\rb^v \lb e^{\sum_{i=2}^NX_2(i) }\rb^{\ru}.$$
Hence
$$\left\vert H^{\bbb,u,v}\lb \bX\rb\right\vert\leq f_{\b_N+v}\lb  X_1{(N)} \rb f_{\b_1+\Re u}\lb X_2{(1)}\rb 
            \prod_{j=1}^{N-1} f_{\b_j}\lb X_1{(j)} \rb\prod_{j=2}^N f_{\b_j}\lb X_2{(j)} \rb.$$
\end{enumerate}
In each of the above four cases we can then use Lemma \ref{lem:bounds of log gamma density by exponential} to arrive at the claimed bound \eqref{eq:exponential bound of weight}. In particular, the condition that $u_0\in(-\min(\bbb),\infty)$ is necessary here to ensure that we can find an open disk $U$ containing $u_0$ such that for all $u\in U$,  $\b_j+\Re u>0$ (as is necessary to apply Lemma \ref{lem:bounds of log gamma density by exponential} (recall $\theta$ there must be strictly positive).
\end{proof}

\begin{proof}[Proof of Proposition \ref{prop:conclusion of step 1}] 
Recall the formula \eqref{eq:quotient expression of P in proof of main theorem LG} for $\cP_{\lgg}^{\bbb,u,v}\lb \bX_1 \rb$.
The function $u\mapsto H^{\bbb,u,v}(\bX)$ is holomorphic in all of $\CC$ for any fixed choice of $\bbb,v$ and $\bX\in \RR^{2N}$ (the only place $u$ comes in is through the boundary weight which is an exponential in $u$). 
The denominator in \eqref{eq:quotient expression of P in proof of main theorem LG} is an integral of this holomorphic function over $\bX\in\RR^{2N}$. Lemma \ref{lemma:bounds on initial weight by exponential} in conjunction with  Lemma \ref{lem:Morena argument} imply that for any $u_0\in(-\min(\bbb),\infty)$, there exists an open disk $u_0\in U\subset \CC$ on which the denominator is holomorphic as a function of $u$. Moreover, when $u\in U$ is real, the denominator is strictly positive (it is an integral of strictly positive weights), thus we can ensure that the open disk $U$ was chosen small enough so that the modulus of the denominator is bounded away from 0. This means that if we can show that the numerator in \eqref{eq:quotient expression of P in proof of main theorem LG} satisfies Property \eqref{Property 1 class X} and \eqref{Property 2 class X}  in Definition \ref{defn:class X of functions}, then it will follow that $U\times\RR^N\ni (u,\bX_1)\mapsto \cP_{\lgg}^{\bbb,u,v}\lb \bX_1 \rb$ is in $\X^N(U)$ as desired. This is because the inverse of the denominator is holomorphic in $U$ (seeing the denominator is bounded from 0 and holomorphic itself).
Considering now the numerator in \eqref{eq:quotient expression of P in proof of main theorem LG}, Property \eqref{Property 1 class X} in Definition \ref{defn:class X of functions} follows from integrating the bound \eqref{eq:exponential bound of weight} in Lemma \ref{lemma:bounds on initial weight by exponential} over $\bX_2$.
Analyticity \eqref{Property 2 class X} in $u\in U$ for any fixed $\bX_1$ follows from Lemma \ref{lemma:bounds on initial weight by exponential} and Lemma \ref{lem:Morena argument}. 
\end{proof}

\subsubsection{Proof of Proposition \ref{prop:conclusion of step 3}}
\label{subsubsec:proof of proposition conclusion of step 3}
We will omit the superscript and subscript in the transition probability density $\cU_{\lgg}^{\bbb,u,v}$ and write $\cU:=\cU_{\lgg}^{\bbb,u,v}$.
By the definition of log-gamma polymer, we can decompose $\cU$ as
\be\label{eq:decomposition of operator U}\cU=\cUr\circ\cUb(N-1)\circ\dots\circ\cUb(1)\circ\cUl,\ee
(recall \eqref{eq:cucomp}) where $\cUl$ is the left boundary local move with parameter $u+\b_1$, $\cUb(j)$ is the bulk local move with parameter $\b_1+\b_{j+1}$, $1\leq j\leq N-1$, and $\cUr$ is the right boundary local move with 
parameter $v+\b_1$.

The left boundary transition density $\cUl$ replaces $X_1{(1)}$ by an independent $-\log\Ga^{-1}(u+\b_1)$ random variable and leaving other factors unchanged. In terms of  its action on functions (recall \eqref{eq:cucomp}),
       \be\label{eq:density transition left boundary operator}
       \lb\cUl\phi\rb\lb x_1',x_2,\dots,x_N\rb= \frac{1}{\Gamma(u+\b_1)} f_{u+\b_1}(-x_1')\int_{\RR}\phi(x_1,\dots,x_N)\d x_1.\ee
       
    Similarly, the right boundary transition density $\cUr$ replaces $X_1{(N)}$ by an independent $\log\Ga^{-1}(v+\b_1)$ random variable and leaving other factors unchanged. In terms of its action on functions,  
       \be\label{eq:density transition right boundary operator}
       \lb\cUr\phi\rb\lb x_1,x_2,\dots,x'_N\rb= \frac{1}{\Gamma(v+\b_1)} f_{v+\b_1}\lb x_N'\rb\int_{\RR}\phi(x_1,\dots,x_N)\d x_N.\ee
       
For $1\leq j\leq N-1$, the bulk transition density $\cUb(j)$ only acts on $X_1{(j)}$ and $X_1{(j+1)}$ and leaves other factors unchanged. By the recurrence relation we have that $e^{L_1'{(j)}}=\vo\lb e^{L_1{(j-1)}}+e^{L_1{(j+1)}}\rb$ where $\vo\sim\Ga^{-1}(\b_1+\b_{j+1})$.
Hence $e^{X_1'{(j)}}=\vo\lb 1+e^{X_1{(j)}+X_1{(j+1)}}\rb$, or in other words 
\be\label{eq:transform one}
X_1'{(j)}=\log\vo+\log\lb 1+e^{X_1{(j)}+X_1{(j+1)}}\rb.\ee
Since $L'{(j-1)}=L{(j-1)}$ and $L'{(j+1)}=L{(j+1)}$,
we also have \be\label{eq:transform two}X_1'{(j+1)}=X_1{(j)}+X_1{(j+1)}-X_1'{(j)}.\ee
The transformation $\lb X_1{(j)},X_1{(j+1)}\rb\mapsto\lb X_1'{(j)},X_1'{(j+1)}\rb$ by \eqref{eq:transform one} and \eqref{eq:transform two} can be written in two steps:
$$\lb X_1{(j)},X_1{(j+1)}\rb\mapsto\lb X_1''{(j)},X_1''{(j+1)}\rb\mapsto\lb X_1'{(j)},X_1'{(j+1)}\rb,$$ 
where the two steps are given by 
\begin{align*}
\big(X_1''{(j)},X_1''{(j+1)} \big)&=\big(X_1{(j)}+X_1{(j+1)}, \log\vo\big),\\
\big(X_1'{(j)},X_1'{(j+1)}\big) &=\Big( X_1''{(j+1)}+\log\big(1+e^{X_1''{(j)}}\big),X_1''{(j)}-X_1''{(j+1)}-\log\big( 1+e^{X_1''{(j)}}\big)\Big).
\end{align*}
%and the second step
%$$X_1'{(j)}=X_1''{(j+1)}+\log\lb 1+e^{X_1''{(j)}}\rb,\quad X_1'{(j+1)}=X_1''{(j)}-X_1''{(j+1)}-\log\lb 1+e^{X_1''{(j)}}\rb.$$
Thus $\cUb(j)$ is the composition of two transition operators $\cUb(j)=\cU^2(j)\circ\mathscr \cU^1(j)$. The first acts as 
\be\label{eq:expression for U1}\lb\cU^1(j)\phi\rb\lb x_1,\dots,x_j'',x_{j+1}'',\dots,x_N\rb=\frac{1}{\Gamma(\b_1+\b_{j+1})} f_{\b_1+\b_{j+1}}\lb x_{j+1}''\rb\int_{\RR}\phi\lb x_1,\dots,x_j''-y,y,\dots,x_N\rb\d y.\ee
The second operator $\cU^2(j)$  acts by applying a coordinate change whereby
$$x_j'=x_{j+1}''+\log\lb1+e^{x_j''}\rb,\quad x_{j+1}'=x_j''-x_{j+1}''-\log\lb1+e^{x_j''}\rb.$$
The inverse of this change of coordinates is given by
$$x_j''=x_j'+x_{j+1}',\quad x_{j+1}''=x_j'-\log\lb1+e^{x_j'+x_{j+1}'}\rb,$$
and the Jacobian $\det\begin{bmatrix}\frac{\partial x_k''}{\partial x_l'}\end{bmatrix}_{k,l\in\{j,j+1\}}=-1$. 
Therefore $\cU^2(j)$ acts on functions $\psi$ as
\begin{equation}\label{eqn:U2action}
    \begin{split}
        \lb\cU^2(j)\psi\rb\lb x_1,\dots,x_j',x_{j+1}',\dots,x_N\rb
&=\psi\lb x_1,\dots,x_j'',x_{j+1}'',\dots,x_N\rb\left|\det\begin{bmatrix}\frac{\partial x_k''}{\partial x_l'}\end{bmatrix}_{k,l\in\{j,j+1\}}\right|\\
&=\psi\lb x_1,\dots,x_j'+x_{j+1}',x_j'-\log\lb1+e^{x_j'+x_{j+1}'}\rb,\ldots,x_N\rb.
    \end{split}
\end{equation}
Therefore, using \eqref{eqn:U2action} and then \eqref{eq:expression for U1} we find that $\cUb(j)\phi$ acts on functions $\phi$ by 
\be   \label{eq:density transition bulk operator}
       \begin{split}
\lb\cUb(j)\phi\rb\lb x_1,\dots,x'_j,x'_{j+1},\dots,x_N\rb
=&\lb\cU^2(j)\lb\cU^1(j)\phi\rb\rb\lb x_1,\dots,x'_j,x'_{j+1},\dots,x_N\rb\\
           =&\lb\cU^1(j)\phi\rb\lb x_1,\dots,x_j'+x_{j+1}',x_j'-\log\lb1+e^{x_j'+x_{j+1}'}\rb,\dots,x_N\rb\\
           =&\frac{1}{\Gamma(\b_1+\b_{j+1})} f_{\b_1+\b_{j+1}}\lb x'_j-\log\lb1+e^{x'_j+x'_{j+1}}\rb\rb\\
           &\times\int_{\RR}\phi\lb x_1,\dots,x'_j+x'_{j+1}-y,y,\dots,x_N\rb\d y.
       \end{split}
\ee

We observe from formulas \eqref{eq:density transition left boundary operator}, \eqref{eq:density transition right boundary operator} and \eqref{eq:density transition bulk operator} that the actions of transition probability densities $\cUl$, $\cUb(j)$ and $\cUr$ on a probability density function $\phi$ can be factorized as a sequence of basic operations starting from $\phi$, involving integration, change of variables, and multiplication with the basic function $f_\theta$. %The actions of transition probability densities $\cUl$, $\cUb(j)$ and $\cUr$ on $\mathbb{C}$-valued functions $\phi_u$ depending on a complex parameter $u\in U\subset\CC$ can be given by the same formulas, provided that the integrals in their RHS make sense.
The next two lemmas prove that these operators preserve the function spaces $\X^m(U)$ from Definition \ref{defn:class X of functions}.

\begin{lemma}\label{lemma:operations preserving class X}
    We have the following properties of the function spaces $\X^m(U)$ from Definition \ref{defn:class X of functions}. As a convention we will simply say that a function $g_u(\x)$ is in $\X^m(U)$ and likewise when we take some integrals, we will treat the result as a function of $u$ and the remaining variables. 
    \begin{enumerate}[wide, labelwidth=0pt, labelindent=0pt]
        \item \label{class X preserved under integration}
If $g_u(x_1,\dots,x_m)$ is in $\X^m(U)$, then $\int_{\RR}g_u(x_1,\dots,x_m)\d x_1$ is in $\X^{m-1}(U)$.
        \item\label{class X preserved under convolution}
    If $g_u(x_1,\dots,x_m)$ is in $\X^m(U)$ then $\int_{\RR}g_u(x_2-y,y,x_3,\dots,x_m)\d y$ is in $\X^{m-1}(U)$.
        \item\label{class X preserved under multiplication}
If $g_u(x_2,\dots,x_m)$ is in $\X^{m-1}(U)$ and $r_u(x_1)$ is in $\X^{1}(U)$ then
    $r_{u}(x_1)g_u(x_2,\dots,x_m)$ is in $\X^{m}(U)$.
    \item\label{f_u in class X}
    Recall $f_\theta(x):=e^{-\theta x-e^{-x}}$. 
    Assume that $U$ is an open disk such that there are $c\in\CC$ and $\ep>0$ satisfying $U\subset\left\{ z\in\CC: \Re z>\ep-c\right\}$. Assume $\Re\gamma>0$.
    Then $\frac{1}{\Gamma(u+c)}f_{u+c}(-x)$ and $\frac{1}{\Gamma(\gamma)}f_\gamma(x)$ on $\RR$ are in $\X^1(U)$.
    \end{enumerate}
\end{lemma}
\begin{proof}
    We first verify Property \eqref{Property 1 class X} from Definition \ref{defn:class X of functions} in the four cases of the lemma. For Case \eqref{class X preserved under integration} of the lemma, this follows by integrating the bound from Property \eqref{Property 1 class X} for $g_u$ over $x_1\in\RR$. Case \eqref{class X preserved under convolution} follows by observing that the convolution of $e^{-r|x|}$ with itself (for $r>0$) is bounded below by $Ce^{-\frac{r}{2}|x|}$ for some $C>0$. Case \eqref{class X preserved under multiplication} follows by multiplying the bounds from Property \eqref{Property 1 class X} for $g_u$ and $r_u$. Case \eqref{f_u in class X} follows from Lemma \ref{lem:bounds of log gamma density by exponential} in view of $|f_{\theta}(x)|=f_{\Re\theta}(x)$.
    Now we verify Property \eqref{Property 2 class X} from Definition \ref{defn:class X of functions} in the four cases of the lemma. Cases \eqref{class X preserved under integration} and \eqref{class X preserved under convolution} follow from integrating the bound in Property \eqref{Property 1 class X} for $\sup_{u\in U}|g_u|$ and using Lemma \ref{lem:Morena argument}. Case \eqref{class X preserved under multiplication} follows from multiplication of holomorphic functions being holomorphic. Case \eqref{f_u in class X} follows from the holomorphicity of Gamma functions.
\end{proof}
\begin{lemma}
    \label{lemma:bulk operator preserve class X}
    If $g_u(\x)$ is in $\X^m(U)$ and $\theta>0$, then 
    $$p_u(\x):=f_{\theta}\lb x_1-\log\lb 1+e^{x_1+x_2}\rb\rb\int_{\RR}g_u(x_1+x_2-y,y,x_3,\dots,x_m)\d y$$
    is in class $\X^m(U)$.
\end{lemma}
\begin{proof}
Holomorphicity of $p_u(\x)$ (Property \eqref{Property 2 class X} from Definition \ref{defn:class X of functions}) follows from Case \eqref{class X preserved under convolution} in Lemma \ref{lemma:operations preserving class X}. 
Now we verify Property \eqref{Property 1 class X} from Definition \ref{defn:class X of functions}. Again, by Case \eqref{class X preserved under convolution} in Lemma \ref{lemma:operations preserving class X} we have
    $$\left\vert\int_{\RR}g_u(x_1+x_2-y,y,x_3,\dots,x_m)\d y\right\vert\leq Ce^{-r\lb|x_1+x_2|+\sum_{j=3}^m|x_j|\rb}$$
    for some $C,r>0$. Property \eqref{Property 1 class X} for $p_u$ will follow if for any $\theta,r>0$ there exists $C,\delta>0$ such that
        $$f_{\theta}\lb x-\log\lb 1+e^{x+y}\rb\rb e^{-r|x+y|}\leq Ce^{-\delta\lb|x|+|y|\rb}$$
        for any $x,y\in\RR$.
    Thus, to conclude this proof we demonstrate this claim.
        Choose $0<s<\min(\theta,r)$. Then
     \begin{equation*}
         \begin{split}
             f_\theta\lb x-\log(1+e^{x+y})\rb e^{-r|x+y|}&=e^{-\theta x}(1+e^{x+y})^\theta e^{-e^{-x}}e^{-e^y}e^{-r|x+y|}\\
             &\leq e^{-\theta x}(1+e^{x+y})^\theta e^{-e^{-x}}e^{-e^y}e^{-s|x+y|}.
         \end{split}
     \end{equation*}
If $x+y\geq 0$, then $e^{x+y}>1$ and $1+e^{x+y} \leq 2 e^{x+y}$. Thus,
\begin{equation*}
    \begin{split}
        e^{-\theta x}(1+e^{x+y})^\theta e^{-e^{-x}}e^{-e^y}e^{-s|x+y|}&\leq 2^\theta e^{-\theta x}e^{\theta (x+y)}e^{-e^{-x}}e^{-e^y}e^{-s(x+y)}\\
        &\leq 2^\theta e^{-sx-e^{-x}}e^{(\theta-s)y-e^y}=2^\theta f_s(x)f_{\theta-s}(-y).
    \end{split}
\end{equation*}
If $x+y<0$, then $e^{x+y}\in(0,1)$ and $1+e^{x+y}\leq 2$. Thus,
\begin{equation*}
    \begin{split}
        e^{-\theta x}(1+e^{x+y})^\theta e^{-e^{-x}}e^{-e^y}e^{-s|x+y|}&\leq 2^\theta e^{-\theta x}e^{-e^{-x}}e^{-e^y}e^{s(x+y)}\\
        &=2^\theta e^{-(\theta-s)x-e^{-x}}e^{sy-e^y}=2^\theta f_{\theta-s}(x)f_s(-y).
    \end{split}
\end{equation*}
Setting $\delta=\min(\theta-s,s)>0$ and using Lemma \ref{lem:bounds of log gamma density by exponential} in both cases we conclude the proof of the claim.
\end{proof}

\begin{proof}[Proof of Proposition \ref{prop:conclusion of step 3}]
By the decomposition \eqref{eq:decomposition of operator U}, we only need to prove that the actions of $\cUl$, $\cUb(j)$ and $\cUr$ preserve the function space $\X^N(U)$. Note that $\X^N(U)$ is invariant under permuting coordinates in $\x$, hence the results of Lemmas \ref{lemma:operations preserving class X} and \ref{lemma:bulk operator preserve class X} can be applied to other coordinates than the first ones. Thus, in light of \eqref{eq:density transition left boundary operator} and \eqref{eq:density transition right boundary operator},  $\cUl$ and $\cUr$ preserve $\X^N(U)$ due to Cases \eqref{class X preserved under integration}, \eqref{f_u in class X} and  \eqref{class X preserved under multiplication} in Lemma \ref{lemma:operations preserving class X}. In light of \eqref{eq:density transition bulk operator}, $\cUb(j)$, $1\leq j\leq N-1$, preserve  $\X^N(U)$ due to Lemma \ref{lemma:bulk operator preserve class X}.
\end{proof}

\subsection{Operator product ansatz}
\label{sec:MPA LG}

As in Section \ref{sec:MPA LPP}, the probability density $\mathrm{P}_{\lgg}^{\hP} \lb \bL_1\rb$ can be written as a matrix product, or more precisely a product of operators. We will use a  change of variables similar as  in Section \ref{sec:MPA LPP}: for $1\leq i\leq N$ let  $\kappa_i=\vert \lambda_1^{(i)}-\la_1^{(i-1)}\vert$ when $\mathbf p_i-\mathbf p_{i-1}=\rightarrow$, $\kappa_i=\vert \lambda_1^{(i-1)}-\la_1^{(i)}\vert$ when $\mathbf p_i-\mathbf p_{i-1}=\downarrow$,  and for $0\leq 1\leq N$ let $r_i=\la_1^{(i)}-\la_2^{(i)}$. 
%\begin{definition}
For each $\kappa\in \mathbb R$ and $\alpha>0$, we define integral operators   $\mathbf O_{\kappa}^{\rightarrow}[\alpha]$ and $ \mathbf O_{\kappa}^{\downarrow}[\alpha]$,  acting on $C(\mathbb R)$,   with kernels given by, for $r,r'\in \mathbb R$,   
	\begin{equation*}
		\mathbf O_{\kappa}^{\rightarrow}[\alpha](r,r') =\wt_{\lgg}\left.\lb\raisebox{-30pt}{\begin{tikzpicture}
				\draw[thick] (0,0)--(0.6,0.6);
				\draw[thick] (0,-1.2)--(0.6,-0.6);
				\draw[dashed] (0,0)--(0.6,-0.6);
				\node[right] at (0.5,0.55) {\small $\la'_1$};
				\node[right] at (0.6,-0.6) {\small  $\la'_2$};
				\node[left] at (0,0) {\small   $\la_1$};
				\node[left] at (0,-1.2) {\small  $\la_2$};
				\node[above] at (0.3,0.3) {\textcolor{red}{\small   $\alpha$}};
				\node[above] at (0.3,-0.9) {\textcolor{red}{\small  $\alpha$}};
		\end{tikzpicture}}\rb\right\vert_{{\footnotesize \begin{matrix}
					\la_1-\la_2=r\\ \la'_1-\la'_2=r'\\ \la'_1-\la_1=\kappa
		\end{matrix}}}  =  e^{-\alpha(r+2\kappa-r')-e^{-\kappa}-e^{-(r'-\kappa)}-e^{-(r+\kappa-r')}},
	%	\label{eq:defMatrix M right}
	\end{equation*} 
	and 
	\begin{equation*}\mathbf O_{\kappa}^{\downarrow}[\alpha](r,r') =\wt_{\lgg}\left. \lb\raisebox{-30pt}{\begin{tikzpicture}
				\draw[thick] (-0.6,0.6)--(0,0);
				\draw[thick] (-0.6,-0.6)--(0,-1.2);
				\draw[dashed] (-0.6,-0.6)--(0,0);
				\node[left] at (-0.5,0.55) {\small   $\la_1$};
				\node[left] at (-0.5,-0.65) {\small  $\la_2$};
				\node[above] at (-0.3,0.3) {\textcolor{red}{\small  $\alpha$}};
				\node[above] at (-0.3,-0.9) {\textcolor{red}{\small  $\alpha$}};
				\node[right] at (0,0) {\small  $\la'_1$};
				\node[right] at (0,-1.2) {\small  $\la'_2$};
		\end{tikzpicture}}\rb\right\vert_{{\footnotesize \begin{matrix}
					\la_1-\la_2=r\\ \la'_1-\la'_2=r'\\ \la_1-\la'_1=\kappa
		\end{matrix}}} = \mathbf O_{\kappa}^{\rightarrow}[\alpha](r',r).
	%	\label{eq:defMatrix M down}
	\end{equation*} 
	We also define functions $\mathbf w(r)=e^{-u r}$ and $\mathbf v(r)=e^{-vr}$. 
%	\label{def:matrices and vectors LG}
%\end{definition}
This allows to write, for any path $\hP$ with vertices $\left(\p_i\right)_{0\leq i\leq N}$ and edge labels $\boldsymbol \beta$, 
$$ \wt_{\lgg}^{\PiP}(\bl) = \mathbf w(r_0) \left( \prod_{i=1}^N \mathbf O^{\p_i-\p_{i-1}}_{\kappa_i}[\beta_i](r_{i-1}, r_i)  \right) \mathbf v(r_N).$$ 
Thus, $\mathrm{P}_{\lgg}^{\hP}$ may be written in operator product form as 
\begin{equation}
	\mathrm{P}_{\lgg}^{\hP} \lb \bL_1\rb= \frac{1}{Z_{\rm LG}} \left\langle \mathbf w,  \left( \prod_{i=1}^N \mathbf O^{\p_i-\p_{i-1}}_{\kappa_i}[\beta_i] \right)\mathbf v \right\rangle, 
	\label{eq:MPAformLG}
\end{equation}
where $\kappa_i=\pm\left(   \bL_i-\bL_{i-1}\right)$ according to the direction of $\p_i-\p_{i-1}$.  
Proposition \ref{prop:finiteness of weights LG} ensures that the R.H.S. of \eqref{eq:MPAformLG} is well defined as long as $u+v>0$. 
%As in Section \ref{sec:MPA LPP}, we expect that an analogue of a matrix product ansatz can be proven,  similar to Proposition \ref{prop:MPA LPP}, and that the operators $\mathbf O_{\kappa}$ defined above provide a representation of the appropriate quadratic algebra. 

As in Section \ref{sec:MPA LPP}, such an operator product is stationary if the operators $\mathbf O_{\kappa}$ satisfy some relations. 
\begin{proposition}\label{prop:MPA LG}
	Assume that a family of probability densities  $\mathrm{P}^{\hP}$ on $\mathbb R^N$, indexed by down right-paths $\hP=(\mathbf p_0,  \dots, \mathbf p_N)$ with edge labels $\boldsymbol \beta$, takes the form 
	\begin{equation}
		\mathrm{P}^{\hP} \lb \kappa_1, \dots, \kappa_N\rb= \frac{1}{Z} \left\langle \overline{\mathbf w},  \left( \prod_{i=1}^N \overline{\mathbf O}^{\p_i-\p_{i-1}}_{x_i}[\beta_i] \right)\overline{\mathbf v}\right\rangle, 
		\label{eq:MPAformLG2}
	\end{equation}
	where $\overline{\mathbf w}, \overline{\mathbf v}$ are elements of some inner product space and $\overline{\mathbf O}^{\rightarrow}_{\kappa}[\beta_i], \overline{\mathbf O}^{\downarrow}_{\kappa_i}[\beta_i]$ are operators acting on that space. Then, the density $\mathrm{P}^{\hP}$ is stationary for the log-gamma  dynamics, in the sense of \eqref{eq:compatibility of probability measure with LG}, if the following commutation relations hold for all $\kappa, \gamma\in \mathbb R$ (recall that $f_\theta(x):=e^{-\theta x-e^{-x}}$): 
	\begin{subequations}
		\label{eq:MPA LG}
		\begin{align}			
			\overline{\mathbf O}^{\rightarrow}_{\kappa}[\alpha]\overline{\mathbf O}^{\downarrow}_{\gamma}[\beta] &= f_{\alpha+\beta}(-\log(e^{-\kappa}+e^{-\gamma})) \int_{\mathbb R}dz 	\overline{\mathbf O}^{\downarrow}_{z+\gamma-\kappa}[\beta]\overline{\mathbf O}^{\rightarrow}_{z}[\alpha],\label{eq:MPAbulk LG}\\ 
			\overline{\mathbf w}^t \overline{\mathbf O}^{\downarrow}_{\kappa}[\alpha] &= f_{\alpha+u}(\kappa)  \int_{\mathbb R}dz  \overline{\mathbf w}^t \overline{\mathbf O}^{\rightarrow}_{z}[\alpha],\label{eq:MPAleft LG}\\ 
			\overline{\mathbf O}^{\rightarrow}_{\kappa}[\alpha]\overline{\mathbf v} &=f_{\alpha+v}(\kappa) \int_{\mathbb R}dz \overline{\mathbf O}^{\downarrow}_z[\alpha]\overline{\mathbf v}.\label{eq:MPAright LG}
		\end{align}
	\end{subequations}
\end{proposition}
\begin{proof}
	As in the proof of Proposition \ref{prop:MPA LPP} it suffices to check that $\mathrm{P}^{\hP}$ pushes forward to $\mathrm{P}^{\hQ}$ under the action of local operators  $\rmUblg, \rmUllg, \rmUrlg$ from Lemma \ref{lem:marginal operator LG}, where $\hQ$ is the corresponding updated path. Under the appropriate changes of variables, \eqref{eq:MPAbulk LG} implies the invariance with respect to $\rmUblg$. Indeed, if the local move transforms the point $\mathbf p_i$ on the path $\hP$ to a new point $\mathbf q_i$, the increments $\kappa_i=\kappa$ and $\kappa_{i+1}=\gamma$ may arise on $\hQ$  only if the weight $\o_{\mathbf q_i}=1/(e^{-\kappa}+e^{-\gamma})$ (which arises with density $f_{\beta_i+\beta_{i+1}}(-\log(e^{-\kappa}+e^{-\gamma}))$), while the increments on $\hP$ before the local move can be anything such that $\kappa_{i+1}-\kappa_i=x-y$. Likewise, \eqref{eq:MPAleft LG} implies invariance with respect to $\rmUllg$, and \eqref{eq:MPAright LG} implies invariance with respect to $\rmUrlg$.
\end{proof}
Again, as in Section \ref{sec:MPA LPP}, it can be checked that the operators $\mathbf O_{\kappa}$ defined above provide a representation of the quadratic algebra defined by the relations \eqref{eq:MPA LG}.

\section{Stationary measure of the open Kardar-Parisi-Zhang equation}
\label{sec:stationary measure KPZ}
\subsection{The open KPZ equation}
\label{subsec:The open KPZ equation}
The open Kardar-Parisi-Zhang (KPZ) equation \eqref{eq:KPZ} on an interval (Section \ref{subsec: stationary measure open KPZ intro}) is rigorously defined using the following Hopf-Cole solution.
\begin{definition}[Hopf-Cole solution to the open-KPZ equation \cite{corwin2016open}]
\label{def:Hopf-Cole transform}
Let $C(E,F)$ denote the space of continuous functions from $E$ to $F$, equipped with the uniform topology. Fix $u,v\in \mathbb R$. A stochastic process $h(t,x)\in C(\mathbb R_{\geq 0}\times [0,L], \mathbb R)$  is a solution to \eqref{eq:KPZ} with initial condition $h(0,x)= h_0(x)$ if $h(t,x)=\log Z(t,x)$ where $Z(t,x)$ is the solution of the  multiplicative noise stochastic heat equation
\begin{equation}
\begin{cases} 
    \partial_t Z(t,x) = \frac{1}{2} \partial_{xx} Z(t,x) +  Z(t,x) \xi(t,x), \\ 
    \partial_x Z(t,x)\Big\vert_{x=0} = \mu Z(t,0), \\ 
    \partial_x h(t,x)\Big\vert_{x=L} = \nu Z(t,0), 
    \end{cases}
    \tag{SHE$_{u,v}$}
    \label{eq:SHE}
\end{equation}
with $\mu=u-1/2$, $\nu=-(v-1/2)$ and  initial condition $Z(0,x)=e^{h_0(x)}$.  A stochastic process $Z(t,x)\in C(\mathbb R_{\geq 0}\times [0,L], \mathbb R_{>0})$ is a solution to \eqref{eq:SHE} if for all $t> 0$ and $x\in [0,L]$, it is measurable with respect to the filtration of the space-time white noise  $\xi$ generated up to and including time $t$ and if it satisfies
\begin{equation}
    Z(t,x) = \int_{0}^L p_{\mu, \nu}(0,y; t,x)Z(0,y) \d y + \int_0^t \int_0^L  p_{\mu, \nu} (s,y; t,x) Z(s,y)\xi(\d s, \d y),
\end{equation}
where the last integral is a stochastic integral in time in the Itô sense, and $p_{\mu, \nu}(s,y; t,x)$ is the heat kernel on $[0,L]$ with Robin type boundary condition, that is the unique solution to 
\begin{equation}\label{eq:heat kernel}
    \begin{cases}
    \partial_t p_{\mu, \nu}(s,y; t,x) = \frac 1 2 \partial_{xx} p_{\mu, \nu}(s,y; t,x),\\
    \partial_x p_{\mu, \nu}(s,y; t,x) \Big\vert_{x=0}=\mu p_{\mu, \nu}(s,y; t,0) \text{ for all }s<t, y\in [0,L],\\ 
        \partial_x p_{\mu, \nu}(s,y; t,x) \Big\vert_{x=L}=\nu p_{\mu, \nu}(s,y; t,L) \text{ for all }s<t, y\in [0,L],\\
        \lim_{t\to s}  p_{\mu, \nu}(s,y; t,x)=\delta(x-y) \text{ in the weak sense on }\mathbb L^2([0,L]). 
    \end{cases}
\end{equation}
\end{definition}
\cite[Proposition 2.7]{corwin2016open}, \cite[Proposition 4.2]{parekh2017kpz} show that \eqref{eq:SHE} admits a unique  solution (subject to a finite second moment condition) that is almost surely positive so that its logarithm is finite.  
\subsection{From discrete to continuous polymers}
\label{subsec:From discrete to continuous polymers}
As discussed in Section \ref{subsec: stationary measure open KPZ intro}, under intermediate disorder scaling the partition function of discrete directed polymer models converge to solutions of the stochastic heat equation. This was demonstrated for full-space polymers in \cite{alberts2014intermediate}. This was extended to the half-space setting in \cite{wu2018intermediate,parekh2019positive, barraquand2022stationary}. Based on these works it is clear how to formulate the convergence of the partition function of the log-gamma polymer model on a strip to a solution of \eqref{eq:SHE}.

Let  $z(n,m)$ be the log-gamma polymer partition function defined in \eqref{eq:defloggammapartitionfunction} with some initial condition $z_0$ on the horizontal path $\mathcal P_h$. We fix $L>0$ and scale the width of the strip and the bulk parameter as well as defined the rescaled free energy in terms of a parameter $\nn>0$ going to zero as 
\begin{equation}
\label{eq:scalings}
\alpha=\frac{1}{2}+\nn^{-1},\qquad N=\nn^{-1} L,\qquad  h^{(\nn)}(t,x):= \log\left( \nn^{-\nn^{-2} t+\nn^{-1} x}\,\cdot\, z\left(\frac{\nn^{-2} t}{2}+\nn^{-1} x, \frac{\nn^{-1} t }{2}\right) \right)
\end{equation} 
This defines $h^{(\nn)}(t,x)$ for $t$ and $x$ so that $\nn^{-2}t/2\in \mathbb Z_{\geq 0}$ and $\nn^{-1} x\in \llbracket 0,\nn^{-1} L\rrbracket$, and we extend $h^{(\nn)}$  to all $(t,x)\in \mathbb R_{\geq 0}$ by linear interpolation. 
We have used scalings very similar with those used in \cite[Theorem 4.1]{barraquand2022stationary}, except that $\nn$ is denoted $n^{-1/2}$ there.
We assume that the rescaled initial condition  $h_0^{(\nn)}(x):= \log\left( \nn^{-\nn^{-1} x}z_0(\nn^{-1} x)\right)$ satisfies 
$\sup_{x\in [0,L]} \mathbb E\left[ e^{2h_0^{(\nn)}(x)}\right]<\infty$
uniformly in $\nn$,  and that $h_0^{(\nn)}$  weakly  converges as $\nn\to 0$  to some continuous process $h_0$ in the space  $C([0,L],\mathbb R)$. Then we expect the following. 
\begin{conjecture} [Space-time process convergence to open KPZ]
For all $t>0$, $u,v\in \mathbb R$, as $\nn\to 0$,
\begin{equation}
h^{(\nn)}(t,x) \Rightarrow h(t,x)
    \label{eq:conjectured convergence}
\end{equation}
in the space $C([0,L],\mathbb R)$, where $h(t,x)$ is the solution to \eqref{eq:KPZ} with initial condition $h_0$. 
\label{conj:convergence}
\end{conjecture}

The proof of this conjecture should follow from the general framework of \cite{alberts2014intermediate} in full-space and its extensions \cite{wu2018intermediate,parekh2019positive,barraquand2022stationary} in half-space (see the discussion in Section \ref{subsec: stationary measure open KPZ intro} in the introduction). The general approach is to rewrite the partition function in terms of a discrete chaos series and then show convergence to the continuum chaos series that defines the solution to \eqref{eq:SHE}. The discrete chaos series involves discrete heat kernels for random walks subject to boundary conditions on the strip. That convergence requires more than just the (fairly classical) point-wise convergence of kernels to their continuum Brownian limits. Rather, it requires temporal and spatial regularity estimates uniform in $N$. Such sharp estimates should be accessible via discrete time modifications of the continuous time random walk estimates in \cite{corwin2016open,parekh2017kpz}. Since these techniques are rather technically involved and orthogonal to the integrable focus of this paper, we will leave the proof of Conjecture \ref{conj:convergence} to subsequent work.

 We will see in Proposition \ref{prop:convergencestationarymeasures} that the stationary measure $\PP_{\stat\lgg}^{\alpha, u,v}$ of the log-gamma free energy on a strip defined in terms of reweighted log-gamma random walks (Definition \ref{def:rescaled random walks}) converges under intermediate disorder scaling to its continuous analogue defined in terms of reweighted Brownian motions (Definition \ref{def:KPZstat}). We first record an alternative expression of $\PP_{\stat\lgg}^{\alpha, u,v}$ that will be useful in the proof of  Proposition \ref{prop:convergencestationarymeasures}.

\begin{lemma}\label{lem:alternative reweighted random walk}
    When $\alpha-v>0$, we have:
    \begin{equation} \PP_{\stat\lgg}^{\alpha, u,v}(\bL)  =\frac{\Gamma(\alpha-v)^N\Gamma(\alpha+v)^N}{\mathcal Z_{\lgg}^{\alpha, u,v}\Gamma(\alpha)^{2N}} \lb\sum_{j=1}^Ne^{L_2{(j)}-L_1{(j-1)}}\rb^{-(u+v)}\PP_{\lgrw}^{\alpha+v, \alpha-v}(\bL).
\label{eq:Pstatloggammabis}
\end{equation}
\end{lemma}
\begin{proof}
By comparing \eqref{eq:Pstatloggamma}   with \eqref{eq:Pstatloggammabis}, one only need to prove that
$$e^{-v(L_1(N)-L_2(N))}\PP_{\lgrw}^{\alpha,\alpha}\lb\bL\rb=\frac{\Gamma(\alpha-v)^N\Gamma(\alpha+v)^N}{\Gamma(\alpha)^{2N}}\PP_{\lgrw}^{\alpha+v,\alpha-v}\lb\bL\rb,$$
which follows immediately from the definition \eqref{eq:weight LG edges} of the log-gamma random walk densities.
\end{proof} 

\begin{remark} 
When $\alpha-v>0$ and $\alpha-u>0$, we may write: 
\begin{equation*} \PP_{\stat\lgg}^{\alpha, u,v}(\bL)  = \frac{\Gamma(\alpha-v)^N\Gamma(\alpha-u)^N}{\mathcal Z_{\lgg}^{\alpha, u,v}\Gamma(\alpha)^{2N}} \lb\sum_{j=1}^Ne^{L_2{(j)}+ L_1{(N)}-L_1{(j-1)}}\rb^{-(u+v)}\PP_{\lgrw}^{\alpha-u, \alpha-v}(\bL).
\label{eq:Pstatloggammater}
\end{equation*} 
which implies that 
\begin{equation}\label{eq:normalizationloggamma} 
\mathcal Z_{\lgg}^{\alpha, u,v} = \frac{\Gamma(\alpha-v)^N\Gamma(\alpha-u)^N}{\Gamma(\alpha)^{2N} }\mathbb E_{\lgrw}^{\alpha-u, \alpha-v}\left[ \left( \sum_{j=1}^N \prod_{i=1}^j w_{2,i}\prod_{i=j}^N w_{1,i}\right)^{-u-v}\right],
\end{equation}
where the $w_{1,i}$ are independent $\Ga^{-1}(\alpha-u)$ random variables and the  $w_{2,i}$ are independent $\Ga^{-1}(\alpha-v)$ random variables. Thus, the sum in \eqref{eq:normalizationloggamma} can be seen as the partition function of a (full-space) log-gamma polymer model on a subset of the $\mathbb Z^2$ lattice of size $N\times 2$.  
\end{remark}

\begin{proposition}[Stationary measure convergence to open KPZ]
For $i\in \lbrace 1,2\rbrace$ defined
\begin{equation} 
B_i^{(\nn)}(x) = -\nn^{-1} x\log(\nn)+ L_i{(\nn^{-1} x)}.
\label{eq:rescaledrandomwalks}
\end{equation}
and let $\alpha$ and $N$ be scaled as in \eqref{eq:scalings}. For any $u,v\in \R$, the law of $(B_1^{(\nn)}, B_2^{(\nn)})$, when $(\bL_1,\bL_2)$ is distributed as $\PP_{\stat\lgg}^{\alpha, u,v}(\bL_1,\bL_2)$ (Definition \ref{def:rescaled random walks}), converges as $\nn\to 0$ to the law of $(B_1,B_2)$ under $\PP_{\KPZstat}^{u,v}$ (Definition \ref{def:KPZstat}). 
\label{prop:convergencestationarymeasures}
\end{proposition}
\begin{proof}
    For a random variable  $w\sim  \Ga^{-1}(\alpha+v)$, if $\alpha=\frac{1}{2}+ \nn^{-1}$, we have the asymptotics 
\begin{equation}
\mathbb E\left[ \log(w) \right] = -\Psi(\alpha+v) = \log(\nn) - \nn v  + O\left(\nn^2\right), \;\; \mathrm{Var}\left[ \log(w) \right]= \Psi_1(\alpha+v) =  \nn - v \nn^2+ O\left(\nn^3\right),
    \label{ea:loggammafirstmoments}
\end{equation}
where $\Psi(z)=\partial_z \log \Gamma(z)$ and  $\Psi_1(z)=\partial_z\Psi(z)$ are the digamma and trigamma functions.
We will use the alternative expression of $\PP_{\stat\lgg}^{\alpha, u,v}$ given by Lemma \ref{lem:alternative reweighted random walk}. 
Recall the scalings \eqref{eq:scalings} and  \eqref{eq:rescaledrandomwalks}. Donsker's theorem implies that the law of $(B_1^{(\nn)},B_2^{(\nn)})$, when $(\bL_1,\bL_2)$ is distributed as  $\PP_{\lgrw}^{\alpha+v, \alpha-v}(\bL)$ (i.e., log-gamma random walks), converges as $\nn\to 0$ to the law of $(B_1,B_2)$ under $\PP_{\BM}^{-v,v}$ (i.e., two independent Brownian motions on $[0,L]$ with drifts $-v$ and $v$ respectively). This implies the weak convergence 
\begin{equation}
\left( \nn \sum_{j=1}^{\nn^{-1} L} e^{B_2^{(\nn)}(\nn j)-B_1^{(\nn)}(\nn(j-1))} \right)^{-u-v} \xRightarrow[\nn\to 0]{}
 \left( \int_0^L \d s e^{-(B_1(s)-B_2(s))}\right)^{-u-v}.
    \label{eq:weakconvergence}
\end{equation}
Recall that the normalization constant $\mathcal Z^{\alpha, u,v}$  in \eqref{eq:Pstatloggammabis} can be written as 
\begin{equation}
    \frac{\mathcal Z^{\alpha, u,v} \Gamma(\alpha+u)^N}{\Gamma(\alpha-v)^N}= \mathbb E_{\lgrw}^{\alpha+v, \alpha-v}\left[ \left( \nn \sum_{j=1}^{\nn^{-1} L} e^{B_2^{(\nn)}(\nn j)-B_1^{(\nn)}(\nn(j-1))} \right)^{-u-v}\right]. 
    \label{eq:normalizationdiscrete}
\end{equation}
We claim that the family of random variables in the left-hand-side of \eqref{eq:weakconvergence} is uniformly integrable, so that it converges in expectation, and we deduce that $\mathcal Z^{\alpha, u,v}$ converges  as $\nn\to 0$ to 
$$  \mathcal Z_{\KPZstat}^{u,v}= \mathbb E_{\BM}^{-v,v}\left[  \left( \int_0^L \d s e^{-(B_1(s)-B_2(s))}\right)^{-u-v} \right],$$
where $\mathbb E_{\BM}^{-v,v}$ is the expectation associated to the probability measure $\PP_{\BM}^{-v,v}$ from Definition \ref{def:KPZstat}. 
More generally, we obtain that for any bounded and continuous function $F: C([0,L], \mathbb R)^2\to \mathbb R$, 
\begin{multline}\mathbb E_{\rm LGRW}^{\alpha+v, \alpha-v}\left[ F\left(B_1^{(\nn)}, B_2^{(\nn)}\right) \left( \nn \sum_{j=1}^{\nn^{-1} L} e^{B_2^{(\nn)}(\nn j)-B_1^{(\nn)}(\nn (j-1))} \right)^{-u-v}\right] \\ 
\xrightarrow[\nn\to 0]{}\mathbb E_{\BM}^{-v,v}\left[ F\left(B_1, B_2 \right)  \left( \int_0^L \d s e^{-(B_1(s)-B_2(s))}\right)^{-u-v} \right].
\label{eq:convergenceinitialcondition}\end{multline} 
Combining \eqref{eq:normalizationdiscrete} and \eqref{eq:convergenceinitialcondition} implies that  $\PP_{\stat\lgg}^{\alpha, u,v}$ weakly converges to $\PP_{\KPZstat}^{u,v}$. 

Therefore, to conclude the proof, it remains to justify the uniform integrability claimed above. When $u+v<0$, it suffices to show that for some $k> \max\lbrace - (u+v),1 \rbrace$, the expectation 
$$C_{\alpha, u,v}:= \mathbb E_{\lgrw}^{\alpha+v, \alpha-v}\left[ \left( \sum_{j=1}^{\nn^{-1} L}  e^{L_2{(j)}-L_1{(j-1)}} \right)^k\right]$$
is uniformly bounded as $\nn\to 0$. By the convexity of $x\mapsto x^k$ for $k>1$, Jensen's inequality implies that 
\begin{align*}
    C_{\alpha, u,v} \leq \mathbb E_{\rm LGRW}^{\alpha+v, \alpha-v} \left[\left(\tfrac{L}{\nn}\right)^{k-1}  \sum_{j=1}^{\nn^{-1} L} \left(e^{L_2{(j)}-L_1{(j-1)}}\right)^k \right]= \left(\tfrac{L}{\nn}\right)^{k-1} \sum_{j=1}^{\nn^{-1} L} \left( \frac{\Gamma(\alpha+v+k)}{\Gamma(\alpha+v)} \right)^{j-1}\!\!\!\left( \frac{\Gamma(\alpha-v-k)}{\Gamma(\alpha-v)} \right)^{j}.
\end{align*}
The equality above used the expression in terms of gamma function ratios for moments of gamma and inverse gamma random variables, along with the fact that $e^{L_2(j)}$ is a product of $j$ independent inverse Gamma variables with parameter $\alpha-v$, while $e^{-L_1(j-1)}$ is a product of $j-1$ independent Gamma variables with parameter $\alpha+v$.
It is now straight-forward (writing ratios of gamma functions as rational functions) to verify that each summand is bounded above by $C'_{L,v,k}\nn^k$ for some constant $C'_{L,v,k}>0$, so that $C_{\alpha, u,v}$ is uniformly bounded (and tending to zero) as $\nn\to 0$. 
When $u+v>0$, we may use the convexity of the function $x\mapsto x^{-k}$ and the argument is very similar. This concludes the proof. 
\end{proof}
\subsection{Proof of Theorem \ref{thm:open KPZ stationary measure} modulo Conjecture \ref{conj:convergence} }
\label{subsec:Proof of open KPZ stationary measure}
Let  $z(n,m)$ be the log-gamma polymer partition function defined in \eqref{eq:defloggammapartitionfunction} on a strip of size $N=\nn^{-1}L$, with initial condition  $z_0=e^{h_0^{(\nn)}}$ on the horizontal path $\mathcal P_h$. Assume that $h_0^{(\nn)}$ follows the same law as $\mathbf L_1$ under the measure $\PP_{\stat\lgg}^{\alpha, u,v}$  from Definition \ref{def:rescaled random walks}, where we recall $\alpha=\frac{1}{2}+\nn^{-1}$. 
Along the path $\mathcal P_h$ shifted by $(\nn^{-2} t/2, \nn^{-2} t/2)$, we know from Theorem \ref{thm:stationary measure log-gamma intro} that the law of $ \left(z(\nn^{-2} t/2+i, \nn^{-2} t/2)/z(\nn^{-2} t/2, \nn^{-2} t/2)\right)_{0\leq i\leq N}$ is the same as $\left(z_0(i)\right)_{0\leq i\leq N}$. Using the weak convergences from Proposition \ref{prop:convergencestationarymeasures} and Conjecure \ref{conj:convergence}, this implies that if $h_0$ is distributed as $B_1$ under $\PP_{\KPZstat}^{u,v}$, then  for any $t>0$, the law of $h(t,x)-h(t,0)$ is also that of $B_1$ under $\PP_{\KPZstat}^{u,v}$. This concludes the proof.

\bibliographystyle{goodbibtexstyle.bst}
\bibliography{stationaryGibbs.bib}

\end{document}